\PassOptionsToPackage{pdftex,unicode,naturalnames,
pdfencoding=auto,
pdfnewwindow=true,
bookmarks=true,
bookmarksnumbered=true,
bookmarksopen=false,
pdfpagemode=UseThumbs,
bookmarksopenlevel=1,
pdfpagelabels=true,
breaklinks=true,
colorlinks=true,
}{hyperref}
\PassOptionsToPackage{dvipsnames}{xcolor}
\documentclass{article}
\usepackage[utf8]{inputenx}

\usepackage{graphicx}

\usepackage{geometry}

\usepackage{nccmath, amsthm, amsfonts, bbm,amssymb}
\usepackage[round]{natbib}
\usepackage{tikz}
\usetikzlibrary{arrows, shapes}
\usetikzlibrary{positioning}
\usetikzlibrary{arrows.meta}
\usepackage{authblk}
\usepackage[normalem]{ulem}
\usepackage{natbib}
\usepackage{hyperref}
\usepackage[capitalise]{cleveref}
\usepackage{bm}
\usepackage{xparse}
\usepackage{enumitem}
\usepackage{mathtools}
\usepackage[disable]{todonotes}
\usepackage{subcaption}
\usepackage{thm-restate}
\usepackage{stmaryrd}
\usepackage{float}
\usepackage[dvipsnames]{xcolor}
\definecolor{purple}{RGB}{128,0,128}
\definecolor{ultramarine}{RGB}{63, 0, 255}
\definecolor{medblue}{RGB}{0, 0, 100}
\definecolor{googleblue}{RGB}{34, 0, 204}
\definecolor{panblue}{RGB}{0,24,150}
\definecolor{carmine}{RGB}{150, 0, 24}
\definecolor{gray}{RGB}{150, 150, 150}
\definecolor{darkgreen}{RGB}{0, 80, 0}
\hypersetup{linkcolor=carmine,citecolor=darkgreen,urlcolor=googleblue,anchorcolor=OliveGreen}
\usepackage{wrapfig}
\usepackage{stmaryrd}
\usepackage{enumitem}
\usepackage{pifont}


\renewcommand{\langle}{(}
\renewcommand{\rangle}{)}

\newcommand{\DropMtoS}{\ensuremath{\mathtt{SplitM\texttt{\em ->}\,S}}}
\newcommand{\ToBidir}{\mathtt{ToSpecial}}

\newcommand{\zeros}{\kern-1.2pt\vec{\kern1.2pt 0}}
\newcommand{\Unif}{\text{Unif}}

\newcommand{\ryan}[1]{\todo[color=orange!30, inline]{Ryan: #1}}

\newcommand{\blk}{\color{black}}

\setlist{itemsep=1pt,topsep=2pt,parsep=1pt,partopsep=1pt}

\usetikzlibrary{arrows}

\definecolor{mypurple}{rgb}{0.54, 0.01, 0.98}
\tikzset{nv/.style={circle, color=red, fill=red, inner sep=0.5mm}}
\tikzset{rv/.style={circle, draw, thick, minimum size=5.5mm, inner sep=0.5mm}}
\tikzset{fv/.style={rectangle, draw, thick, minimum size=5mm, inner sep=0.5mm}}
\tikzset{lv/.style={circle, color=red, fill=gray!20, draw, thick, minimum size=5.5mm, inner sep=0.5mm}}
\tikzset{mv/.style={lv}}
\tikzset{sv/.style={circle, color=blue, fill=gray!20, draw, thick, minimum size=5.5mm, inner sep=0.5mm}}
\tikzset{slv/.style={circle, draw, fill=gray!20, draw=red, double circle={0.9mm}{black}, minimum size=5.5mm, inner sep=0.5mm}}
\tikzset{smv/.style={slv}}
\tikzset{mdot/.style={circle, minimum size=2mm, inner sep=0mm, fill=red}}
\tikzset{sdot/.style={circle, minimum size=2mm, inner sep=0mm, fill=blue}}
\tikzset{smdot/.style={circle, minimum size=2mm, inner sep=0mm, fill=mypurple}}
\tikzset{rve/.style={ellipse, draw, thick, minimum size=6.5mm, inner sep=0.5mm}}

\tikzset{deg/.style={->, very thick, color=blue}}
\tikzset{degl/.style={->, very thick, color=red}}
\tikzset{degs/.style={->, very thick, color=black}}
\tikzset{ueg/.style={very thick}}

\usetikzlibrary{calc,swigs}
\tikzset{swig hsplit={gap=5pt}}
\tikzset{swig vsplit={gap=5pt}}
\tikzset{vspl/.style={shape=swig vsplit, thick}}
\tikzset{vspll/.style={shape=swig vsplit={gap=10pt, line color right=red!1}, thick}}
\tikzset{vsplr/.style={shape=swig vsplit={line color left=white}, thick}}
\tikzset{hspl/.style={shape=swig hsplit, thick}}

\tikzset{
    old inner xsep/.estore in=\oldinnerxsep,
    old inner ysep/.estore in=\oldinnerysep,
    double circle/.style 2 args={
        circle,
        old inner xsep=\pgfkeysvalueof{/pgf/inner xsep},
        old inner ysep=\pgfkeysvalueof{/pgf/inner ysep},
        /pgf/inner xsep=\oldinnerxsep+#1,
        /pgf/inner ysep=\oldinnerysep+#1,
        alias=sourcenode,
        append after command={
        let     \p1 = (sourcenode.center),
                \p2 = (sourcenode.east),
                \n1 = {\x2-\x1-#1-0.5*\pgflinewidth}
        in
            node [inner sep=0pt, draw, very thick, circle, minimum width=2*\n1,at=(\p1),#2] {}
        }
    },
    double circle/.default={2pt}{blue}
}

\newtheorem{lemma}{Lemma}[section]
\newtheorem{theorem}[lemma]{Theorem}
\newtheorem{proposition}{Proposition}

\newtheorem{definition}[lemma]{Definition}

\newcommand{\G}{\mathcal{G}}
\newcommand{\D}{\mathcal{D}}
\newcommand{\E}{\mathcal{E}}
\newcommand{\calE}{\mathcal{E}}
\newcommand{\B}{\mathcal{B}}
\newcommand{\calL}{\mathcal{L}}
\newcommand{\calS}{\mathcal{S}}

\newcommand{\RR}{R}
\newcommand{\RRs}{{\RR^\sharp}}
\renewcommand{\tilde}{\widetilde}

\newcommand{\can}{\mathtt{can}}
\newcommand{\DAG}{\text{DAG}}
\newcommand{\cmid}{\,|\,}

\newcommand{\ang}[1]{\langle #1 \rangle}

\newcommand{\doo}{{\mathtt{do}}}

\newcommand{\sx}{{\bm{x}}}
\newcommand{\sm}{{\bm{m}}}

\NewDocumentCommand\an{m+g}{\IfNoValueTF{#2}{\operatorname{an}(#1)}{\operatorname{an}_{#1}({#2})}} 
\NewDocumentCommand\pa{m+g}{\IfNoValueTF{#2}{\operatorname{pa}(#1)}{\operatorname{pa}_{#1}({#2})}}
\NewDocumentCommand\pas{m+g}{\IfNoValueTF{#2}{\operatorname{pa}^\sharp(#1)}{\operatorname{pa}^\sharp_{#1}({#2})}}
\NewDocumentCommand\ch{m+g}{\IfNoValueTF{#2}{\operatorname{ch}(#1)}{\operatorname{ch}_{#1}({#2})}}
\NewDocumentCommand\de{m+g}{\IfNoValueTF{#2}{\operatorname{de}(#1)}{\operatorname{de}_{#1}({#2})}}
\NewDocumentCommand\inc{m+g}{\IfNoValueTF{#2}{\operatorname{in}(#1)}{\operatorname{in}_{#1}({#2})}}
\NewDocumentCommand\out{m+g}{\IfNoValueTF{#2}{\operatorname{out}(#1)}{\operatorname{out}_{#1}({#2})}}
\NewDocumentCommand\io{m+g}{\IfNoValueTF{#2}{\operatorname{io}(#1)}{\operatorname{io}_{#1}({#2})}}
\NewDocumentCommand\nb{m+g}{\IfNoValueTF{#2}{\operatorname{nb}(#1)}{\operatorname{nb}_{#1}({#2})}}
\NewDocumentCommand\sib{m+g}{\IfNoValueTF{#2}{\operatorname{sb}(#1)}{\operatorname{sb}_{#1}({#2})}}

\DeclareMathOperator{\paN}{pa}

\newcommand{\cl}{{\operatorname{cl}}}

\newcommand{\calC}{\mathcal{C}}
\newcommand{\calD}{\mathcal{D}}
\newcommand{\calG}{\mathcal{G}}
\newcommand{\calM}{\mathcal{M}}
\newcommand{\calP}{\mathcal{P}}
\newcommand{\dom}[1]{\mathfrak{X}_{#1}}

\newcommand{\exogenize}{\mathfrak{r}}

\newcommand{\slp}{\mathfrak{p}}
\newcommand{\dotcup}{\dot{\cup}}

\newcommand{\pathto}{\mathbin{\dashrightarrow}}

\newcommand{\upathto}{\;\hbox{-\hspace{-0.17em} -\hspace{-0.17em} -}\;}
\newcommand{\bidir}{\mathbin{\leftrightarrow}}

\newcommand{\undir}{\mathbin{-\!\!\!-}}

\title{Distinguishability of causal structures under latent confounding and selection}

\author[1]{Ryan Carey\thanks{Corresponding author. E-mail: \href{mailto:ryan.carey@trajectorylabs.net}{ryan.carey@trajectorylabs.net}}}
\author[2,3]{Marina Maciel Ansanelli}
\author[2,3]{Elie Wolfe}
\author[1]{Robin J. Evans}

\affil[1]{University of Oxford, Department of Statistics, 24-29 St Giles', Oxford OX1 3LB, UK}
\affil[2]{Perimeter Institute for Theoretical Physics, 31 Caroline Street North, Waterloo, Ontario, Canada N2L 2Y5}
\affil[3]{Department of Physics and Astronomy, University of Waterloo, Waterloo, Ontario, Canada, N2L 3G1}

\date{ }

\usepackage{etoolbox}
\makeatletter
\patchcmd{\l@section}{1.0em}{0.8em}{}{\PackageError{main}{ToC spacing patch failed}{}}
\makeatother

\begin{document}
\maketitle

\section*{Abstract}

Statistical relationships in observed data can arise 
for several different reasons:
the observed variables may be causally related, 
they may share a latent common cause, or there may be selection bias.
Each of these scenarios can be modelled using 
different causal graphs.
Not all such causal graphs, however, can be distinguished 
by experimental data. In this paper, we formulate the equivalence class of causal graphs as a novel graphical structure, the selected-marginalized directed graph (smDG). That is, we show that two directed acyclic graphs with latent and selected vertices have the same smDG if and only if they are indistinguishable, even when allowing for arbitrary interventions on the observed variables. As a substitute for the more familiar $d$-separation criterion for DAGs, we provide an analogous sound and complete separation criterion in smDGs for conditional independence relative to passive observations.
Finally, we provide a series of sufficient conditions under which two causal structures are indistinguishable when there is only access to passive observations.

\noindent\textbf{Keywords:} selected-marginalized directed graph, cyclic causal models, Markov equivalence, hierarchical model, selection effects, latent-confounding

\noindent\textbf{MSC 2020:} 62D20, 62H22, 62D05

\newpage
\tableofcontents
\section{Introduction} \label{sec:intro}

A statistical correlation between two variables $X_a$ and $X_b$ can be explained by three different types of causal mechanisms:

\begin{enumerate}[label=\alph*)]
  \item \textbf{Direct causal influence}, where $X_a$ exerts an influence on $X_b$ or vice-versa.
  \item \textbf{Latent confounding}, where a hidden common cause $X_m$ induces correlation between $X_a$ and $X_b$.
  \item \textbf{Selection effects}, where conditioning on a variable $X_s$ that is influenced by both $X_a$ and $X_b$ creates correlation (e.g.\ when we restrict the study sample).
\end{enumerate}

For example, if we observe a correlation between obesity and mortality in a sample of hospital patients, this could be a) because obesity worsens heart disease, and/or b) because there is a health condition that causes weight loss, while also shortening lifespan, and/or c) because we are not taking data from the general population, but only from the hospital population. While obesity is positively correlated with mortality in the general population, among the hospital population they could be negatively correlated, because when a person in the hospital is underweight this can be due to serious disease.
 
In this work, we want to analyze alternative causal structures, possibly involving hidden common causes and selection effects, that could underlie a given empirical dataset. A causal structure can be represented using a directed acyclic graph (DAG); an example where variables $b$ and $c$ have a shared hidden common cause is shown in \Cref{fig:confounding-dag2}.
Through this paper, visible (observable) vertices will be represented in black with a white background, and marginalized (unobservable) vertices will be represented in red with a grey background.

\begin{figure}[ht]
    \centering
     \begin{subfigure}{.20\textwidth}
        \centering
        \begin{tikzpicture}
            [rv/.style={circle, draw, very thick, minimum size=5.5mm, inner sep=0.5mm}, 
             node distance=20mm, >=stealth]

            \node[mv] (a) {$a$};
            \node[rv, below =4mm of a, xshift=-6mm] (b) {$b$};
            \node[rv, below =4mm of a, xshift=6mm] (c) {$c$};

            \draw[->, very thick] (a) -- (b);
            \draw[->, very thick] (a) -- (c);

        \end{tikzpicture}
        \caption{A DAG.}
        \label{fig:confounding-dag2}
    \end{subfigure}\hspace{2mm}%
    \begin{subfigure}{.22\textwidth}
        \centering
        \begin{tikzpicture}
            [rv/.style={circle, draw, very thick, minimum size=5.5mm, inner sep=0.5mm}, 
             node distance=20mm, >=stealth]

            \node[mv] (a) {$a$};
            \node[mv, right =6mm of a] (ap) {$a'$};
            \node[rv, below =4mm of a, xshift=-6mm] (b) {$b$};
            \node[rv, below =4mm of ap, xshift=6mm] (c) {$c$};

            \draw[->, very thick] (a) -- (ap);
            \draw[->, very thick] (a) -- (b);
            \draw[->, very thick] (ap) -- (c);

        \end{tikzpicture}
        \caption{Another DAG.}
        \label{fig:confounding-dag1}
    \end{subfigure}\hspace{2mm}%
    \begin{subfigure}{.36\textwidth}
        \centering
        \begin{tikzpicture}
            [rv/.style={circle, draw, very thick, minimum size=5.5mm, inner sep=0.5mm}, 
             node distance=20mm, >=stealth]

            \node[rv] (b) {$b$};
            \node[rv, right =6mm of b] (c) {$c$};
            \node[below=2mm of c, inner sep=0, minimum size=0] (invisibleNode) {};

            \draw[<->, very thick, color=red] (b) -- (c);

        \end{tikzpicture}
        \caption{Marginalized DAG (mDAG) associated with both (a) and (b).}
        \label{fig:confounding-mdag}
    \end{subfigure}%
    \caption{Two DAGs that share the same marginalized DAG.}
    \label{fig:confounding}
\end{figure}

Sometimes, no experimental data on the visible variables can distinguish between two distinct DAGs.
For example, consider the DAGs shown in \cref{fig:confounding-dag1,fig:confounding-dag2}. Both causal structures can explain any correlation between $b$ and $c$ in the observational data, and in both the correlation is broken by any intervention on $b$ or $c$. 

To capture information about indistinguishability of causal structures in the case without selection effects, \citet{evans2016graphs} introduced a data structure called \emph{marginalized DAG (mDAG)}, and a procedure that maps DAGs to mDAGs, called \emph{latent projection}. For example, \cref{fig:confounding-mdag} shows the mDAG that is obtained by latent projection of both \ref{fig:confounding-dag1} and \ref{fig:confounding-dag2}.
\citet{evans2016graphs} proved that
two different DAGs that have the same latent projection can realize the exact same sets of probability distributions obtained by passive observation of the visible variables, i.e.\ they are observationally equivalent. Later, \citet{ansanelli2024everything} proved that there exists \emph{some} interventional experiment that can distinguish two DAGs \emph{if and only if} the DAGs have different latent projections. That is, there exists at least one set of data obtained from a very powerful\footnote{As argued in \citet{ansanelli2024everything}, the Observe\&Do probing scheme is \emph{informationally complete}: it provides all of the information about the causal mechanisms that is accessible by interaction with the visible variables (without considering edge interventions~\citep{shpitser2016causal}).} interventional probing scheme called Observe\&Do (that will be defined later on) that is realizable by one DAG but not by the other if and only if these DAGs correspond to different mDAGs.
In this sense, equality of the latent projection is a sound and complete criterion for the interventional indistinguishability of DAGs containing only visible and marginalized vertices.

\begin{wrapfigure}{r}{0.5\columnwidth}
    \centering
    \begin{subfigure}{.24\textwidth}
        \centering
        \begin{tikzpicture}
            [rv/.style={circle, draw, very thick, minimum size=5.5mm, inner sep=0.5mm}, 
             node distance=20mm, >=stealth]
            \node[rv] (a) {$a$};
            \node[sv, below =4mm of a] (s) {$s$};
            \node[rv, below =10mm of a, xshift=-12mm] (b) {$b$};
            \node[rv, below =10mm of a, xshift=12mm] (c) {$c$};

            \draw[->, very thick] (a) -- (s);
            \draw[->, very thick] (b) -- (s);
            \draw[->, very thick] (c) -- (s);
        \end{tikzpicture}
        \caption{A DAG}
        \label{fig:selection-dag1}
    \end{subfigure}%
    \begin{subfigure}{.24\textwidth}
        \centering
        \begin{tikzpicture}
            [rv/.style={circle, draw, very thick, minimum size=5.5mm, inner sep=0.5mm}, 
             node distance=20mm, >=stealth]
            \node[rv] (a) {$a$};
            \node[sv, below =1mm of a, xshift=-10mm] (s1) {$s_1$};
            \node[sv, below =1mm of a, xshift=10mm] (s2) {$s_2$};
            \node[sv, below =14mm of a] (s3) {$s_3$};
            \node[rv, below =10mm of a, xshift=-12mm] (b) {$b$};
            \node[rv, below =10mm of a, xshift=12mm] (c) {$c$};

            \draw[->, very thick] (a) -- (s1);
            \draw[->, very thick] (b) -- (s1);
            \draw[->, very thick] (a) -- (s2);
            \draw[->, very thick] (c) -- (s2);
            \draw[->, very thick] (b) -- (s3);
            \draw[->, very thick] (c) -- (s3);
        \end{tikzpicture}
        \caption{Another DAG}
        \label{fig:selection-dag2}
    \end{subfigure}%
    \caption{Two selected DAGs that are interventionally distinguishable but correspond to the same MAG.
    }
    \label{fig:selection}
\end{wrapfigure}
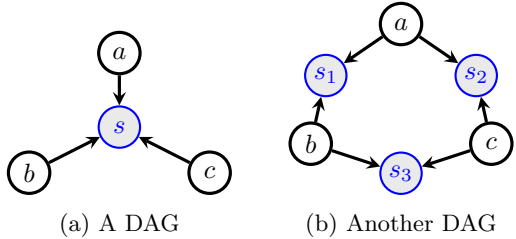

Causal structures can also represent selection effects using \emph{selected vertices},
such as in the DAG of \Cref{fig:selection}; here, selected vertices are represented in blue with a grey background.

There exist proposed methods for projecting such graphs down to the visible 
variables, such as maximal ancestral graphs (MAGs) \citep{richardson2002ancestral},
but these projections lose information about the set of data that is realizable by the DAG\blk. For example, the two DAGs of \Cref{fig:selection} have the same MAG (which is saturated, i.e.~it allows all joint distributions over the three visible variables), despite the fact that these DAGs are observationally distinguishable.
\Cref{fig:selection-dag1} is compatible with the selected distribution $P(X_a=1,X_b=0,X_c=0 \mid X_s=0)=P(X_a=0,X_b=1,X_c=0 \mid X_s=0)=P(X_a=0,X_b=0,X_c=1 \mid X_s=0)=\frac{1}{3}$, 
while \Cref{fig:selection-dag2} is not. In fact, two DAGs that correspond to the same MAG are \emph{only} guaranteed to have the same set of conditional independence constraints. The observational distributions that are realizable by them can be constrained by different sets of  Verma constraints \citep{richardson2023nested}, inequalities \citep{wolfe2019inflation} and hierarchical 
constraints induced by selection \citep{lauritzen1998generating}. Moreover, even in cases where the observational distributions realizable by them are the same, they can differ in the sets of data that they realize via interventional probing schemes, such as Observe\&Do.
\blk

Thus the current literature
lacks an analogue of the mDAG that faithfully represents what can be learned about the causal structure from experimental data in the presence of \emph{both} marginalization and selection.

This paper fills this gap by introducing the \emph{selected-latent projection}, a map that takes a DAG with visible, marginalized and selected vertices and maps it to a \emph{selected-marginalized directed graph (smDG)}, a new data structure that generalizes the mDAG.  Analogously to what was done in \citet{ansanelli2024everything} for mDAGs, here we establish the two following results for smDGs:
\begin{itemize}
  \item \textbf{Soundness:} If two DAGs project to the same smDG, they are interventionally equivalent (i.e.~indistinguishable even under access to the Observe\&Do probing scheme, which is very strong).
  \item \textbf{Completeness:} If their smDGs differ, there exists an
        intervention that distinguishes the original DAGs. (Thus the smDG is as 
        granular as possible.)
\end{itemize}

After that, as a secondary goal, we will explore alternative probing schemes for distinguishing different smDGs, such as passive observation or more limited interventional experiments. In particular, we will provide a series of rules to show that two smDGs are observationally equivalent, that is, indistinguishable by passive observations (For example, see the pair of graphs in \cref{fig:teaser-dags}).

\begin{figure}[ht]
    \centering
    \begin{subfigure}{.5\textwidth}
        \centering
        \begin{tikzpicture}
            [rv/.style={circle, draw, very thick, minimum size=5.5mm, inner sep=0.5mm}, 
             node distance=20mm, >=stealth]

            \node[rv] (a) {$a$};
            \node[rv, right =6mm of a, yshift=-6mm] (b) {$b$};
            \node[rv, right =6mm of b, yshift=4mm] (c) {$c$};
            \node[rv, right =6mm of c] (d) {$d$};

            \node[mv, below=4mm of a, xshift=-6mm] (ma) {$m_1$};
            \node[sv, above=4mm of a, xshift=-6mm] (s1) {$s_1$};

            \node[sv, above =6mm of b] (s2) {$s_2$};
            \node[mv, below =6mm of a, xshift=4mm] (mb) {$m_2$};
            \node[mv, below =5mm of c] (mc) {$m_3$};
            \node[mv, below =5mm of d] (md) {$m_4$};
            \node[sv, above =4mm of c, xshift=6mm] (s3) {$s_3$};

            \draw[->, very thick] (b) -- (a);
            \draw[->, very thick] (a) -- (s1);

            \draw[->, very thick] (a) -- (s2);
            \draw[->, very thick] (b) -- (s2);
            \draw[->, very thick] (c) -- (s2);

            \draw[->, very thick] (ma) -- (s1);
            \draw[->, very thick] (ma) -- (a);
            \draw[->, very thick] (mb) -- (a);
            \draw[->, very thick] (mb) -- (b);
            \draw[->, very thick] (mc) -- (c);
            \draw[->, very thick] (md) -- (d);
            \draw[->, very thick] (c) -- (s3);
            \draw[->, very thick] (d) -- (s3);

        \end{tikzpicture}
        \caption{A DAG}
        \label{fig:teaser-dag-1}
    \end{subfigure}%
        \begin{subfigure}{.5\textwidth}
        \centering
        \begin{tikzpicture}
            [rv/.style={circle, draw, very thick, minimum size=5.5mm, inner sep=0.5mm}, 
             node distance=20mm, >=stealth]

            \node[rv] (a) {$a$};
            \node[rv, right =6mm of a, yshift=-6mm] (b) {$b$};
            \node[rv, right =6mm of b, yshift=6mm] (c) {$c$};
            \node[rv, right =6mm of c] (d) {$d$};

            \node[mv, below=4mm of a] (m1) {$m_1$};

            \node[mv, below =5mm of c] (m2) {$m_2$};
            \node[mv, below =5mm of d] (m3) {$m_3$};
            \node[sv, above =4mm of c, xshift=5mm] (s1) {$s_1$};

            \draw[->, very thick] (b) -- (a);

            \draw[->, very thick] (c) -- (a);

            \draw[->, very thick] (m1) -- (a);
            \draw[->, very thick] (m2) -- (b);
            \draw[->, very thick] (m2) -- (c);
            \draw[->, very thick] (m3) -- (d);
            \draw[->, very thick] (c) -- (s1);
            \draw[->, very thick] (d) -- (s1);
        \end{tikzpicture}
            \begin{minipage}{.1cm}
            \vfill
            \end{minipage}
    \caption{An alternative DAG}
        \label{fig:teaser-dag-2}
    \end{subfigure}
    \caption{
    Can these DAGs produce different data on 
    $a,b,c,d$ after marginalization of the $m_i$ vertices 
    and selection of the $s_i$ vertices? Spoiler: as we will see, they can be distinguished by interventions but not by passive observations.}
    \label{fig:teaser-dags}
\end{figure}
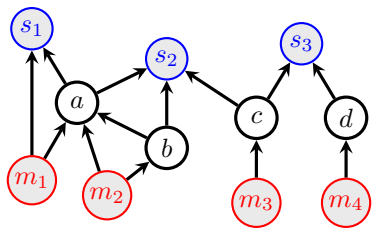
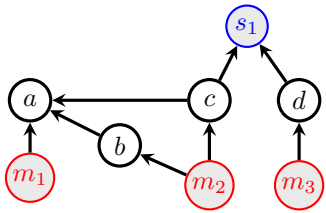

The paper is organized as follows.
In \Cref{sec:related-work}, we will review related literature.
In \Cref{sec:setup}, we will recap DAG models and 
formalize interventional and observational indistinguishability in the presence of selected vertices.
In \Cref{sec:invariances},
we will describe some simplifying operations that can be performed on a 
DAG that is partitioned into visible, {marginalized}, and selected variables, while keeping it interventionally indistinguishable from the original DAG.
We will see that these operations lead us to a canonical DAG, which is invariant to the simplifying operations.
In \Cref{sec:smDGs}, we will present the \emph{selected-marginalized directed graph} (smDG), 
a data structure that captures all information about the canonical DAG.
Then, in \Cref{sec_smDG_captures_everything}, we show that whenever a pair of DAGs have different smDGs,
they can be distinguished by the Observe\&Do interventional probing scheme.
In \Cref{sec:observational-equivalence}, we outline some methods 
for identifying pairs of smDGs that cannot be distinguished by passive observation of the visible variables.
In \Cref{sec:discuss-and-conclude} we will draw final conclusions and outline open questions.

\section{Related work} \label{sec:related-work}
We will now elaborate upon the state of current work on the marginal and selected models of graphs, and their graphical representation.

\sloppypar
For DAGs with visible and marginalized variables, one form of latent projection may be obtained 
with an edge $v_1 \to v_2$
whenever there exists a directed path $v_1 \to m_1\to\cdots \to m_n \to v_2$ through marginalized variables,
and a bidirected edge $v_1 \bidir v_2$
whenever there exists a path ${v_1 \gets m_1 \gets \cdots \gets m_c \to \cdots \to m_n \to v_2}$ with ${c \in \{1,\ldots, n\}}$ through marginalized variables \citep{verma1991equivalence,verma1993graphical}.
This results in a mixed graph that characterizes all of the independence relations of the original DAG.
The problem, however, is that information about many \emph{inequality} constraints is lost.
For example, when three variables have one joint marginalized parent, all distributions are possible; the same cannot be said when three variables have three separate marginalized parents, one for each pair of variables.
As such, the equivalence class is coarsened by representing both of these situations with the same mixed graph.
In full generality, it is not possible to preserve all such information with a graph 
in which each edge only connects two vertices \citep[Corollary 1]{evans2016graphs}.\footnote{The arXiv version has a distinct numbering system to the published version.  Here we use the numbers from the journal.}
As such, Evans has proposed a marginal DAG (mDAG) structure based on a lossless latent projection operation \citep{evans2016graphs}. mDAGs include 
hyperedges that enable them to preserve all of the constraints from the original DAG.

The lossless latent projection operation of \citet{evans2016graphs} can be understood as involving two parts.
First, canonicalization, 
where an arbitrary DAG with { marginalized} and visible variables is transformed into a canonical DAG
that is as simple as possible, while preserving the constraints from the original graph.
For example, the canonical DAG for \Cref{fig:confounding-dag1} is \Cref{fig:confounding-dag2}.
Second, projection, 
where the canonical DAG is replaced with an mDAG; for example, the mDAG for \Cref{fig:confounding-dag2} is \Cref{fig:confounding-mdag}.
Canonicalization is done in a series of substeps.
Whenever a marginalized variable has one or more incoming edges (i.e.~edges from parents), 
a step called \emph{exogenization} allows those edges to be replaced 
with a collection of direct edges from each parent to every child \citep{evans2016graphs}.
For example, in \Cref{fig:confounding-dag1}, the edge $a \to a'$ could be replaced with the edge $a \to c$.
A further step allows a marginalized variable to be removed if it has only one child, which would then allow $a'$ to be 
removed, yielding the graph in \Cref{fig:confounding-dag2}.
A third step allows a marginalized variable to be removed if its children are a subset of those of another marginalized variable.
Once all these simplification steps are followed, a canonical DAG is obtained.
Importantly, in replacing a canonical DAG with an mDAG, we need to use hyperedges, rather than 
simple bidirected edges, in order to preserve all the information about the marginal confounders \citep{evans2016graphs}.
To be more precise, the formalism represents confounding using a simplicial complex,
since this structure implies that for every hyperedge, all its subsets are also present. 
Then, a \emph{marginal DAG}, or mDAG, is a representation that is defined as consisting of a DAG, which describes direct relationships between observed variables, 
and a simplicial complex, which identifies those sets of variables that are confounded by marginalized common causes.

A set of interventional distributions (obtained, for example, by Pearl's $\doo$-operation)
that are compatible with an mDAG is called its \emph{marginal causal model}.
Any two DAGs, containing only visible and marginalized variables, with the same mDAG have an identical marginal causal model. 
On the other hand, if a pair of mDAGs differ, it is always possible to devise an experiment that can tell them apart
by jointly considering the observed and intervened data \citep{ansanelli2024everything}.
The marginal model of an mDAG refers to the set of observational distributions that are Markov-compatible with the structure, 
i.e.~we do not consider the results of interventional experiments.
Another useful property of an mDAG's marginal model is that, if observed variables are discrete, it forms a curved exponential family, which should aid in performing 
Bayesian inference over the observational distribution given finite data and an mDAG structure.
Sometimes, a pair of distinct mDAGs may share the same marginal model, that is be observationally indistinguishable, 
because the constraints on the marginal model are a subset of those of the marginal causal model.
Nonetheless, the set of constraints to the observational distributions on a DAG may be enumerated by the inflation technique \citep{wolfe2019inflation}.

There are different ways in which selection effects might arise in available data.
Past work considers when some variable has been conditioned to a single assignment, 
and this is the notion of selection that we will work with throughout the rest of the paper \citep{armen2018towards}.
Another possibility is that a variable has been conditioned to a set of assignments, and this may be of interest for future work.
The independence constraints under selection have long been characterized \citep{lauritzen1998generating,richardson2002ancestral}
and may be computed in polynomial time \citep{ali2009markov},
but the non-independence constraints are generally much less well understood \citep{lauritzen1998generating, evans2015recovering, armen2018towards}. 
In the case of log-linear models, we know that
these non-independence constraints cannot generally be summarized using a simple graph, 
as the results of \citet{lauritzen1998generating} have shown that any hierarchical model can be generated.
Our work therefore offers a more complex graphical structure.
Our work differs from that of \citet{evans2015recovering} and \citet{armen2018towards} 
in that we focus on the problem of projecting a DAG down to its visible variables, 
and on establishing equivalence, rather than seeking to enumerate the (in)equality constraints of a selected model.

It is also worth highlighting \citet{chen2024modeling}, 
which discusses a question related to our own:
how one SCM may be modified by intervention and conditioning.
Although it is not their main focus, they define a 
graphical operation \citep[Definition 13]{chen2024modeling}
that describes the result of selection.
The resulting graph does not preserve all information about the interventional models, however.
For instance, given the DAG $\calG: v \to s \gets u$, 
their selection operation yields the mixed graph $v \bidir u$, 
which is the same as the object they obtain by marginalizing $m$
in $\calG': v \gets m \to u$.
The problem is that $\calG$ and $\calG'$ are compatible with different sets of interventional distributions:
in $\calG$, intervening on $X_v$ can change the distribution over $X_u$ under selection.
For example, if $X_u$ and $X_v$ are Bernoulli, and $X_s=0$ precisely when $X_u$ and $X_v$ are equal, 
then intervening $\doo(X_v=0)$ will yield $P(X_u=0 \mid \doo(X_v=0),X_s=0)=1$, 
whereas $P(X_u=0 \mid X_s=0)=\frac{1}{2}$.
In $\calG'$, on the other hand, intervening on $X_u$ will sever the edge from $m$ to $u$, 
and so the distribution of $X_u$ is unaffected.

\section{Setup} \label{sec:setup}
We will now review the key formalism needed for our paper: DAGs, their models, and definitions of equivalence. We will also introduce new Definitions (\ref{def_smi_model}, \ref{def_smo_model}, \ref{def_interv_equiv} and \ref{def_obs_equiv}) that will be useful for the case where there are selection effects.

\subsection{Directed acyclic graphs}

\begin{definition}[Directed Acyclic Graph]
    A directed graph $\D$ is a pair $(W,\E)$, where $W$ is a finite set of vertices and $\E$ is a collection of edges, 
    which are ordered pairs of distinct vertices.  $\D$ is said to be a  \emph{directed acyclic graph} (DAG) if there are no cycles, 
    that is, if there exist no sequences of edges of the form $w_1 \to \cdots \to w_k \to w_1$ where $k \geq 1$.
\end{definition}
If there is an edge $(v,w)$, illustrated as $v \to w$, then $v$ is said to be a \emph{parent} of $w$, and $w$ is said to be a \emph{child} of $v$.
We say that $v$ is a \emph{neighbour} of $w$ if it is either a parent or a child of $w$.
This edge is said to have an \emph{arrowhead} at $w$ and a \emph{tail} at $v$.
We will denote the set of parents of $w$ in $\calD$ by $\pa{\calD}{w}$, and the set of children of $v$ by $\ch{\calD}{v}$, omitting the subscript when the graph is clear from context.

A \emph{path} from $w_0$ to $w_k$ is an alternating sequence of distinct vertices and edges $\langle w_0,e_1,w_1,\ldots$ $e_k,w_k \rangle$, 
such that each edge $e_i$ is between the vertices $w_{i-1}$ and $w_i$, and its length $k$ is equal to the number of edges.
\ryan{Probably we can replace this with a ``sequence of edges'' definition now that \~ all of our analyses use DAGs rather than smDAGs}
If a path is of the form $w_0 \to \cdots \to w_k$, then we call it \emph{directed} and write it as $w_0 \pathto w_k$.
If there is a directed path from $v$ to $w$, of length at least $0$, then $v$ is an ancestor of $w$, and 
we denote the set of ancestors of $w$ as $\an{\calD}{w}$.
We apply definitions of parents, children, and ancestors to sets of vertices so that for a set of vertices $W$, we have
$\pa{\calD}{W} = \bigcup_{w \in W}\pa{\calD}{w}$, and so on.
In general, we will use $A \dotcup B$ to denote a union of vertices, when $A$ and $B$ are disjoint.

Given DAGs $\D = \ang{W,\E}$ and $\D' = \ang{W',\E'}$, if $W' \subseteq W$ and $\E' \subseteq \E$, we will say that $\D'$ is a subgraph of $\D$, 
written as $\D' \subseteq \D$.
The induced subgraph of $\D$ over $A \subseteq W$ is the DAG $\D_A$ with vertices $A$ and edges that have both endpoints inside $A$.

\subsection{Observational equivalence}
When a DAG $\calD=\ang{W,\E}$ represents a causal structure, each of its vertices is associated with a random variable. In this paper, the variable associated with the vertex $a\in W$ will be denoted by $X_a$. Similarly, the set of variables associated with a \emph{set} of vertices $A\subseteq W$ will be denoted by $X_A$. An assignment to $X_a$ will be denoted with the lower case $x_a$, and a set of assignments corresponding to the set of variables $X_A$ will be denoted by $x_A$. 
The state space or \emph{domain} of each variable $X_a$ will be denoted by $\mathfrak{X}_a$. 
In the case of continuous random variables,
we work with probability measures on the unit interval 
$\left[0,1\right]$;
by the Borel‐isomorphism theorem, every distribution on any other standard Borel space can be transported to (and recovered from) one on the unit interval.
We will then denote a density or mass function over $X_A$ with respect to a suitable measure, by $P(x_A)$.
Note that all of our results will be applicable to discrete and continuous random variables.

Data that is generated from a particular causal structure will respect certain constraints.
To begin with, let us consider constraints on observational data, represented by a joint distribution over variables.
This joint distribution (or density) should factorize so that the distribution for each variable can be expressed conditional on the assignments to its parents;
this is called Markov compatibility.

\begin{definition}[Markov Compatibility; Definition 1.2.2, \citealp{pearl2009causality}]
Let $\calD=\ang{W,\E}$ be a DAG. If a probability distribution (or density) $P$ over the variables $X_W$ admits the factorization 
$P(x_W)=\prod_{a\in W} P(x_a \mid x_{\pa{\calD}{a}})$, we say that $\calD$ and $P$ are compatible, or that $P$ is Markov relative to $\calD$. 
Each factor $P(x_a \mid x_{\pa{a}})$ is called a Markov factor, or a \emph{kernel}.
\end{definition}

We remark that in the case of \emph{deterministic} variables, it is possible that no conditional density exists.  We do not consider this possibility as it is not our focus, but the model above may be replaced with the \emph{structural equation model}; under this collection of distributions, each variable is a function of its parents in the graph \citep{evans2016graphs}.

Beyond Markov compatibility, we will place one further requirement: that visible variables 
are a deterministic function of their parents' assignments. 
In general, we will say that a kernel $P(x_v \mid x_{\pa{v}})$ of a variable is \emph{deterministic}
if it outputs $x_v$ as a deterministic function of its parents' values.
That is, a visible variable has an independent source of randomness precisely when it has
a marginalized parent.
This assumption corresponds to the notion of deterministic variables from \citet{geiger1990identifying}, 
and we use it because it makes our findings simpler and more general.

Our assumption that visible variables are deterministic (some of them having a marginalized parent) is actually strictly more general 
than the more-common setting where all variables are stochastic.
To see this, note that we can treat the standard setting 
as the special case where each visible variable has a marginalized parent.\footnote{

To translate a stochastic variable DAG into our notation, 
just add a marginalized parent to each visible vertex.
To translate back, simply remove  every marginalized vertex that is parentless and has a single (visible) child.}
\

We are now in a position to define the set of observational data realizable by a causal structure, called its \emph{selected-marginal observational model}. 
For this, we will rely on a partition $(V,M,S)$ of the vertices in the DAG into 
those associated with visible, marginalized, and selected variables respectively.
Throughout the paper, we will denote their union $W:=V\dotcup M \dotcup S$.
We will refer to the union of marginalized \emph{and} selected variables as \emph{non-visible variables}.
In the selected data, each selected variable $X_s$, $s \in S$, is fixed to one value
which we conventionally call $0$ \citep{armen2018towards}.
We assume that there is \emph{no} access to data obtained by selecting upon other values, such as $P(X_V \mid X_{s}=1)$.

\begin{definition}[Selected-Marginal Observational Model] \label{def_smo_model}
    \sloppy The selected-marginal observational (SMO) model of the DAG $\calD$ 
    with vertices partitioned as $(V,M,S)$, denoted by $\calM(\calD,V \mid S)$, is the set of all distributions 
    that can be obtained 
    from some $P(X_{V \dotcup M \dotcup S})$ that is Markov relative to
    $\calD$
    and where all visible variables have deterministic kernels
    by
    marginalizing $X_M$ and 
    taking the regular conditional probability distribution given
    $X_S=\zeros$.
    (Distributions where the probability density
    $P(X_S)$ is zero are excluded.)%
\end{definition}
Regular conditional probabilities are a generalization 
of conditional probabilities that allow conditioning 
on an event with probability zero, so long as 
the probability density is nonzero,
as described in \Cref{app:regular-conditional-probabilities}.

We now formalize the concept of two DAGs being distinguishable by observational data.

\begin{definition}[Observational Equivalence and Dominance] \label{def_obs_equiv}
    Let $\calD$ and $\calD'$ be two DAGs
    with vertices partitioned as $(V,M,S)$ and $(V,M',S')$ respectively. 
    We say that $\calD$ and $\calD'$ are \emph{observationally equivalent} if their SMO models are identical, i.e.~if $\calM(\calD,V \mid S)=\calM(\calD',V \mid S')$.
More generally, $\calD'$ \emph{observationally dominates} $\calD$ if $\calM(\calD, V\mid S)\subseteq\calM(\calD', V\mid S')$. 
\end{definition}

\subsection{Interventional equivalence}
\label{sec_interv_eq}
In this paper, we will describe equivalence relations between DAGs under intervention, 
and for these results to be as powerful as possible, we want to define a notion of interventional data
that is as rich as possible.
To this end, we let the interventional data include all soft interventions---ones that
set each variable to a distribution over values, rather than to just one value.
For the same reason, we use an Observe\&Do probing scheme \citep{ansanelli2024everything}--- one where, for each variable, the experimenter first records its natural value and then performs an intervention that forces it to take another (potentially different) value.  \blk
Following \citet{ansanelli2024everything}, 
we will denote the intervened version of $v$ as $v^\sharp$ (sharp)
and the non-intervened version as $v^\flat$ (flat). Then, the distribution 
obtained from the Observe\&Do probing scheme will be of the form $P(X_{V^\sharp},X_{V^\flat} \mid X_S=\zeros)$.

More specifically, we will assume that the experimenter knows the temporal ordering of the visible variables,\footnote{A discussion about this assumption can be found in Section 4 of \citet{ansanelli2024everything}.} and that they will follow this order when probing each visible variable. For concreteness, imagine the example of a drug trial, where the administration (or not) of a drug is encoded in a variable $X_d$, and the recovery (or not) is encoded in a variable $X_r$. The experimenter initially does not know if the drug causes recovery, but they definitely know that $X_d$ comes temporally before $X_r$. In this case, the Observe\&Do probing scheme consists of the following steps:  first, the experimenter records the natural value of $X_{d^\flat}$, i.e.~they ask each subject if they would like to take the drug or not. Then, the experimenter intervenes on $d$, forcing $X_{d^\sharp}$ to take a value that is potentially different from the value of $X_{d^\flat}$. That is, they assign each subject to the drug or the placebo, with the assigned treatment potentially differing from the subject’s expressed preference. After that, they will observe the value that the recovery variable exhibits (as a response to $X_{d^\sharp}$), and this value will be recorded in the variable $X_{r^\flat}$. For completeness, in the end the experimenter intervenes on $r$, obtaining a variable $X_{r^\sharp}$.

It is clear that the Observe\&Do probing scheme is more powerful than simple do-interventions or simple passive observations: from the distribution $P(X_{V^\sharp},X_{V^\flat} \mid X_S=\zeros)$, we can infer the distribution that would have been obtained had we performed any of these other probing schemes.\footnote{See more discussion about the relationship between the distribution $P(X_{V^\sharp},X_{V^\flat})$ (in the case without selection) and the distributions obtained from other probing schemes in Section 4 of \citet{ansanelli2024everything}.} In fact, as discussed in \citet{ansanelli2024everything}, the Observe\&Do probing scheme captures all information about a data generating process that could be obtained from 
interactions with the visible variables.\footnote{This discussion did not include the stronger \emph{edge} interventions~\citep{shpitser2016causal}, so these are a potential exception.} In particular, this implies that when we present conditions under which two DAGs are indistinguishable even under access to the Observe\&Do probing scheme with soft interventions (Section~\ref{sec:invariances}), such conditions will also imply that the two DAGs in question are indistinguishable under more commonly studied weaker probing schemes. 

Now, we will formally define the set of data obtained when the experimenter performs the Observe\&Do probing scheme with soft interventions, that is, if they can prepare the intervened variable $X_{v^\sharp}$ in a generic probability distribution $Q(X_{v^\sharp})$. \blk In the interest of clarity, 
let us first define the selected marginal interventional (SMI) distribution obtained from one specific soft intervention
and one specific set of kernels.

\begin{definition}[SMI Distribution] \label{def_smi_pair}
Given a DAG $\calD$ with 
vertices partitioned as $(V, M, S)$,
kernels $P(X_w \mid X_{\pa{w}}),w \in V \dotcup M \dotcup S$,
and a soft intervention $Q(X_{V^\sharp})$, 
the SMI distribution is defined as:
    \begin{gather}
    P_{Q(X_{V^\sharp})}(X_{V^\sharp},X_{V^\flat} \mid X_S=\zeros) = \frac{\int_{\dom{M}} P_{Q(X_{V^\sharp})}(X_{V^\sharp},X_{V^\flat},X_M,X_S=\zeros) \, dx_M}{
    \int_{\dom{V^\sharp,V^\flat,M}} P_{Q(X_{V^\sharp})}(X_{V^\sharp},X_{V^\flat},X_M,X_S=\zeros) \, d(x_{V^\sharp},x_{V^\flat},x_{M})} \nonumber\\
    \text{for } \quad P_{Q(X_{V^\sharp})}(X_{V^\flat},X_{V^\sharp},X_{M}, X_S) = 
Q(X_{V^\sharp})\prod_{y \in V^\flat \dotcup M \dotcup S} P(X_{y} \mid X_{\paN^\sharp(y)}),     
    \label{eq_Markov_interventional}
    \end{gather}

where
$\zeros$ is the zero vector,
$\paN^\sharp(y) := \{z^\sharp \mid z \in \paN(y) \cap V\} \cup (\paN(y) \setminus V)$,  and $\dom{A}$ denotes the set of all joint probability distributions of the variables in a set $A$.
The SMI distribution $P_{Q(X_{V^\sharp})}(X_{V^\sharp},X_{V^\flat} \mid X_S=\zeros)$ 
represents the effect of an intervention that fixes the visible variables to the distribution $Q(X_{V^\sharp})$.
\end{definition}
\ryan{Do I need to further change this language/notation to include continuous variables}

Note that flat variables are directly influenced by their sharp visible parents and ordinary non-visible parents, whereas the sharp variables are set by the intervened distribution $Q(X_{V^\sharp})$.

By considering all possible kernels and soft interventions, we 
define the totality of interventional data that can come from one causal structure.

\begin{definition}[SMI Model] \label{def_smi_model} \label{def_interventional_compatibility}
For a DAG $\calD$ with vertices partitioned as $(V,M,S)$, 
consider the set of pairs:
\begin{equation} 
    P_* = \big\{\big(Q(X_{V^\sharp}),P_{Q(X_{V^\sharp})}(X_{V^\sharp},X_{V^\flat}\cmid X_S=\zeros)\big
    ) : Q(X_{V^\sharp})\in \calP(\dom{V}) \big\}.
    \label{eq_abc} 
\end{equation}
{We say that $P_*$ is a \emph{set of selected-marginal interventional (SMI) pairs compatible with} $\calD$ if:
\begin{itemize}
    \item each distribution $P_{Q(X_{V^\sharp})}(X_{V^\sharp},X_{V^\flat}\cmid X_S=\zeros)$ is constructed according to \Cref{eq_Markov_interventional}, using a choice of kernels $\{P(X_w \mid X_{\pa{\calD}{w}}), w \in V \dot\cup M \dot\cup S\}$ that is fixed across the elements of $P_*$, and where the kernel of each visible variable is deterministic, and
    \item $P_*$ contains every element where $Q(X_{V^\sharp})\in\calP(\dom{V})$ is
    such that $P_{Q(X_{V^\sharp})}(X_S=\zeros)>0$ 
    (where this probability is computed using $Q(X_{V^\sharp})$ and the kernels).
\end{itemize}}
\blk

The \emph{selected-marginal interventional (SMI) model} of the DAG $\calD$ with vertices partitioned as $(V,M,S)$, denoted $\mathcal{C}(\calD, V\mid S)$, is the collection of all sets $P_*$ of SMI pairs compatible with $\calD$ 
under the partition $(V,M,S)$.
\end{definition}

To recap, the SMI model of $\cal D$ is a collection of sets of pairs, one set $P_*$ 
for each specific choice of kernels $P(X_w|X_{\pa{w}})$ for the DAG.
Each element of $P_*$ contains 
a soft intervention $Q(X_{V^\sharp})$ 
paired with its SMI distribution
$P_{Q(X_{V^\sharp})}(X_{V^\sharp,V^\flat}\cmid X_S=\zeros)$ --- $P_*$ has one pair for each possible soft intervention.
Overall, the SMI model describes the visible consequences of all soft interventions 
given all choices of kernels.

We now highlight four aspects of the SMI model definition.

Firstly, why does $P_*$ contain pairs of distributions
$\big(Q(X_{V^\sharp}),P_{Q(X_{V^\sharp})}(X_{V^\sharp}\cmid X_S=\zeros)\big)$
rather than only distributions $P_{Q(X_{V^\sharp})}(X_{V^\sharp}\cmid X_S=\zeros)$?
Given a soft intervention, the distribution over the intervened variables 
can change with selection: in general $P(X_{V^\sharp} \mid X_S=0) \neq Q(X_{V^\sharp})$.
In our formalism, both distributions are made available, 
which in practical terms means we are assuming that the experimenter 
can remember what experiment they have performed.
For example, in a graph $v \to s$ where $X_S=X_V$,
if the experimenter performs the intervention that sets $Q(x_{v^\sharp})$ to be $\operatorname{Bernoulli}(\frac{1}{2})$,
they will obtain the pair 
$\big(Q(X_{v^\sharp}=0)=Q(X_{v^\sharp}=1)=\frac{1}{2},\,P(X_{v^\sharp}=0 \mid X_s=0)=1\big)$. The fact that $P(X_{V^\sharp} \mid X_S=0) \neq Q(X_{V^\sharp})$ shows us that, even with only one visible variable, it is already possible to detect the presence of selection effects.

Secondly, notice that hard interventions---those that fix $X_V$ to a single value---are included in $P_*$ as the case where $Q(X_{V^\sharp})$ is a point distribution; that is, $Q(X_{V^\sharp})$ is an indicator function $\delta(X_V=x_V)$ 
that is equal to $1$ if $X_V=x_V$ is true, and $0$ otherwise. 
\ryan{Is Elie satisfied with this notation of / description of atomic interventions?}

Thirdly, the effect of interventions to arbitrary subsets of vertices can always be recovered from $P_*$.
In a graph with vertices $V$, if we want know the effect of intervening $\doo(X_A=x_A)$
where $A \subsetneq V$
then we can choose $Q_{V^\sharp}$ to intervene $\doo(X_A=x_A)$ 
and assign a uniform random distribution to $X_{V \setminus A}$,
and then condition on $X_{Z^\sharp}$ equalling its natural value. 
That is, the interventional distribution $P(X_Z=x_Z\cmid\doo(X_A=x_A))$ is equivalent to
$P_{Q_{V^\sharp}}(X_{Z^\flat}= x_Z \mid X_{A^\sharp}=x_A, X_{Z^\sharp}=x_Z)$ \citep{ansanelli2024everything}.

Fourthly, our definition uses selection, so in some models, 
we will obtain a selected intervened distribution $P_{Q(X_{V^\sharp})}(X_{V^\flat},X_{V^\sharp}\mid X_s=0)$ 
that is undefined.
Conceptually, this represents a situation where the experimenter performs an intervention, 
and then gets back no data.
For instance, perhaps one wants to learn the relationship between having a bank account and 
employment status
but only those with a bank account possess a phone line.
So, an experimenter forces some subjects to go without a bank account,
and attempts a phone survey about employment status, thus obtaining no data.
This scenario makes intuitive sense, so from a formal perspective, 
whenever an intervened distribution $P_{Q(X_{V^\sharp})}(X_{V^\flat},X_{V^\sharp}\mid X_s=0)$
is undefined then we omit it from $P_*$.

We will mention one alternative approach, however.
Instead, it is possible to enforce strict positivity on both the model and the set of available interventions.
In this approach, $Q(X_{V^\sharp})$ is required to have 
a density $q$ such that 
$Q(X_{V^\sharp}) > 0$ for $P$-almost all $x_V \in \dom{V}$.
Then, in \Cref{def_interventional_compatibility}, \eqref{eq_abc} is replaced by:
 	\begin{equation}
		P_* = \big\{\big(Q(X_{V^\sharp}),P_{Q(X_{V^\sharp})}(X_{V^\sharp},X_{V^\flat}\cmid X_S=\zeros)\big) : Q(X_{V^\sharp})\in \calP(\dom{V}) \,\big|\, Q(X_{V^\sharp}) \text{ has full support} \big\}.
  \label{eq_abcd}
	\end{equation}
Such interventions are called \emph{full-support interventions}, in contrast to the unrestricted interventions of \Cref{def_interventional_compatibility}. 
Having presented two alternatives, 
this paper will use unrestricted interventions throughout, 
while all of our results will be valid for both types of probing schemes.

Similar to observational equivalence, we can define interventional equivalence under access to the Observe\&Do probing scheme  as follows.

\begin{definition}[Observe\&Do  Equivalence and Dominance] \label{def_interv_equiv}
    Let $\calD$ and $\calD'$ be two DAGs
    with vertices partitioned as $(V,M,S)$ and $(V,M',S')$ respectively. 
    We say that $\calD$ and $\calD'$ are \emph{Observe\&Do  equivalent} if their SMI models are identical, i.e.~if $\calC(\calD, V\mid S)=\calC(\calD', V\mid S')$,
    and that $\calD'$ \emph{Observe\&Do  dominates} $\calD$ if $\calC(\calD, V\mid S)\subseteq\calC(\calD', V\mid S')$. 
\end{definition}

Whenever we informally say that two DAGs are \emph{interventionally equivalent}, we will be referring to Observe\&Do equivalence.
Clearly, any two DAGs that are Observe\&Do equivalent are also observationally equivalent. However, the converse is not true: for example, the graphs $v_1 \to v_2$ and $v_1 \gets v_2$ are observationally equivalent 
but not Observe\&Do equivalent.

We finish this section by stressing that, when we ask about observational or interventional equivalence, we start from two candidate causal structures and ask if there is, \emph{in principle}, the possibility of obtaining a set of data that would tell us that the underlying causal structure is definitely one and not the other. In other words, there exists \emph{one} set of data that is realizable by one of these causal structures but not by the other. Importantly, even if two DAGs $\D$ and $\D'$ are observationally (interventionally) inequivalent, it is still possible that the observational (interventional) data obtained from a particular experiment does not allow us to tell if the data generating process followed $\D$ or $\D'$. For example: a distribution that is completely factorizing, that is, where none of the visible variables depend on any other visible variable (be it under interventions or not) is realizable by \emph{every} causal structure. Therefore, if one starts with such a distribution as their data, one certainly will not be able to tell if it was generated by $\D$ or $\D'$ (if one does not impose any model-selection principles such as faithfulness).\footnote{A discussion about using model-selection principles to adjudicate between the different causal structures that realize our data can be found in Section 3.4.1 of~\citet{ansanelli2025observational}. That discussion focuses on observational data without selection effects, but the same ideas apply here.}
In other words, the fact that two DAGs $\D$ and $\D'$ are observationally (interventionally) inequivalent means that their SMO (SMI) models are \emph{not the same}, and not that their SMO (SMI) models have an empty overlap.

\blk

\section{Transformations that maintain interventional equivalence} \label{sec:invariances}

In this section, we will present a series of operations that 
 { preserve the SMI model of a DAG, that is, operations such that the initial and the final DAGs are Observe\&Do  equivalent. As a consequence, these operations also preserve equivalence under any probing scheme that is weaker than Observe\&Do.}

The first action we discuss is called \emph{exogenization}, because it makes all { marginalized} vertices \emph{exogenous} (i.e.~parentless). The second is called \emph{terminalization}, because it  makes all selected vertices \emph{terminal} (i.e.~childless). 
The following proposition defines these operations and establishes that they do not change the SMI model.
\begin{restatable}{proposition}{ExogAndTerm}
Let $\calD$ be a DAG with vertices partitioned as $(V,M,S)$ where \ldots
\begin{enumerate}[label=\alph*)]
\item \label{prop_exogenizing_interv}
$m \in M$. Let $\exogenize_m(\calD)$ be the DAG obtained from $\calD$ by \emph{exogenizing} $m$, that is, by i) adding an edge $l \to k$ from every $l \in \pa{\calD}{m}$ to every $k \in \ch{\calD}{m}$, and ii) deleting the edges $l\to m$ for $l \in \pa{\calD}{m}$. Then, $\calD$ and $\exogenize_m(\calD)$ are Observe\&Do equivalent; that is, $\calC(\calD,V \!\mid\! S) \!=\! \calC(\exogenize_m(\calD),V \!\mid\! S)$.
\item \label{prop_terminalizing}
$s \in S$. Let $\calD_{\underline{s}}$ be the graph obtained from $\calD$ by removing edges outgoing from $s$. Then, $\calD$ and $\calD_{\underline{s}}$ are Observe\&Do equivalent; that is, $\calC(\calD,V \mid S) = \calC(\calD_{\underline{s}},V \mid S)$.
\end{enumerate}

Note that the order in which these operations are applied does not matter.  Therefore, we can define a transformation, denoted by {\tt Exog}, that acts on a DAG by sequentially exogenizing all of the marginalized vertices of a DAG. Similarly,  we can define a transformation, denoted by {\tt Term}, that acts on a DAG by terminalizing all of the selected vertices of a DAG.

\label{prop_exog_term}
\end{restatable}

The operation of exogenization of a marginalized node, which was first presented by \citet{evans2016graphs}, is exemplified in \cref{fig:invariances-1}.

\begin{figure}[ht]
    \centering
\begin{subfigure}{.31\textwidth}
    \centering
    \begin{tikzpicture}
        [rv/.style={circle, draw, very thick, minimum size=5.5mm, inner sep=0.6mm}, 
         node distance=12mm, >=stealth]

    \node[rv] (v1) {$v_1$};
    \node[mv, right=5.5mm of v1] (m) {$m$};
    \node[rv, below right=5.5mm and 5.5mm of m] (v2) {$v_2$};
    \node[rv, above right=5.5mm and 5.5mm of m] (v3) {$v_3$};

    \draw[->, very thick] (v1) -- (m);
    \draw[->, very thick] (m) -- (v2);
    \draw[->, very thick] (v1) -- (v2);
    \draw[->, very thick] (m) -- (v3);
    \end{tikzpicture}
    \caption{A DAG $\calD$.}
    \label{fig:dag-to-exogenize}
\end{subfigure}
\begin{subfigure}{.31\textwidth}
    \centering
    \begin{tikzpicture}
        [rv/.style={circle, draw, very thick, minimum size=5.5mm, inner sep=0.6mm}, 
         node distance=12mm, >=stealth]

    \node[rv] (v1) {$v_1$};
    \node[mv, right=5.5mm of v1] (m) {$m$};
    \node[rv, below right=5.5mm and 5.5mm of m] (v2) {$v_2$};
    \node[rv, above right=5.5mm and 5.5mm of m] (v3) {$v_3$};

    \draw[->, very thick] (v1) -- (v2);
    \draw[->, very thick] (v1) -- (v3);
    \draw[->, very thick] (m) -- (v2);
    \draw[->, very thick] (m) -- (v3);
    \end{tikzpicture}
    \caption{The DAG {\tt Exog}($\calD$).}
    \label{fig:exognized-dag}
\end{subfigure}
\caption{Example of the application of \cref{prop_exog_term}a).} \label{fig:invariances-1}
\end{figure}

The proof of \cref{prop_exog_term} is in \Cref{app_exog_term}.
Exogenization is already known to preserve
 Observe\&Do equivalence 
when there are no selected variables \citep{ansanelli2024everything}.
We merely prove that after subsequent selection, the models remain equivalent.
For selection, the proof idea is that after edges from $s$ to its children are removed, those children
can behave just as they would when $X_s=0$, and then we will recover the same selected distribution.

Once {\tt Exog} and {\tt Term} have been performed, 
we will obtain a DAG $\mathtt{Exog}\circ \mathtt{Term}(\calD)$, 
wherein the marginalized vertices are parentless and the selected vertices are childless. 
It follows that the { marginalized} and selected vertices in the resulting DAG are not neighbours of any vertices of their own type. It is still possible, however, to have marginalized vertices that neighbour selected vertices.

Now, we proceed to the next two operations that preserve Observe\&Do equivalence. They consist of merging marginalized vertices that share a common selected child, and merging selected vertices that share a common marginalized parent.

\begin{restatable}{proposition}{MergingRestatable}
\label{prop_merging}
Let $\calD$ be a DAG with vertices partitioned as $(V, M, S)$ where $\calD=\mathtt{Exog}\circ\mathtt{Term}(\calD)$ and\ldots
\begin{enumerate}[label=\alph*),ref=\labelcref{prop_merging}\alph*)]
\item %
\label{le:merging-endovars-projection}
Let $m_1, m_2 \in M$, 
and $\ch{\calD}{m_1} \cap \ch{\calD}{m_2} \cap S \neq \emptyset$.
Obtain the modified DAG $\tilde{\calD}$ by merging $m_1$ and $m_2$ into one vertex $m$ with $\ch{\tilde{\calD}}{m}=\ch{\calD}{m_1} \cup \ch{\calD}{m_2}$.
Then, $\calD$ and $\tilde{\calD}$ are Observe\&Do equivalent; that is, $\calC(\calD,V \mid S)=\calC(\tilde{\calD},V\mid S)$.
\item %
\label{le:merging-selected-projection}
Let $s_1, s_2 \in S$,
and $\pa{\calD}{s_1} \cap \pa{\calD}{s_2} \cap M \neq \emptyset$.
Obtain the modified DAG $\tilde{\calD}$ by merging $s_1$ and $s_2$ into one vertex $s$ with $\pa{\tilde{\calD}}{s}=\pa{\calD}{s_1} \cup \pa{\calD}{s_2}$.
Then, $\calD$ and $\tilde{\calD}$ are Observe\&Do equivalent; that is, $\calC(\calD,V \mid S)=\calC(\tilde{\calD},V\mid (\{s\} \cup S) \setminus \{s_1,s_2\} )$.
\end{enumerate}

The order in which these operations are applied does not matter.  Therefore, we can define a transformation, denoted by {\tt MergeM}, that acts on a DAG by merging all of the marginalized vertices that can be merged by a). Similarly, we can define a transformation, denoted by {\tt MergeS}, that acts on a DAG by merging all of the selected vertices that can be merged by b).
\end{restatable}

For example, this implies that 
since the marginalized vertices $m_1$ and $m_2$ in 
\Cref{fig:smm-invariances-e-before} share the selected child $s$,
the DAG in \Cref{fig:smm-invariances-e-after} where the { marginalized} variables are merged 
is Observe\&Do equivalent 
to the original.
Since vertices $s_1$ and $s_2$ 
in \Cref{fig:smm-invariances-f-before} 
share the parent $m$,
when we merge them, 
the resulting graph \Cref{fig:smm-invariances-f-after}, 
then it is also equivalent to its original.

\begin{figure}[ht]
    \centering
\begin{subfigure}{.21\textwidth}
    \centering
    \begin{tikzpicture}
        [rv/.style={circle, draw, very thick, minimum size=5.5mm, inner sep=0.6mm}, 
         node distance=12mm, >=stealth]

    \node[mv] (m1) {$m_1$};
    \node[mv, right = 5.5mm of m1] (m2) {$m_2$};
    \node[rv, right = 5.5mm of m2] (v3) {$v_3$};
    \node[rv, below = 5.5mm of m2] (v2) {$v_2$};
    \node[rv, below = 5.5mm of m1] (v1) {$v_1$};
    \node[sv, below = 5.5mm of v3] (s) {$s$};

    \draw[->, very thick] (m1) -- (s);
    \draw[->, very thick] (m1) -- (v1);
    \draw[->, very thick] (m2) -- (s);
    \draw[->, very thick] (m2) -- (v2);
    \draw[->, very thick] (v3) -- (s);
    \end{tikzpicture}
    \caption{A DAG $\calD_1$.}
    \label{fig:smm-invariances-e-before}
\end{subfigure}%
\begin{subfigure}{.294\textwidth}
    \centering
    \begin{tikzpicture}
        [rv/.style={circle, draw, very thick, minimum size=5.5mm, inner sep=0.6mm}, 
         node distance=12mm, >=stealth]

    \node[mv] (y) {$m$};
    \node[rv, right = 5.5mm of y] (v3) {$v_3$};
    \node[rv, below = 5.5mm of y] (v2) {$v_2$};
    \node[rv, left = 5.5mm of v2] (v1) {$v_1$};
    \node[sv, below = 5.5mm of v3] (s) {$s$};

    \draw[->, very thick] (y) -- (s);
    \draw[->, very thick] (y) -- (v1);
    \draw[->, very thick] (y) -- (v2);
    \draw[->, very thick] (v3) -- (s);
    \end{tikzpicture}
    \caption{The DAG {\tt MergeM}($\calD_1$).}
    \label{fig:smm-invariances-e-after}
\end{subfigure}%
\begin{subfigure}{.21\textwidth}
    \centering
    \begin{tikzpicture}
        [rv/.style={circle, draw, very thick, minimum size=5.5mm, inner sep=0.6mm}, 
         node distance=12mm, >=stealth]

    \node[rv] (v1) {$v_1$};
    \node[rv, right = 5.5mm of v1] (v2) {$v_2$};
    \node[mv, right = 5.5mm of v2] (m) {$m$};
    \node[sv, below = 5.5mm of v1] (s1) {$s_1$};
    \node[sv, below = 5.5mm of v2] (s2) {$s_2$};
    \node[rv, below = 5.5mm of m] (v3) {$v_3$};

    \draw[->, very thick] (v1) -- (s1);
    \draw[->, very thick] (v2) -- (s2);
    \draw[->, very thick] (m) -- (s1);
    \draw[->, very thick] (m) -- (s2);
    \draw[->, very thick] (m) -- (v3);
    \end{tikzpicture}
    \caption{A DAG $\calD_2$.}
    \label{fig:smm-invariances-f-before}
\end{subfigure}%
\begin{subfigure}{.295\textwidth}
    \centering
    \begin{tikzpicture}
        [rv/.style={circle, draw, very thick, minimum size=5.5mm, inner sep=0.6mm}, 
         node distance=12mm, >=stealth]

    \node[rv] (v1) {$v_1$};
    \node[rv, right = 5.5mm of v1] (v2) {$v_2$};
    \node[mv, right = 5.5mm of m2] (m) {$m$};
    \node[sv, below = 5.5mm of v2] (y) {$s$};
    \node[rv, below = 5.5mm of m] (v3) {$v_3$};

    \draw[->, very thick] (v1) -- (y);
    \draw[->, very thick] (v2) -- (y);
    \draw[->, very thick] (m) -- (y);
    \draw[->, very thick] (m) -- (v3);
    \end{tikzpicture}
    \caption{The DAG {\tt MergeS}($\calD_2$).}
    \label{fig:smm-invariances-f-after}
\end{subfigure}%
\caption{Example of the application of 
\Cref{prop_merging}.
} \label{fig:smm-invariances-ef}
\end{figure}
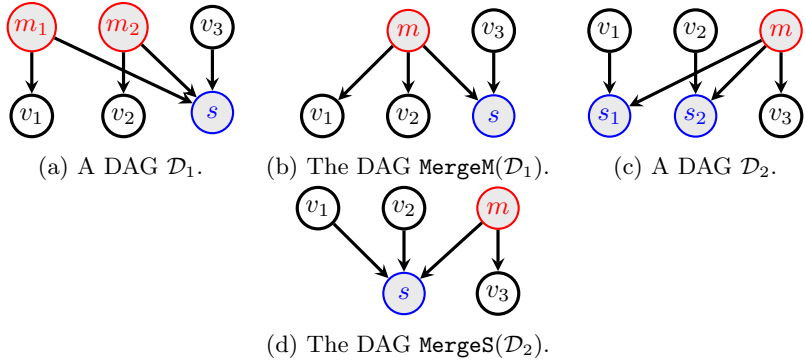

The proof of \cref{prop_merging} is in \Cref{app:smm-invariances-e}.
For Proposition~\labelcref{le:merging-endovars-projection}, to prove that $\tilde{\calD}$ can realize any data that is generated by $\calD$, 
our approach is to let $X_m$ equal the Cartesian product of $X_{m_1}$ and $X_{m_2}$. Conversely, to prove that $\calD$ can realize any data that is realizable by $\tilde{\calD}$, we let both variables $X_{m_1}$ and $X_{m_2}$ of $\calD$ sample over the whole domain of $X_m$ in $\tilde{\calD}$, and take their shared selected child to be $0$ 
precisely when $X_{m_1}=X_{m_2}$.
For Proposition~\labelcref{le:merging-selected-projection}, to prove that $\tilde{\calD}$ can realize any data that is generated by $\calD$, we 
will let $X_s$ equal Cartesian product of $X_{s_1}$ and $X_{s_2}$. Conversely, to show that $\calD$ can realize any data generated by $\tilde{\calD}$, we will let $X_m$ (the shared parent of $s_1,s_2$)
be the Cartesian product of two variables, one sampled uniformly over the domain of $X_{\pa{\calD}{s_1}}$ and the other sampled uniformly over the domain of $X_{\pa{\calD}{s_2}}$. Furthermore, we take $s_1$ and $s_2$ to equal $0$ precisely when the two variables that compose $X_m$ equal $X_{\pa{\calD}{s_1}}$ and $X_{\pa{\calD}{s_2}}$, respectively. Then, under selection, $X_{s_1}$ and $X_{s_2}$ will behave just as if they had direct access to $\pa{\calD}{s_1}\cup \pa{\calD}{s_2}$. 

After applying \cref{prop_exog_term,prop_merging}, we obtain the DAG
$\mathtt{MergeS}\circ\mathtt{MergeM}\circ\mathtt{Exog}\circ\mathtt{Term}(\calD)$ where each 
{ marginalized} or selected variable has at most one non-visible neighbour, which must be of the opposite type. However, each non-visible variable may still have multiple visible neighbours.

The next operation that preserves Observe\&Do equivalence deals with arrows $m\to s$ where the marginalized vertex $m$ has multiple visible children and/or the selected vertex $s$ has multiple visible parents, such as in \Cref{fig_deleteedge_a}. This operation transforms the DAG into one where every marginalized vertex either has only visible children or it has exactly one selected child and one visible child, and every selected vertex either has only visible parents or it has exactly one marginalized parent and one visible parent, such as in \Cref{fig_deleteedge_b}.

\begin{restatable}{proposition}{SplittingArrows} \label{prop:splitarrows}
    Let $\calD$ be a DAG with vertices partitioned as $(V,M,S)$ such that $\calD=\mathtt{MergeS}\circ\mathtt{MergeM}\circ\mathtt{Exog}\circ\mathtt{Term}(\calD)$, 
    and $\calD$ contains the edge $m \to s$, where $m\in M$ and $s\in S$. Let $V_s$ denote all visible vertices which are parents of $s$, i.e.~$V_s\coloneqq V \cap \pa{\calD}{s}$, and similarly let $V_m$ denote all visible vertices which are children of $m$, i.e.~$V_m\coloneqq V \cap \ch{\calD}{m}$. If either $V_s$ or $V_m$ are non-singleton sets (that is, if $m$ has more than one visible child or $s$ has more than one visible parent), then we perform the following operations to create a new DAG $\tilde{\calD}$:
    \begin{enumerate}
        \item For every ordered pair of visible vertices $(a,b)$ with $a \in V_s$ and $b \in V_m$         %
        (i.e.\ for every walk $a \to s \gets m \to b$) \blk
        add the new selected vertex $s_{ab}$ and the new marginalized vertex $m_{ab}$ 
        along with the three new edges
        $a \to s_{ab} \gets m_{ab} \to b$. 
        \par\noindent {(Note that if $a \in V_m \cap V_s$, 
        we will add the edges $a \to s_{aa} \gets m_{aa} \to a$)}.
        
        \item Delete the edge $m\to s$.
    \end{enumerate}
    Then, $\calD$ and $\tilde{\calD}$ are Observe\&Do equivalent; that is, $\calC(\calD,V \mid S)=\calC(\tilde{\calD},V\mid S\cup S')$
    where $S'$ denotes the set of added selected vertices.

The order of application of these operations on different edges $m\to s$ does not matter.  Therefore, we can define a transformation, denoted by $\DropMtoS$, that acts on a DAG by  applying the steps above to each edge $m \to s$ 
until every edge $m \to s$ satisfies that $m$ has only one visible child 
and $s$ has only one visible parent.
    
\end{restatable}

The proof of \Cref{prop:splitarrows} is presented in \Cref{app_split_arrows}. The intuition is that all that
the paths in \Cref{fig_deleteedge_a} that go through $m \to s$ do is to
allow children of $m$ to depend arbitrarily on parents of $s$.
So, we need a way to assign the children of $m$ 
in \Cref{fig_deleteedge_b}
a distribution $P(X_{v_3},X_{v_4} \mid X_{v_1},X_{v_2},X_S=\zeros)$
that matches the same distribution in \Cref{fig_deleteedge_a}.
To achieve this, we will sample each $X_{m_{ij}}$ uniformly over the domain of $X_{v_i}$
and let each $X_{s_{ij}}$ equal zero precisely when $X_{m_{ij}}=X_{v_i}$.
Then, under selection, each child variable $v_j$ essentially has access to the 
assignment of each parent $v_i$.
We can then use $X_m$ to ensure that $X_{v_3}$ and $X_{v_4}$ have the right distribution 
given any joint assignment $x_{v_1},x_{v_2}$.
Conversely, we can reproduce the behaviour of $m_{13},m_{14}$ in $\calD$ by having the
variable $X_m$ be a Cartesian product of $X_{m_{13}}$ and $X_{m_{14}}$.

\Cref{fig_deleteedge-loop} provides a second example of the application of \Cref{prop:splitarrows}, where the vertices $v_1$ and $v_2$ are simultaneously parents of $s$ and children of $m$.

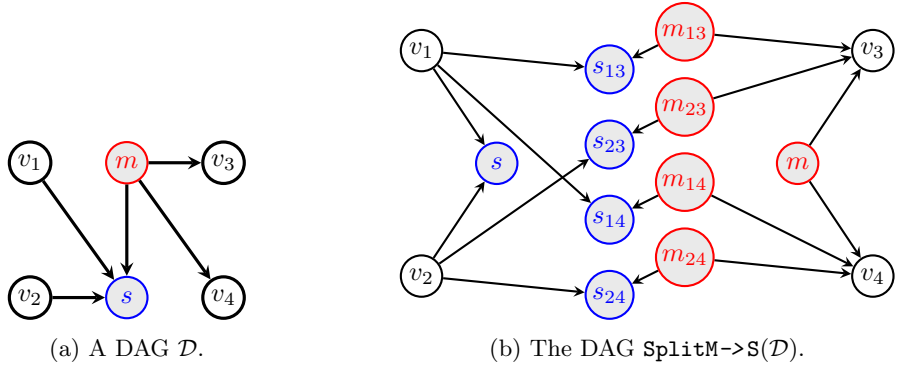
\begin{figure}[ht]
    \centering
    \begin{subfigure}{.45\textwidth}
    \centering
\begin{tikzpicture}
    [rv/.style={circle, draw, very thick, minimum size=5.5mm, inner sep=0.6mm}, >=stealth]
    
    \node[rv] (v1) {$v_1$};
    \node[mv, right=7mm of v1] (m) {$m$};
    \node[rv, right=7mm of m] (v3) {$v_3$};
    
    \node[rv, below=12mm of v1] (v2) {$v_2$};
    \node[sv, right=7mm of v2] (s) {$s$};
    \node[rv, right=7mm of s] (v4) {$v_4$};
    
    \draw[->, very thick] (v1) -- (s);
    \draw[->, very thick] (v2) -- (s);
    \draw[->, very thick] (m) -- (s);
    \draw[->, very thick] (m) -- (v3);
    \draw[->, very thick] (m) -- (v4);
\end{tikzpicture}

    \caption{A DAG $\calD$.} 
    \label{fig_deleteedge_a}
    \end{subfigure}
    \begin{subfigure}{.45\textwidth}
    \centering
        \begin{tikzpicture}[>=stealth]
\node[rv] (v1) at (0,  1.5) {$v_1$};
\node[rv] (v2) at (0, -1.5) {$v_2$};

\node[rv] (v3) at (6,  1.5) {$v_3$};
\node[rv] (v4) at (6, -1.5) {$v_4$};

\node[sv] (s12) at (1, 0) {$s$};
\draw[->, thick] (v1) -- (s12);
\draw[->, thick] (v2) -- (s12);

\node[mv] (m34) at (5, 0) {$m$};
\draw[->, thick] (m34) -- (v3);
\draw[->, thick] (m34) -- (v4);

\node[sv] (s13) at (2.5,  1.25) {$s_{13}$};
\node[mv] (m13) at (3.5,  1.75) {$m_{13}$};  %

\draw[->, thick] (v1) -- (s13);
\draw[->, thick] (m13) -- (s13);
\draw[->, thick] (m13) -- (v3);

\node[sv] (s14) at (2.5, -0.75) {$s_{14}$};
\node[mv] (m14) at (3.5, -0.25) {$m_{14}$};  %

\draw[->, thick] (v1) -- (s14);
\draw[->, thick] (m14) -- (s14);
\draw[->, thick] (m14) -- (v4);

\node[sv] (s23) at (2.5,  0.25) {$s_{23}$};
\node[mv] (m23) at (3.5,  0.75) {$m_{23}$}; %

\draw[->, thick] (v2) -- (s23);
\draw[->, thick] (m23) -- (s23);
\draw[->, thick] (m23) -- (v3);

\node[sv] (s24) at (2.5, -1.75) {$s_{24}$};
\node[mv] (m24) at (3.5, -1.25) {$m_{24}$}; %

\draw[->, thick] (v2) -- (s24);
\draw[->, thick] (m24) -- (s24);
\draw[->, thick] (m24) -- (v4);
    \end{tikzpicture}
    \caption{The DAG $\DropMtoS(\calD)$.} 
    \label{fig_deleteedge_b}
    \end{subfigure}   
    \caption{First example of the application of \Cref{prop:splitarrows}.}
    \label{fig_deleteedge}
\end{figure}

\begin{figure}[ht]
    \centering
    \begin{subfigure}{.43\textwidth}
    \centering
\begin{tikzpicture}
    [rv/.style={circle, draw, very thick, minimum size=5.5mm, inner sep=0.6mm}, >=stealth]
    
    \node[rv] (v1) {$v_1$};
    \node[sv, right=7mm of v1] (s) {$s$};
    \node[mv, above=7mm of s] (m) {$m$};
    \node[rv, right=7mm of s] (v4) {$v_2$};
    \draw[->, very thick] (m) -- (v1);
    \draw[->, very thick] (v1) -- (s);
    \draw[->, very thick] (v4) -- (s);
    \draw[->, very thick] (m) -- (s);
    \draw[->, very thick] (m) -- (v4);
\end{tikzpicture}

    \caption{A DAG $\calD$.} 
    \label{fig_deleteedge_a-loop}
    \end{subfigure}
    \begin{subfigure}{.48\textwidth}
    \centering
        \begin{tikzpicture}[>=stealth]
    \node[rv] (v1) {$v_1$};
    \node[sv, right=7.5mm of v1, yshift=4mm] (s12) {$s$};
    \node[mv, right=7.5mm of v1, yshift=-4mm] (m) {$m$};
    \node[rv, right=20mm of v1] (v4) {$v_2$};

\draw[->, thick] (v1) -- (s12);
\draw[->, thick] (v4) -- (s12);
\draw[->, thick] (m) -- (v1);
\draw[->, thick] (m) -- (v4);

 \node[sv, above=5mm of v1, xshift=7mm] (s14) {$s_{12}$};
 \node[sv, below=5mm of v4, xshift=-7mm]  (s21) {$s_{21}$};
\node[mv, above=5mm of v4, xshift=-7mm] (m14) {$m_{12}$};
 \node[mv, below=5mm of v1, xshift=7mm]  (m21) {$m_{21}$};

\draw[->, thick] (m21) -- (v1);
\draw[->, thick] (m21) -- (s21);
\draw[->, thick] (v4) -- (s21);

\draw[->, thick] (v1) -- (s14);
\draw[->, thick] (m14) -- (s14);
\draw[->, thick] (m14) -- (v4);

\node[sv, right=7mm of v4] (s24) {$s_{22}$};
\node[mv, above=5mm of s24] (m24) {$m_{22}$};

\node[sv, left=7mm of v1] (s11) {$s_{11}$};
\node[mv, above=5mm of s11] (m11) {$m_{11}$};

\draw[->, thick] (v4) -- (s24);
\draw[->, thick] (m24) -- (s24);
\draw[->, thick] (m24) -- (v4);

\draw[->, thick] (v1) -- (s11);
\draw[->, thick] (m11) -- (s11);
\draw[->, thick] (m11) -- (v1);
    \end{tikzpicture}
    \caption{The DAG $\DropMtoS(\calD)$,  which visually resembles the Pok\'emon Magnemite.} 
    \label{fig_deleteedge_b-loop}
    \end{subfigure}   
    \caption{Second example of the application of \Cref{prop:splitarrows}.} \label{fig_deleteedge-loop}
\end{figure}
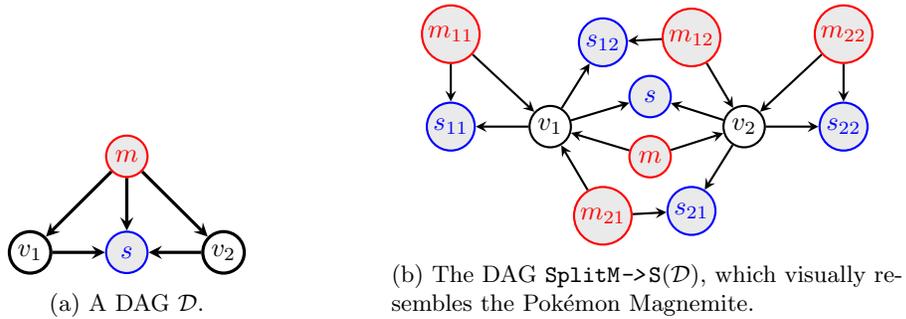

In the resulting DAG 
$\DropMtoS \circ \mathtt{MergeS}\circ\mathtt{MergeM}\circ\mathtt{Exog}\circ\mathtt{Term}(\calD)$,
whenever two non-visible variables are neighbours, 
each has precisely one visible neighbour.
Essentially, our DAG now consists of:
1) { marginalized} vertices with exclusively visible children, 
2) selected vertices with exclusively visible parents, 
3) directed edges $a\to b$ between visible vertices, and
4) paths $a \to s_{ab} \gets m_{ab} \to b$ between visible vertices $a,b$, which will be called \emph{special edges} between $a$ and $b$. 

After selection, both directed edges and special edges 
allow $X_{b}$ to vary in response to an intervention to $X_a$. This is the idea behind the next operation that preserves Observe\&Do equivalence: replacing directed edges $a\to b$ by special edges $a \to s_{ab} \gets m_{ab} \to b$ whenever $a$ has a selected child and $b$ has a marginalized parent.

\begin{restatable}{proposition}{interchangeedges} \label{lemma_interchangeedges}
    Let $\calD$ be a DAG with vertices partitioned as $(V,M,S)$ such that $\calD= \DropMtoS \circ \mathtt{MergeS}\circ\mathtt{MergeM}\circ\mathtt{Exog}\circ\mathtt{Term}(\calD)$, and let $a\to b$ be an edge of $\calD$
    where the vertex $a$ has at least one selected child and the vertex $b$ has at least one marginalized parent.
    Let the DAG $\calD'$ be constructed by 
    replacing $a \to b$ with a special edge; that is, with
   a path $a \to s_{ab} \gets m_{ab} \to b$
    where $s_{ab}$ is a selected vertex and $m_{ab}$ is a marginalized vertex.
    Then, $\calD$ and $\calD'$ are Observe\&Do equivalent; that is, $\calC(\calD,V \mid S)=\calC(\calD',V\mid S \cup \{s_{ab}\})$.

    The order in which the directed edges are transformed into special edges does not matter. Therefore, we can define a transformation, denoted by $\ToBidir$, that acts on a DAG by replacing every directed edge between visible vertices with a special edge, whenever the parent has a selected child and the child has a marginalized parent. 
\end{restatable}

The proof of \cref{lemma_interchangeedges}  is presented in \Cref{app_interchange}. The intuition is 
that to produce a distribution from $\calD$ on $\calD'$, 
we will set $X_{s_{ab}}$ to equal $0$ if and only if $X_a$ equals $X_{m_{ab}}$, 
and set $X_{b}$ to depend on $X_{m_{ab}}$ in $\calD'$ in the same way that it depends on $X_a$ in $\calD$. 
Conversely, to reproduce a distribution from $\calD'$ on $\calD$, 
we can obtain $P'(X_{b}\mid X_{s_{ab}}=0,X_{\pa{\calD}{b}})$ from $\calD'$
and use this as the kernel for $X_{b}$ in $\calD$.

At first sight it may be unclear why we choose to replace the simple directed edges $a \to b$ 
with the more complicated-looking special edges. We do that because in many cases the inverse transformation, that is replacing special edges with regular directed edges, would induce a cycle. Therefore, since our aim is to conflate these two indistinguishable patterns to obtain a data structure that captures all interventional equivalences, we must use \cref{lemma_interchangeedges} to transform directed edges into special edges wherever possible.

Our final two operations that preserve Observe\&Do equivalence consist in removing \emph{redundant} marginalized and selected vertices, which are marginalized vertices whose children are a subset of the children of another marginalized vertex or selected vertices whose parents are a subset of the parents of another selected vertex. An example is shown in  \Cref{fig_remove_redund}. The removal of redundant marginalized vertices in the case without selection effects was already proven to preserve observational equivalence \citep{evans2016graphs}, and more recently so was Observe\&Do equivalence \citep{ansanelli2024everything}.

\begin{restatable}{proposition}{removalrestatable}
\label{prop_remove_redund}
Let $\calD$ be a DAG with vertices partitioned as $(V, M, S)$ such that  $\calD=\ToBidir\circ\DropMtoS\circ\mathtt{MergeS}\circ\mathtt{MergeM}\circ\mathtt{Exog}\circ\mathtt{Term}(\calD)$, and\ldots
\begin{enumerate}[label=\alph*)]
    \item There are $m_1, m_2 \in M$ such that $\ch{m_1}\subseteq \ch{m_2}$. Let  $\calD_{-m_1}$ be the subgraph of $\calD$ on vertices other than $m_1$. Then, $\calD$ and $\calD_{-m_1}$ are Observe\&Do equivalent, so $\calC(\calD,V \mid S) = \calC(\calD_{-m_1},V \mid S)$.
    \item There are $s_1, s_2\in S$ such that $\pa{\calD}{s_1}  \subseteq \pa{s_2}$. Let  $\calD_{-s_1}$ be the subgraph of $\calD$ on vertices other than $s_1$. Then, $\calD$ and $\calD_{-s_1}$ are Observe\&Do equivalent, so $\calC(\calD,V \mid S) = \calC(\calD_{-s_1},V \mid S \setminus \{s_1\})$.
\end{enumerate}

 The order in which redundant marginalized and selected vertices are removed does not matter. Therefore, we can define a transformation, denoted by {\tt RmvRedM}, that acts on a DAG by applying a) until no redundant marginalized vertices remain. Similarly, we can define a transformation, denoted by {\tt RmvRedS}, that acts on a DAG by applying b) until no redundant selected vertices remain.
\end{restatable}

The proof of \cref{prop_remove_redund} is presented in \Cref{app_proof_removeredund}.
Part (a) follows from the fact that $\calD$ and $\calD_{-m_1}$ have the same
Observe\&Do model by Theorem 1 of \citet{ansanelli2024everything}.
For part (b), the idea is that we can take the variable $X_{s_2}$ in $\calD_{-s_1}$ to 
equal zero when $X_{s_1}=0$ and $X_{s_2}=0$ in $\calD$.

\begin{figure}[ht]
    \centering
    \begin{subfigure}{.24\textwidth}
    \centering
    \begin{tikzpicture}
        [rv/.style={circle, draw, very thick, minimum size=5.5mm, inner sep=0.6mm}, 
         node distance=12mm, >=stealth]

    \node[mv] (y) {$m_2$};
    \node[rv, below = 5.5mm of y] (v2) {$v_2$};
    \node[rv, left = 5.5mm of v2] (v1) {$v_1$};
    \node[rv, right = 5.5mm of v2] (v3) {$v_3$};
    \node[mv, above = 5.5mm of v3] (m1) {$m_1$};
    \node[sv, below = 5.5mm of v2] (s2) {$s_2$};
    \node[sv, below = 5.5mm of v3] (s1) {$s_1$};

    \draw[->, very thick] (y) -- (v3);
    \draw[->, very thick] (y) -- (v1);
     \draw[->, very thick] (m1) -- (v3);
    \draw[->, very thick] (m1) -- (v2);
    \draw[->, very thick] (y) -- (v2);
    \draw[->, very thick] (v3) -- (s1);
    \draw[->, very thick] (v3) -- (s2);
    \draw[->, very thick] (v2) -- (s1);
    \draw[->, very thick] (v2) -- (s2);
    \draw[->, very thick] (v1) -- (s2);
    \end{tikzpicture}
    \caption{A DAG $\calD$.}
    \label{fig:smm-invariances-prop5-before}
\end{subfigure} \hspace{5mm}
\begin{subfigure}{.6\textwidth}
    \centering
    \begin{tikzpicture}
        [rv/.style={circle, draw, very thick, minimum size=5.5mm, inner sep=0.6mm}, 
         node distance=12mm, >=stealth]

    \node[mv] (y) {$m_2$};
    \node[rv, below = 5.5mm of y] (v2) {$v_2$};
    \node[rv, left = 5.5mm of v2] (v1) {$v_1$};
    \node[rv, right = 5.5mm of v2] (v3) {$v_3$};
    \node[sv, below = 5.5mm of v2] (s2) {$s_2$};

    \draw[->, very thick] (y) -- (v3);
    \draw[->, very thick] (y) -- (v1);
    \draw[->, very thick] (y) -- (v2);
    \draw[->, very thick] (v3) -- (s2);
    \draw[->, very thick] (v2) -- (s2);
    \draw[->, very thick] (v1) -- (s2);
    \end{tikzpicture}
    \caption{The DAG $\mathtt{RmvRedS}\circ\mathtt{RmvRedM}(\calD)$.}
    \label{fig:smm-invariances-prop5-after}
\end{subfigure}%
    \caption{Example of the application of \Cref{prop_remove_redund}.}
    \label{fig_remove_redund}
\end{figure}
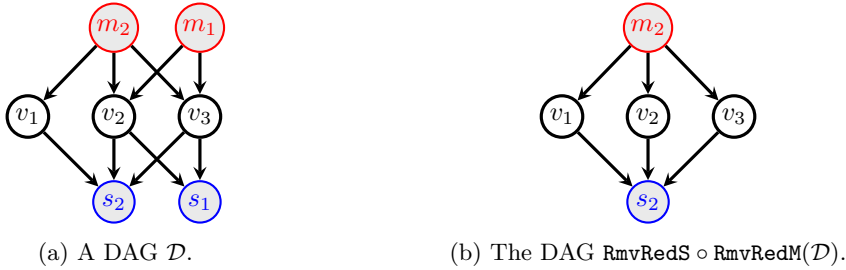

After applying all of the operations of 
\Cref{prop_exog_term,prop_merging,prop:splitarrows,lemma_interchangeedges,prop_remove_redund}
on an initial DAG with vertices partitioned as $(V,M,S)$,
we obtain a DAG $\calD$ with partition $(V,M',S')$, such that
\begin{equation}
    \calD=\mathtt{RmvRedS}\circ\mathtt{RmvRedM}\circ \ToBidir \circ\texttt{SplitM->S}\circ\mathtt{MergeS}\circ\mathtt{MergeM}\circ\mathtt{Exog}\circ\mathtt{Term}(\calD). \nonumber
\end{equation}
\ryan{This statement would be a bit more rigorous with a definition of partitioned DAG}

When this condition is satisfied, we will call $\calD$
a \emph{canonical DAG}.
In a canonical DAG, the { marginalized} (or selected) variables with only visible neighbours 
will be called the
\emph{{ marginalized} (or selected) variables of the faces}; the reason for this will be evident from
\Cref{def:slp-special-case}.

The operation ($\mathtt{RmvRedS}\circ \ldots \circ \mathtt{Term}(\calD)$) that yields 
the canonical DAG has time complexity $\mathcal{O}(n^4)$ in the worst case,
where $n$ is the number of vertices of the DAG. This is explained in \Cref{app:algorithmic}.
\blk

\section{Selected-marginalized directed graphs (smDGs)} \label{sec:smDGs}
\subsection{smDGs and the selected-latent projection}
\blk

In this section, we will formulate a data structure that already encodes all of the regularities of the canonical DAG; this is a generalization of the mDAG~\citep{evans2016graphs} for the case where there are selection effects. 
The main idea is to view the canonical DAG as
a set of visible vertices connected by three kinds of relationships:
latent confounding, 
selection effects,
and paths of the form $a \to b$ or $a \to s' \gets m' \to b$.
To represent latent confounding and selection effects, our new data structure will 
use two \emph{independence systems}:

\begin{definition}[Independence System]
An \emph{independence system} $\B$ over a finite set $V$ is a collection of subsets of $V$
that contains the empty set $\emptyset \in \B$, and such that the system is closed under taking subsets:
if $B \in \B$ and $A \subseteq B$ then $A \in \B$.
We call $A\in \B$ a \emph{maximal face} of $\B$ if it is not a strict subset of any other set in $\B$. 
\end{definition} %

We will use an independence system $\calL$ to represent  latent confounding, 
and another independence system $\calS$ to represent selection effects.
Each subset $V_m\in \calL$ corresponds to the visible children of a marginalized vertex $m$ in the canonical DAG. An independence system neatly captures the fact that $\mathtt{RmvRedM}$ has removed all marginalized variables whose children are a subset of another marginalized variable's children. The original definition of mDAGs~\citep{evans2016graphs} uses a \emph{simplicial complex}, which is an independence system that contains every singleton set, i.e.~such that $\{v\}\in \B$ for all $v\in V$. 
Our reason for relaxing this assumption is that $\{v\} \not \in \calL$ represents the case where $X_v$ must be a deterministic function of its visible parents.

Each subset $V_s\in \calS$ corresponds to the visible parents of one selected vertex $s$ of the canonical DAG. Similar to the case of latent confounding, 
nested selected variables have been removed by $\mathtt{RmvRedS}$,
making an independence system a natural representation of this structure. 
The presence or absence of a vertex $\{v\} \in \calS$
in the selected independence system is meaningful, as it can be detected with soft interventions.

Recall that \Cref{lemma_interchangeedges} 
has shown that whenever $a$ has selected children and $b$ has marginalized parents, a path $a \to b$ can be replaced by a path $a \to s \gets m \to b$ while preserving Observe\&Do equivalence. 
In our new data structure, both types of paths  $a \to b$ and $a \to s \gets m \to b$ will be represented by a directed edge from $a$ to $b$, and the presence or absence of a selected face that includes $a$ and a marginalized face that includes $b$ is what will tell us what is the type of path that appears in the corresponding canonical DAG. With this, note that this directed graph in the new data structure need not be acyclic. Our new data structure is defined as:

\begin{definition}[smDG] \label{def_smgraph}
A \emph{selected-marginalized directed graph} (smDG) $\G$ is a quadruple $(V,\E,\calL, \calS)$ where $(V,\E)$ is a 
directed graph called the \emph{directed structure},
$\calL$ is the \emph{marginalized independence system},
and $\calS$ is the \emph{selected independence system}.
(The directed structure is allowed to contain cycles and even self-loops $v \to v$.)
\end{definition}

It will be useful to fix some notation for an smDG $\calG=(V,\E, \calL, \calS)$.
The parents of a vertex $b$, denoted $\pa{\calG}{b}$, 
are the set of vertices such that $a \to b$ is in $\calE$.
This includes both the cases where the canonical graph contains the path
$a \to b$ or a path $a \to s \gets m \to b$.
The subgraph $\calG_{V'} = (V',\E', \calL', \calS')$ 
is the smDG with sets of edges $\E'$, marginalized faces $\calL'$, and selected faces $\calS'$ in $\G$
that contain only those vertices $V'$.

When we draw smDGs in this paper, the graph $(V, \E)$ will have ordinary black edges, 
the maximal faces of $\calL$ will be represented by red hyperedges with arrowheads,
and the maximal faces of $\calS$ will be represented by blue hyperedges without arrowheads.
When a maximal marginal (or selected) face contains only one vertex $v$, we will illustrate it 
using a red dot incoming to $v$
(or a blue dot with an undirected edge to $v$),
as shown in \Cref{fig:example-slpb}.
If the directed structure of an smDG is acyclic, all of the vertices are in some face of the marginalized independence system\footnote{This is required because, as already discussed, \citet{evans2016graphs} uses a simplicial complex instead of an independence system to define mDAGs.}, and 
the selected independence system contains only the empty set, 
then the smDG is simply an mDAG~\citep{evans2016graphs}.

We connect a canonical DAG to an smDG via the operation of  \emph{selected-latent projection}, named in analogy to the latent projection operation \citep{Verma1988soundness,evans2016graphs}.
This definition will use the notion of a closure $\cl(A)$ of a set of vertices $A$, 
defined as the set of all subsets of $A$.

\begin{definition}[Selected-latent projection] \label{def:slp-special-case}
Let $\calD^\text{can}$ be a canonical DAG with vertices partitioned as $(V,M,S)$.
Its selected-latent projection (SLP), denoted $\slp(\calD^\text{can},V \mid S)$, is the smDG $(V,\E,\calL, \calS)$ where the directed structure 
$(V,\E)$ is the induced subgraph $\calD^{\text{can}}_V$
with an added edge $a \to b$ for each path $a \to s \gets m \to b$ 
where $a,b \in V,m \in M,s \in S$,
the marginalized independence system $\calL$ is the union $\bigcup_{m \in M}\cl(\ch{\calD^\text{can}}{m}\cap V)$,
and the selected independence system $\calS$ is the union $\bigcup_{s \in S} \cl(\pa{\calD^\text{can}}{s}\cap V)$.

For a non-canonical DAG $\calD$  with vertices partitioned as $(V,M,S)$, 
we define its selected-latent projection $\slp(\calD,V \mid S)$
as the SLP of $\calD^\text{can}$, where $\calD^\text{can}$
is the canonical DAG obtained from $\calD$ via
the rules described in \Cref{sec:invariances}.
\end{definition}

\begin{figure}[ht]
    \centering
    \begin{subfigure}{.5\textwidth}\centering
        \begin{tikzpicture}
            [rv/.style={circle, draw, very thick, minimum size=5.5mm, inner sep=0.5mm}, 
             node distance=20mm, >=stealth]

            \node[rv] (a) {$a$};
            \node[rv, right =13mm of a, yshift=-10mm] (b) {$b$};
            \node[rv, right =6mm of b, yshift=4mm] (c) {$c$};
            \node[rv, right =6mm of c] (d) {$d$};

            \node[mv, below=4mm of a, xshift=-6mm] (ma) {$m_1$};
            \node[sv, above=4mm of a, xshift=-6mm] (s1) {$s_1$};

            \node[sv, above =12mm of b] (s2) {$s_2$};
            \node[mv, below =12mm of a, xshift=2mm] (mb) {$m_2$};
            \node[mv, below =4mm of a, xshift=7mm] (mab) {$m_5$};
            \node[sv, below =-5mm of a, xshift=13mm] (sab) {$s_4$};
            \node[mv, below =5mm of c] (mc) {$m_3$};
            \node[mv, below =5mm of d] (md) {$m_4$};
            \node[sv, above =4mm of c, xshift=6mm] (s3) {$s_3$};

            \draw[->, very thick] (mab) -- (a);
            \draw[->, very thick] (mab) -- (sab);
            \draw[->, very thick] (b) -- (sab);
            \draw[->, very thick] (a) -- (s1);

            \draw[->, very thick] (a) -- (s2);
            \draw[->, very thick] (b) -- (s2);
            \draw[->, very thick] (c) -- (s2);

            \draw[->, very thick] (ma) -- (s1);
            \draw[->, very thick] (ma) -- (a);
            \draw[->, very thick] (mb) -- (a);
            \draw[->, very thick] (mb) -- (b);
            \draw[->, very thick] (mc) -- (c);
            \draw[->, very thick] (md) -- (d);
            \draw[->, very thick] (c) -- (s3);
            \draw[->, very thick] (d) -- (s3);
        \end{tikzpicture}
        \caption{A canonical DAG}
        \label{fig:example-dag}
    \end{subfigure}%
    \begin{subfigure}{.5\textwidth}\centering
        \begin{tikzpicture}[rv/.style={circle, draw, very thick, minimum size=5.5mm, inner sep=0.5mm}, 
                            node distance=20mm, >=stealth]
            \pgfsetarrows{latex-latex};

            \node[rv] (a) {$a$};
            \node[rv, right =7mm of a, yshift=-9mm] (b) {$b$};
            \node[rv, right =7mm of b, yshift=9mm] (c) {$c$};
            \node[mdot] (mab) at (a.south) [yshift=-7mm,xshift=2.5mm] {};
            \node[mdot] (mc) at (c.south) [yshift=-6mm] {};
             \node[sdot] (s) at (b.north) [yshift=15mm] {};
             \node[rv, right =7mm of c] (d) {$d$};
            \node[mdot] (md) at (d.south) [yshift=-6mm] {};
            \node[sdot, right=2.5mm of c, yshift=5mm](se){};
            \draw[->, very thick,color=red] (md) -- (d);
             
            \draw[-, very thick, color=blue] (a) -- (s);
            \draw[-, very thick, color=blue] (b) -- (s);
            \draw[-, very thick, color=blue] (c) -- (s);
            \draw[->, very thick] (b) -- (a);
            \draw[->, very thick,color=red] (mab) -- (a);
            \draw[->, very thick,color=red] (mab) -- (b);
            \draw[->, very thick,color=red] (mc) -- (c);

            \draw[-, very thick,color=blue] (se) -- (c);
            \draw[-, very thick,color=blue] (se) -- (d);
            \draw[->, very thick, loop left, looseness=10] (a) to (a);
        \end{tikzpicture}
            \begin{minipage}{.1cm}
            \vfill
            \end{minipage}
    \caption{The selected latent projection of (a)}
        \label{fig:example-slpb}
    \end{subfigure}
    \caption{
    Under marginalization of $m_1,m_2,m_3,m_4,m_5$ and selection of 
    $s_1,s_2,s_3,s_4$, we obtain the selected-latent projection shown in (b).
    }
    \label{fig:example-slp}
\end{figure}
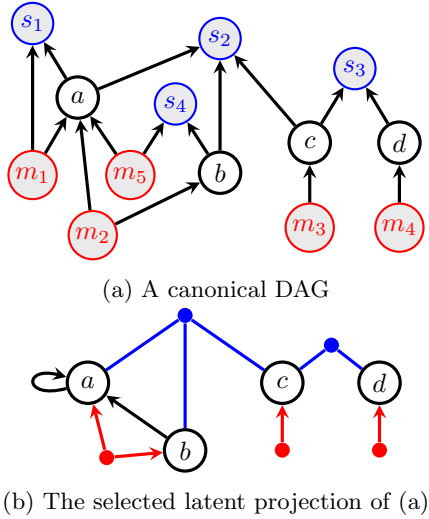

For instance, the selected-latent projection of both the non-canonical DAG of \Cref{fig:teaser-dag-1} and the canonical DAG of
\Cref{fig:example-dag}  (that are related by $\ToBidir$) is the smDG of \Cref{fig:example-slpb}.
It is also possible to define the canonical graph associated with any smDG as follows.
\begin{definition}[Canonical graph associated with an smDG] \label{def_inverse_slp}
The \emph{canonical graph} associated with an smDG $\G = (V,\E,\calL, \calS)$ is the graph $\can(\G)$ obtained by
starting with the graph $(V,\emptyset)$ and, 
\begin{enumerate}
    \item for each edge $a \to b \in \E$ where $a$ is an element of some $V_s \in \calS$ 
    and $b$ is an element of some $V_m \in \calL$,
    add a marginalized vertex $m_{ab}$, a selected vertex $s_{ab}$
    and a path $a \to s_{ab}\gets m_{ab}\to b$;
    \item for each other edge $a \to b \in \E$, add an edge $a \to b$;
\item for each maximal face $V_m\in \calL$, add a marginalized vertex $m$ and edges from $m$ to each vertex in $V_m$;
    \item for each maximal face $V_s$ in $\calS$, add a selected vertex $s$ and edges from each vertex in $V_s$ to $s$.
\end{enumerate}
In the case where $\can(\G)$ is acyclic, we call it the canonical DAG associated with $\G$.
\end{definition}

\subsection{Which smDGs are physically realizable?}
 When an smDG $\G$ has a canonical DAG, i.e., when its canonical graph is acyclic, this intuitively means that $\G$ corresponds to a realistic data-generating process. \blk In the case where the canonical graph \blk $\can(\G)$ is acyclic, the canonical DAG associated with $\calG$ can be mapped back to $\calG$ via selected-latent projection.

\begin{proposition} \label{prop:inverse_slp_inverts_slp}
Any canonical DAG is the canonical graph of its selected-latent projection $\calG$
(up to a renaming of its non-visible variables).
\end{proposition}
\begin{proof}
    The edges of a canonical DAG $\calD$
    can be partitioned into:
a) edges $a \to b$ where $a,b \in V$ and 
either $a \notin \pa{S}$ or $b \notin \ch{M}$ 
(due to $\ToBidir$);
    b) paths $a \to s \gets m \to b$ where $a,b \in V,s\in S,m\in M$;
    c) edges $m \to v$ where $m$ is a { marginalized} variable of a face and $v \in V$;
    d) edges $v \to s$ where $s$ is a selected variable of a face and $v \in V$.
Each edge $a \to b$ in $\calD$ is mapped to $a \to b$ in the smDG $\calG$ by \Cref{def:slp-special-case} and then back to $a \to b$ in $\calD$ by \Cref{def_inverse_slp}.
Each path $a \to s \gets m \to b$ in $\calD$ is mapped to $a \to b$ in $\calG$ and then back to 
$a \to s_{a,b} \gets m_{a,b} \to b$ in $\calD$.
Each {marginalized} variable of a face $m$ with children $V_m$ is mapped to a maximal face in $\calL$ and then back to a 
marginal variable $m'$ with children $V_m$.
Each selected variable of a face $s$ with parents $V_s$ is mapped to a maximal face in $\calS$ 
and then back to a selected variable $s'$ with parents $V_s$.
No other vertices or edges are added by \Cref{def_inverse_slp}, so 
the function thus defined is a left-inverse of the SLP.
\end{proof}

Not every smDG represents a realistic data-generating process however, because the canonical graph is not guaranteed to be acyclic.

For example, in the two vertex smDG $\G:a\to b\to a$, the canonical graph  $\can(\G)$ is equal to $\G$ itself.
 By \Cref{prop:inverse_slp_inverts_slp}, such an smDG cannot be a selected-latent projection for any DAG.

When an smDG does have a canonical DAG, that is, when it can be lifted to a DAG via the process described in \Cref{def_inverse_slp}, we call it \emph{liftable}:

\begin{definition}[Liftable smDG]
    An smDG $\calG$ is said to be \emph{liftable} if its canonical graph is acyclic.
\end{definition}

An smDG will \emph{not} be liftable when its directed edges contain a cycle that has no edge that satisfies condition 1 of \Cref{def_inverse_slp}:

\begin{proposition} \label{prop:acyclic-canonical-dag}
An smDG $\calG=(V,\cal E, L, S)$  is liftable
if and only if for every cycle in $\calG$, at least one edge $a \to b$ in that cycle has 
$a \in \bigcup \calS$
and $b \in \bigcup \calL$. 
\end{proposition}
\begin{proof}
In the canonical graph, marginalized vertices lack parents 
and selected vertices lack children, so any edge transformed by step (1) of \Cref{def_inverse_slp} cannot produce a cycle.
Thus a cycle will arise in the canonical graph precisely when there is a 
cycle in the smDG, and no edge $a \to b$ in that cycle 
satisfies the conditions for step (1), 
i.e.~$a \not \in \bigcup \calL$ or
$b \not \in \bigcup \calS$.
\end{proof}

In principle, to read the conditional independence relations between variables in an smDG, it is possible 
to transform it into a canonical DAG and then use the D-separation criterion of \citet{geiger1990identifying}.
Unlike the lowercase d-separation criterion, this criterion is specifically designed to deduce independence relations
where particular variables are known to be a deterministic function of their parents.
It will sometimes be more convenient, however, to apply a criterion directly to an smDG, 
just as the m-separation criterion is convenient for evaluating mDAGs.
As such, in \Cref{sec:sm-sep}, we have defined sm-separation, a criterion that
extends m-separation and is a sound and complete for evaluating independence relations 
in the selected-latent projection of any canonical DAG.

\section{Interventional inequivalence of smDGs} \label{sec_smDG_captures_everything}
\subsection{Distinct smDGs are Observe\&Do distinguishable} \label{subsec_smDG_obs_and_do_distinguishable}

So far, we have established that DAGs with the same selected-latent projection cannot be 
distinguished by the Observe\&Do probing scheme.
We now prove that all DAGs with different selected-latent projections are Observe\&Do distinguishable.
These results jointly imply that the set of liftable smDGs
(as can be identified by \Cref{prop:acyclic-canonical-dag}) are isomorphic 
to the equivalence classes of DAGs that share the same SMI model.

That DAGs with different mDAGs are Observe\&Do distinguishable was established by
 \citet{ansanelli2024everything}. 
 To generalize this result to smDGs, we will need to consider not just how to distinguish 
 a relationship $a \to b$ with a relationship where $a$ and $b$ 
 are in the same marginalized face, but also how to 
 distinguish each of these from a scenario where $a$ and $b$ are in the same selected face.

\begin{restatable}[Same smDG $\Leftrightarrow$ Same SMI model]{theorem}{smgraphSMIEquivalence} \label{thm:diff-slp-implies-diff-model}
The DAGs $\calD$ and $\calD'$ with vertices partitioned as $(V,M,S)$ and $(V, M', S')$ respectively 
are Observe\&Do  equivalent, i.e.~$\calC(\D,V \mid S)=\calC(\calD',V \mid S')$, if and only if they have the same selected-latent projection, i.e.~$\slp(\calD,V \mid S)=\slp(\calD',V\mid S')$.
\end{restatable}

The proof is detailed in \Cref{app_diff_slp_diff_model} but we summarize our approach here. A pair of SLPs can differ in four ways:
(1) different self-loops (i.e.~edges from a vertex to itself), 
(2) other differences in their directed structures,
(3) different marginalized independence systems,
(4) different selected independence systems.
For case (1), we show that if an SLP contains a self-loop $(v,v)$, i.e.~if the canonical DAG contains a path $v\to s \gets m\to v$, and we ignore all of the other variables, say by setting them to constant values,
and we simultaneously observe $X_{v^\flat}$ and intervene on $X_{v^\sharp}$, 
then the value of $X_{v^\flat}$ might depend on the value of $X_{v^\sharp}$.\footnote{\label{footnote_sec6} This might seem paradoxical at first sight, but here is a scenario where this could occur: suppose that a person's income ($m$) affects both their likelihood of having a phone ($s$) and of having a bank account ($v$). Further imagine that having a bank account is helpful in getting a phone plan, so it increases one's likelihood having a phone. The causal structure of such a scenario is given by $v\gets m \to s$ and $v \to s$. An experimenter asks what fraction of the population has a bank account, via a phone survey, so it is subjected to selection effects on $s$. Suppose that before conducting this survey, the experimentalist opens a bank account for some of the subjects ($v^\sharp$). Some low-income people will now gain access to a phone, and will therefore appear in the survey data. So, if the survey asks ``Did you have a bank account \emph{before} the experiment started?" ($v^\flat$), we will see a higher percentage of the answer ``no" than what we would have seen without the intervention---something that cannot happen if
the self-loop is not present.} 
For case (2), we show that if an smDG contains a directed edge $a\to b$, then 
$X_{b}$ can statistically depend on $X_a$ when $X_a$ is intervened even while 
the distribution of $X_a$ is unchanged by an intervention to $X_{b}$.
For case (3), we show that only with a marginalized face can we see
a conditional dependence between a set of variables that disappears
whenever all but one of the variables is intervened upon.
Finally, for case (4), we show that only with a selected face
we can see a conditional dependence between a set of variables
that persists whenever all but one of its variables is intervened upon.

Note that the fact that two smDGs are always Observe\&Do distinguishable \emph{does not} imply that, starting from a specific set of Observe\&Do data (i.e.~a set $P_*$ of SMI pairs), one can always identify what was the smDG that generated it. This is because, as we stressed at the end of Section~\ref{sec:setup}, the fact that two smDGs are always Observe\&Do distinguishable only implies that the SMI models of two different smDGs are \emph{never the same}, and not that they never overlap.

\subsection{Distinct smDGs are not always Observe-Or-Do distinguishable} \label{subsec_smDG_obs_or_do_distinguishable}
So far, we have chosen the most powerful probing scheme possible to make sure the smDG structure captures all  important differences between causal structures.
Under weaker probing schemes, however, some pairs of smDGs 
may be indistinguishable.

One kind of weaker probing scheme is the Observe-Or-Do scheme \citep{ansanelli2024everything}, where
each visible variable is allowed to be observed or intervened upon, but 
not both at once. We provide a formal definition of the Observe-Or-Do probing scheme in \cref{app_obs_or_do}. Under access to this probing scheme only, the pair of graphs in \Cref{fig_obs_or_do} cannot be distinguished even though 
they have distinct smDGs.

\begin{figure}[h!]
\centering

    \begin{subfigure}{0.42\textwidth}
    \centering
    \begin{tikzpicture}[>=stealth]
        \node[rv] (a) at (0,0) {$v$};
        \node[mv] (m) at (7mm, -7mm) {$m$};
        \node[sv] (s) at (7mm, 7mm) {$s$};

        \draw[->, very thick] (m) -- (a);
        \draw[->, very thick] (a) -- (s);
        \draw[->, very thick] (m) -- (s);
    \end{tikzpicture}
    \caption{A DAG where $v$ has a self-loop} \label{fig_obs_or_do1}
    \end{subfigure}
    \hspace{5mm}
    \begin{subfigure}{0.42\textwidth}
    \centering
    \begin{tikzpicture}[>=stealth]
        \node[rv] (a) at (0,0) {$v$};
        \node[mv] (m) at (7mm, -7mm) {$m$};
        \node[sv] (s) at (7mm, 7mm) {$s$};

        \draw[->, very thick] (m) -- (a);
        \draw[->, very thick] (a) -- (s);
    \end{tikzpicture}
    \caption{A DAG where $m \to s$ is removed} \label{fig_obs_or_do2}
    \end{subfigure}

\caption{Two canonical DAGs whose selected-latent projections differ in the 
self-loop $v \to v$.}
\label{fig_obs_or_do}
\end{figure}
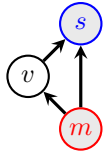
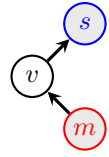

The selected-latent projections of
\Cref{fig_obs_or_do1} and \Cref{fig_obs_or_do2} have identical marginal and selected 
independence systems (which just contain the empty set and the set $\{v\}$).
The SLPs differ, however, in their directed structure:
in the SLP of \Cref{fig_obs_or_do1}, the directed structure contains the self-loop $v \to v$, 
which is not present in the SLP of \Cref{fig_obs_or_do2}.
In {\Cref{prop_app_figs} of}  \blk\Cref{app_obs_or_do}, we establish that these DAGs \emph{cannot} be distinguished by the Observe-or-Do probing scheme.
This means that a pair of DAGs is sometimes distinguishable by observing the intervened and non-intervened distributions jointly, even if they 
cannot be distinguished by observing the same distributions separately.

Note that this is an important difference from the case without selection: when two causal structures without selected vertices are associated with different mDAGs, they are distinguishable even by the weaker Observe-or-1Do probing scheme, where a variable cannot be simultaneously observed and intervened upon \emph{and} we are restricted to hard interventions that set variables to the value $0$ \citep{ansanelli2024everything}.

\section{Observational equivalence between smDGs} \label{sec:observational-equivalence}

Up to this point, we have studied how causal structures can be distinguished when one has access to interventions. In particular, our main focus was the very powerful Observe\&Do probing scheme (in all of the paper except for~\Cref{subsec_smDG_obs_or_do_distinguishable}). In this section, we go to the other extreme, investigating the weakest probing scheme: passive observation of the visible variables. This probing scheme is of particular relevance, since many times it is not possible to implement interventions. 

While different liftable graphs can always be distinguished by the Observe\&Do probing scheme, there are smDGs that are observationally equivalent. As such, it is useful to define observational equivalence and dominance between liftable smDGs. This is defined via their canonical DAGs:

\begin{definition}[Observational Equivalence of smDGs]
Let $\calG$ and $\calG'$ be two liftable smDGs that have the same set of vertices $V$. We say that $\calG$ and $\calG'$ are \emph{observationally equivalent} if their canonical DAGs are observationally equivalent; that is, if $\calM(\can(\calG),V\cmid S)=\calM(\can(\calG'),V\cmid S')$, where $S$ is the set of selected vertices in $\can(\calG)$ and  $S'$ is the set of selected vertices in $\can(\calG')$. 
\end{definition}

The problem of establishing observational equivalence classes of causal structures is complicated, 
and it is not fully solved even in the case without selection effects. In fact, at the level of four visible vertices there are already mDAGs for which it is unknown whether they are observationally equivalent \citep{ansanelli2025observational}.

Therefore, in this section we do not aim to completely solve this problem for the case with selection effects. Instead, we will simply present a list of rules that give sufficient---but not necessary---conditions for two smDGs to be observationally equivalent.

\subsection{Adding marginalized faces}
\label{sec_added_marg_node}

Our first rule for proving observational equivalence (but not interventional equivalence) gives conditions under which a marginalized face can be \emph{added} without changing the selected-marginal observational model. In many cases, this added marginalized face will enable the application of our subsequent rule, \Cref{prop:smDG-split}, where this would not be otherwise possible. 

\begin{restatable}{proposition}{SelectedToMarginalObsEq} \label{prop:selected-to-marginal}
Let $\calG=(V,\cal E, L, S)$ be a liftable smDG
that includes a maximal selected face $V_s \in \calS$. 
Assume that:
\begin{itemize}
\item $V_s$ includes all its parents, i.e.~$\pa{\calG}{V_s}\subseteq V_s$;
\item every vertex $v\in V_s$ belongs to some marginalized face of $\calL$; 
and 
\item every marginalized face $V_m\in \calL$ that intersects with $V_s$ is subsumed by $V_s$, 
i.e.~$V_s\cap V_m\neq \emptyset \implies V_m\subseteq V_s$. 
\end{itemize}
Let $\calG'=(V,\cal E, L', S)$ be the smDG constructed from $\calG$ by adding 
$\cl(V_s)$ to the marginalized independence system, 
i.e.~$\calL'=\calL\cup\cl(V_s)$. 
Then, $\calG$ and $\calG'$ are observationally equivalent.
\end{restatable}

\Cref{fig_running_obs_ex1} gives an example of the application of \Cref{prop:selected-to-marginal}. The idea behind this proposition is as follows: a selected vertex $s$ can induce correlation between its ancestors, which is mediated by the parents of $s$. For example, the DAG $a\to b \to s \gets c$ can explain correlations between $X_a$ and $X_c$ if we condition on $X_s=0$, but \emph{not} if we further condition on $X_b$. In the special case where the ancestors of a selected vertex $s$ coincide with its parents, the presence of $s$ can induce \emph{any} correlation between its ancestors, without any mediation. That is, it is at least as powerful as a marginalized vertex $m$ pointing to all of the ancestors of $s$. The proof of this proposition is supplied in \Cref{app_selected_to_marginal}.

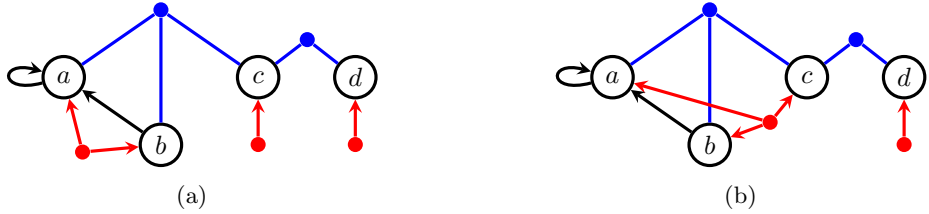
\begin{figure}[ht]
    \centering
    \begin{subfigure}{.5\textwidth}\centering
        \centering
        \begin{tikzpicture}[rv/.style={circle, draw, very thick, minimum size=5.5mm, inner sep=0.5mm}, 
                            node distance=20mm, >=stealth]
            \pgfsetarrows{latex-latex};

            \node[rv] (a) {$a$};
            \node[rv, right =7mm of a, yshift=-9mm] (b) {$b$};
            \node[rv, right =7mm of b, yshift=9mm] (c) {$c$};
            \node[mdot] (mab) at (a.south) [yshift=-7mm,xshift=2.5mm] {};
            \node[mdot] (mc) at (c.south) [yshift=-6mm] {};
             \node[sdot] (s) at (b.north) [yshift=15mm] {};
             \node[rv, right =7mm of c] (d) {$d$};
            \node[mdot] (md) at (d.south) [yshift=-6mm] {};
            \node[sdot, right=2.5mm of c, yshift=5mm](se){};
            \draw[->, very thick,color=red] (md) -- (d);
             
            \draw[-, very thick, color=blue] (a) -- (s);
            \draw[-, very thick, color=blue] (b) -- (s);
            \draw[-, very thick, color=blue] (c) -- (s);
            \draw[->, very thick] (b) -- (a);
            \draw[->, very thick,color=red] (mab) -- (a);
            \draw[->, very thick,color=red] (mab) -- (b);
            \draw[->, very thick,color=red] (mc) -- (c);
            \draw[-, very thick,color=blue] (se) -- (c);
            \draw[-, very thick,color=blue] (se) -- (d);
            \draw[->, very thick, loop left, looseness=10] (a) to (a);
        \end{tikzpicture}
            \begin{minipage}{.1cm}
            \vfill
            \end{minipage}
            \caption{}
        \label{fig_running_obs_ex1a}
    \end{subfigure}
        \begin{subfigure}{0.45\textwidth}\centering
        \begin{tikzpicture}[rv/.style={circle, draw, very thick, minimum size=5.5mm, inner sep=0.5mm}, 
                            node distance=20mm, >=stealth]
            \pgfsetarrows{latex-latex};
            \node[rv] (a) {$a$};
            \node[rv, right =7mm of a, yshift=-9mm] (b) {$b$};
            \node[rv, right =7mm of b, yshift=9mm] (c) {$c$};
            \node[mdot, right =4mm of b, yshift=3mm] (m) {};
             \node[sdot] (s) at (b.north) [yshift=15mm] {};
             \node[rv, right =7mm of c] (d) {$d$};
            \draw[-, very thick, color=blue] (a) -- (s);
            \draw[-, very thick, color=blue] (b) -- (s);
            \draw[-, very thick, color=blue] (c) -- (s);
            \draw[->, very thick, color=red] (m) -- (a);
            \draw[->, very thick,color=red] (m) -- (b);
            \draw[->, very thick,color=red] (m) -- (c);
            \node[sdot, right=2.5mm of c, yshift=5mm](se){};
            \draw[-, very thick,color=blue] (se) -- (c);
            \draw[-, very thick,color=blue] (se) -- (d);
             \node[mdot] (md) at (d.south) [yshift=-6mm] {};
            \draw[->, very thick,color=red] (md) -- (d);
            \draw[->, very thick] (b) -- (a);
            \draw[->, very thick, loop left, looseness=10] (a) to (a);
        \end{tikzpicture}
            \begin{minipage}{.1cm}
            \vfill
            \end{minipage}
            \caption{}
        \label{fig_running_obs_ex1b}
    \end{subfigure}
    \caption{Two smDGs that can be shown observationally equivalent by \Cref{prop:selected-to-marginal}.}
    \label{fig_running_obs_ex1}
\end{figure}

\subsection{Removing special edges}
\label{sec_remove_special_obs}

Our second rule for proving observational equivalence will pertain to the removal of special edges. 
Specifically, whenever the start and endpoint of a special edge 
are in the same marginal face, 
it will be possible to sever that edge while 
preserving observational equivalence.

\begin{restatable}{proposition}{ObsRemoveSpecialEdges}
 \label{prop:special-edge-in-district-removal}
 Let $\calG=(V,\cal E, L, S)$ be a liftable smDG. Suppose that $\calG$ contains an edge $a\to b$ such that $a$ is an element of some $V_s\in \mathcal{S}$ and $b$ is an element of some $V_m\in \mathcal{L}$. Further suppose that $a$ and $b$ belong to the same face of the marginalized independence system, that is, that $a$ is also an element of $V_m$. Then, $\calG$ is observationally equivalent to the smDG $\calG'$ obtained by starting from $\calG$ and removing the edge $a\to b$.
\end{restatable}

The proof is provided in \Cref{app_proof_special_edge_obs_removal}. All of the edges $a\to b$ described correspond to special edges $a\to s\gets m\to b$ in the canonical DAG, where $s$ is selected and $m$ is marginalized. The intuition is that if $a$ and $b$ share a common marginalized parent $m'$, selecting on $s$ can establish (through $m'$) all of the correlation between $a$ and $b$ (and the correlation between $b$ and the ancestors of $a$) that could have been established through the special edge.

From this, we can extract the following corollary:
\begin{restatable}{corollary}{SelfLoopSMOEquivalent}\label{prop:self-loop-smo-equivalent}
Let $\calG$ be a liftable smDG containing the self-loop $a\to a$. Let  $\calG'$ be the smDG constructed by removing this self-loop. Then, $\calG$ and $\calG'$ are observationally equivalent.
\end{restatable}
\begin{proof}
    In any liftable smDG with a self-loop $a\to a$, the vertex $a$ is included in both a marginalized face and a selected face. Therefore, this is the special case of \Cref{prop:special-edge-in-district-removal} where $a$ is the same vertex as $b$.
\end{proof}

We showed in \Cref{subsec_smDG_obs_or_do_distinguishable} a particular smDG where the self-loop $v \to v$
was undetectable by the Observe-Or-Do probing scheme. \Cref{prop:self-loop-smo-equivalent} now establishes that at least as far as passive observation goes, 
such self-loops can \emph{never} be detected.

\blk

With this proposition, we can show for example that the smDG of \Cref{fig_running_obs_ex1b} is observationally equivalent to the smDG of \Cref{fig_running_obs_ex2}.

\begin{figure}[ht]
    \centering
        \begin{tikzpicture}[rv/.style={circle, draw, very thick, minimum size=5.5mm, inner sep=0.5mm}, 
                            node distance=20mm, >=stealth]
            \pgfsetarrows{latex-latex};
            \node[rv] (a) {$a$};
            \node[rv, right =7mm of a, yshift=-9mm] (b) {$b$};
            \node[rv, right =7mm of b, yshift=9mm] (c) {$c$};
            \node[sdot, right=2.5mm of c, yshift=5mm](se){};
            \node[mdot, right =4mm of b, yshift=3mm] (m) {};
             \node[sdot] (s) at (b.north) [yshift=15mm] {};
             \node[rv, right =7mm of c] (d) {$d$};
            \draw[-, very thick, color=blue] (a) -- (s);
            \draw[-, very thick, color=blue] (b) -- (s);
            \draw[-, very thick, color=blue] (c) -- (s);
            \draw[->, very thick, color=red] (m) -- (a);
            \draw[->, very thick,color=red] (m) -- (b);
            \draw[->, very thick,color=red] (m) -- (c);
            \draw[-, very thick,color=blue] (se) -- (c);
            \draw[-, very thick,color=blue] (se) -- (d);
               \node[mdot] (md) at (d.south) [yshift=-6mm] {};
            \draw[->, very thick,color=red] (md) -- (d);
        \end{tikzpicture}
    \caption{An smDG that can be shown observationally equivalent to \Cref{fig_running_obs_ex1b} by \Cref{prop:special-edge-in-district-removal} (and \Cref{prop:self-loop-smo-equivalent}).}
    \label{fig_running_obs_ex2}
\end{figure}
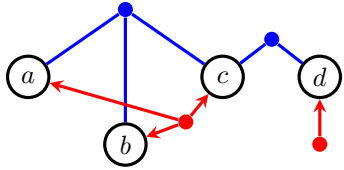

\subsection{Removing selected faces} \label{subsec:remove-selected}
We say that a collider $a\to c \leftarrow b$ is \emph{shielded} if the parents $a$ and $b$ are connected via a directed edge, either $a\to b$ or $b\to a$. In a DAG without marginalized vertices, if a selected vertex has no unshielded colliders among its ancestors, 
that is if ``moralization'' has been performed, then the selected vertex may be removed without altering 
the SMO model \citep[][Chapter 3]{lauritzen1996graphical}. 
In the presence of marginalization, 
however, it is possible to remove a selected vertex in a wider range of circumstances.
Instead of necessarily requiring that the  parents $a,b$ of a collider be connected by a directed edge, it also suffices that they share a common marginalized parent.

In this section, we will formulate a condition on smDGs where, based on the idea described above, a maximal selected face can be removed. Note that this condition is \emph{not} simply that the canonical DAG of the smDG is such that ancestors of the corresponding selected vertex $s$ obey the condition above; that is, include no colliders $a\to c \leftarrow b$ where $a$ and $b$ are not connected via a directed edge nor a marginalized common cause. This is because there are cases where the canonical DAG does not obey this condition, but some non-canonical DAG that corresponds to the same smDG does. For example, the DAG $s\gets a \to s' \gets m' \to b\to s$, where $a$ points to $b$ with a special edge and both $a$ and $b$ are parents of a selected vertex $s$, is canonical. However, in this DAG the vertices $a$ and $b$ are not connected via a directed edge nor share a common marginalized parent, so apparently one cannot remove the vertex $s$. However, if one looks at the non-canonical DAG $s\gets a \to b\to s$, which is interventionally equivalent to $s\gets a \to s' \gets m' \to b\to s$ by \Cref{lemma_interchangeedges}, one concludes that $s$ can be removed.

\begin{restatable}{proposition}{SelectedNodeRemovalObsEqTwo} \label{prop:smDG-split}  
Let $\calG=(V,\cal E, L, S)$ be a liftable smDG, let $V_s\in \calS$ be one of its maximal selected faces, 
and let $\calG_{\an{V_s}}$ be the subgraph of $\calG$ over the set $V_s$ and all of its ancestors. Suppose that:
\begin{enumerate}[label=\alph*)]
   \item There are no cycles in $\calG_{\an{V_s}}$. \label{cond:no-cycles}
    \item In $\calG_{\an{V_s}}$, vertices that    share a child or are in the same selected face  are neighbours or are in the same marginalized face. \label{cond:smDG2}
    \item All of the maximal marginalized faces of $\calL$ that include vertices of $\an{\calG}{V_s}$ are disjoint.
    \label{sublemma:disjoint_districts} \label{cond:smDG3}
    \item  For each maximal marginalized face $V_m\in \calL$ that includes vertices of $\an{\calG}{V_s}$, 
    the parents of $V_m$ that are not in $V_m$ are shared by all the vertices of $V_m$,
    i.e.~$\pa{\calG}{V_m} \setminus V_m \subseteq \bigcap_{v \in V_m}\pa{\calG}{v}$.
    \label{cond:smDG4}
    \end{enumerate}
Construct the independence system $\calS'$ by starting from $\calS$ and then deleting $V_s$ and all of its subsets.
Define the smDG $\calG'=(V,\cal E, L, S')$.
Then, $\calG$ and $\calG'$ are observationally equivalent.
\blk
\end{restatable}

When one applies \Cref{prop:special-edge-in-district-removal} before \Cref{prop:smDG-split}, the condition a), which prohibits cycles in $\calG_{\an{V_s}}$, is equivalent to the weaker requirement that every cycle of $\calG_{\an{V_s}}$ has at least one edge $a\to b$ in the cycle such that $a\in \bigcup \calS$ and both $a$ and $b$ are in the same marginal face, that is, $a,b\in V_m$ for some $V_m\in \cal L$. This is because all such cycles will be broken by \Cref{prop:special-edge-in-district-removal}.

The proof of \Cref{prop:smDG-split} is in \Cref{app_selected_removal2}. The general idea is that a selected vertex establishes correlation between its ancestors, but if all colliders in the ancestry of a selected vertex are connected via a directed edge or a marginalized common cause, then any such correlations can be reproduced without using the selected vertex. The  conditions of \Cref{prop:smDG-split} guarantee that this is the case in \emph{some} DAG that has $\calG$ as its selected-latent projection.

With \Cref{prop:smDG-split}, we can show that the smDG of \Cref{fig_running_obs_ex2} is observationally equivalent to the smDG of \Cref{fig_running_obs_ex3}. Note that here it was important that we had previously applied \Cref{prop:selected-to-marginal} to \Cref{fig_running_obs_ex1a}: if the  marginalized face $\{a,b,c\}$ was not present, the condition b) of  \Cref{prop:smDG-split} would not be satisfied for $V_s=\{a,b,c\}$. Therefore, by making the subgraph $\calG_{\an{\calG}{V_s}}$ saturated, \Cref{prop:selected-to-marginal} enabled the application of \Cref{prop:smDG-split}. Note, however, that  \Cref{prop:smDG-split} can also be applied in some cases where $\calG_{\an{\calG}{V_s}}$ is not saturated; \Cref{fig_ex_unshielded} is an example.

\begin{figure}[ht]
    \centering
        \begin{tikzpicture}[rv/.style={circle, draw, very thick, minimum size=5.5mm, inner sep=0.5mm}, 
                            node distance=20mm, >=stealth]
            \pgfsetarrows{latex-latex};
            \node[rv] (a) {$a$};
            \node[rv, right =7mm of a, yshift=-9mm] (b) {$b$};
            \node[rv, right =7mm of b, yshift=9mm] (c) {$c$};
            \node[mdot, right =4mm of b, yshift=3mm] (m) {};
             \node[rv, right =7mm of c] (d) {$d$};
             \draw[->, very thick, color=red] (m) -- (a);
            \draw[->, very thick,color=red] (m) -- (b);
            \draw[->, very thick,color=red] (m) -- (c);
           \node[sdot, right=2.5mm of c, yshift=5mm](se){};
            \draw[-, very thick,color=blue] (se) -- (c);
            \draw[-, very thick,color=blue] (se) -- (d);
               \node[mdot] (md) at (d.south) [yshift=-6mm] {};
            \draw[->, very thick,color=red] (md) -- (d);
        \end{tikzpicture}
    \caption{An smDG that can be shown observationally equivalent to \Cref{fig_running_obs_ex2} by \Cref{prop:smDG-split}.}    
    \label{fig_running_obs_ex3}
\end{figure}
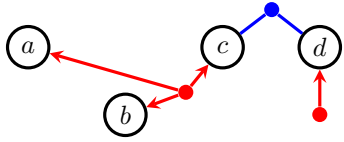

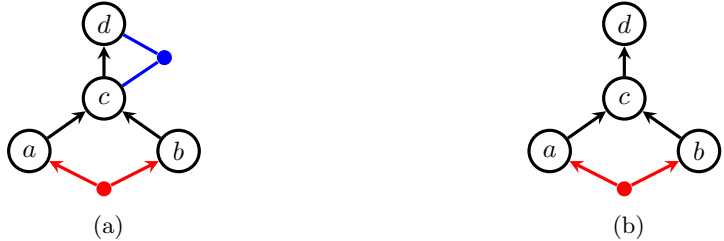
\begin{figure}[h!]
     \begin{subfigure}{0.45\textwidth}\centering
        \begin{tikzpicture}[rv/.style={circle, draw, very thick, minimum size=5.5mm, inner sep=0.5mm}, 
                            node distance=20mm, >=stealth]
            \pgfsetarrows{latex-latex};
            \node[rv] (a) {$a$};
            \node[rv, right =4mm of a, yshift=7mm] (c) {$c$};
            \node[rv, right =4mm of c, yshift=-7mm] (b) {$b$};
            \node[rv, above =4mm of c] (d) {$d$};
            \node[mdot, below =8mm of c] (m) {};
            \node[sdot, right=4mm of d,yshift=-4.5mm](sb){};
             \draw[-, very thick,color=blue] (sb) -- (c);
              \draw[-, very thick,color=blue] (sb) -- (d);
             \draw[->, very thick, color=red] (m) -- (a);
            \draw[->, very thick,color=red] (m) -- (b);
            \draw[->, very thick] (a) -- (c);
            \draw[->, very thick] (b) -- (c);
            \draw[->, very thick] (c) -- (d);
        \end{tikzpicture}
            \begin{minipage}{.1cm}
            \vfill
            \end{minipage}
            \caption{ }
        \end{subfigure}
            \begin{subfigure}{0.45\textwidth}\centering
        \begin{tikzpicture}[rv/.style={circle, draw, very thick, minimum size=5.5mm, inner sep=0.5mm}, 
                            node distance=20mm, >=stealth]
            \pgfsetarrows{latex-latex};
            \node[rv] (a) {$a$};
            \node[rv, right =4mm of a, yshift=7mm] (c) {$c$};
            \node[rv, right =4mm of c, yshift=-7mm] (b) {$b$};
            \node[rv, above =4mm of c] (d) {$d$};
            \node[mdot, below =8mm of c] (m) {};
             \draw[->, very thick, color=red] (m) -- (a);
            \draw[->, very thick,color=red] (m) -- (b);
            \draw[->, very thick] (a) -- (c);
            \draw[->, very thick] (b) -- (c);
            \draw[->, very thick] (c) -- (d);
        \end{tikzpicture}
            \begin{minipage}{.1cm}
            \vfill
            \end{minipage}
            \caption{ }
        \end{subfigure}
    \caption{Two smDGs that can be shown observationally equivalent by \Cref{prop:smDG-split}, for $V_s=\{c,d\}$. This is an example where  \Cref{prop:smDG-split} can be applied even though $\calG_{\an{\calG}{V_s}}$ is not saturated.}
    \label{fig_ex_unshielded}
\end{figure}

\subsection{Leveraging mDAG equivalence for smDG equivalence} \label{subsec:leveraging-mdag}

Now, we present our last technique for showing observational equivalence between smDGs: leveraging the known rules to prove observational equivalence between mDAGs. To do so, we will take the canonical DAGs of our smDGs, and treat their selected vertices as visible. If the resulting DAGs are observationally equivalent, then so are 
the original smDGs.

\begin{proposition} \label{prop:equivalence}
Let $\calG=(V,\cal E, L, S)$ and  $\calG'=(V,\cal E', L', S)$  be liftable smDGs without deterministic vertices; that is, where every vertex is included in some marginalized face\footnote{
We require each vertex to have marginalized parents because 
visible variables in an mDAG are never required to 
be a deterministic function of their visible parents, 
and so the mDAG equivalence rules as currently stated can only be applied in this case.
As soon as the mDAG is generalized to have deterministic variables, 
and there are rules for their equivalence, 
then \Cref{prop:equivalence} can be trivially extended
to accommodate their equivalences.}.
Construct the mDAGs $\calG_\text{mDAG}=(\tilde{V},\tilde{\E}, \calL, \emptyset)$ and $\calG'_\text{mDAG}=(\tilde{V},\tilde{\E'}, \calL', \emptyset)$ by transforming every selected vertex of $\can(\calG)$ and  $\can(\calG')$ into a visible vertex, and taking the latent projection thereof. If the mDAGs $\calG_\text{mDAG}$ and $\calG'_\text{mDAG}$ are observationally equivalent, then the smDGs $\calG$ and $\calG'$ are also observationally equivalent.

\end{proposition}

\begin{proof}
   Since  $\calG_\text{mDAG}$ and $\calG'_\text{mDAG}$ are observationally equivalent,  $\can(\calG_\text{mDAG})$ and $\can(\calG'_\text{mDAG})$  have the same marginal observational model. The selected-marginal observational model of $\can(\calG)$ is obtained by conditioning the distributions of the marginal observational model of $\can(\calG_\text{mDAG})$ on an assignment of values to the selected variables. The same is true for  $\can(\calG')$ and $\can(\calG'_\text{mDAG})$.  Since the set of selected vertices of  $\can(\calG)$ and $\can(\calG')$ is the same, their selected-marginal observational models are also equal.
\end{proof}

For example, this proposition together with the rule for proving observational equivalence of mDAGs presented in Proposition 5 of \citet{evans2016graphs} can show the observational equivalence between the smDGs of \Cref{fig_running_obs_ex3}  and \Cref{fig_running_obs_ex5}. 

The steps shown in this section answer the question posed in the example of \Cref{fig:teaser-dags}: since the SLP of \Cref{fig:teaser-dag-1} is \Cref{fig_running_obs_ex1} and the SLP of \Cref{fig:teaser-dag-2} is \Cref{fig_running_obs_ex5}, we learn that those two DAGs are observationally equivalent. (Since they correspond to different smDGs, we also know that they are interventionally inequivalent.)

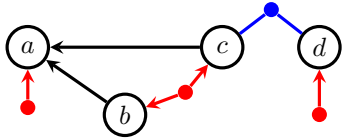
\begin{figure}[ht]
    \centering
        \centering
        \begin{tikzpicture}[rv/.style={circle, draw, very thick, minimum size=5.5mm, inner sep=0.5mm}, 
                            node distance=20mm, >=stealth]
            \pgfsetarrows{latex-latex};
            \node[rv] (a) {$a$};
            \node[rv, right =7mm of a, yshift=-9mm] (b) {$b$};
            \node[rv, right =7mm of b, yshift=9mm] (c) {$c$};
            \node[mdot, right =4mm of b, yshift=3mm] (m) {};
            \node[sdot, right=2.5mm of c, yshift=5mm](se){};
             \node[rv, right =7mm of c] (d) {$d$};
              \node[mdot, below =4mm of a] (ma) {};
             \draw[->, very thick, color=red] (ma) -- (a);
            \draw[->, very thick,color=red] (m) -- (b);
            \draw[->, very thick,color=red] (m) -- (c);
            \draw[-, very thick,color=blue] (se) -- (c);
            \draw[-, very thick,color=blue] (se) -- (d);
            \draw[->, very thick] (b) -- (a);
            \draw[->, very thick] (c) -- (a);
               \node[mdot] (md) at (d.south) [yshift=-6mm] {};
            \draw[->, very thick,color=red] (md) -- (d);
        \end{tikzpicture}
            \begin{minipage}{.1cm}
            \vfill
            \end{minipage}
    \caption{An smDG that can be shown to be observationally equivalent to \Cref{fig_running_obs_ex3} by \Cref{prop:equivalence} together with Proposition 5 of \citet{evans2016graphs}. That is, we treat all of the selected vertices of the canonical DAG as visible, apply Proposition 5 and then take the selected-latent projection.}
    \label{fig_running_obs_ex5}
\end{figure}

The complete set of currently known rules for proving observational equivalence between mDAGs is compiled in \citet{ansanelli2025observational}. Via \Cref{prop:equivalence}, all of these rules may be employed to establish observational equivalence between smDGs.

\section{Discussion and conclusion} \label{sec:discuss-and-conclude}
The first contribution of our work is to completely characterize all of the information about the causal structure that can be obtained from the Observe\&Do interventional probing scheme, which is as powerful as possible, in the presence of selection and marginalization.
We have established that the SMI model is invariant 
to several kinds of graphical transformations (\cref{sec:invariances}), which can involve 
merging and then splitting non-visible components of the graph.
This series of changes maps an arbitrary  DAG to a canonical DAG, 
and ultimately, to the smDG, a cyclic graph with marginal and selected independence systems,
which more naturally represents the structure under marginalization and selection.
We have further established that two DAGs have an identical selected marginal interventional model 
if and only if their smDGs are the same. 

Our second contribution is to investigate weaker probing schemes.
In the Observe-Or-Do probing scheme, defined in \cref{app_obs_or_do}, a pair of distinct smDGs are sometimes indistinguishable.
For the case where only passive observation is allowed, we have presented further rules for removing or replacing 
selected variables, and have discussed how knwon rules  to show observational equivalence between mDAGs
may be applied to the case of smDGs.

We now highlight five possible extensions of the present work.

Firstly, the selected-latent projection is complete for Observe\&Do equivalence, but we still do not have a complete characterization of observational equivalence between smDGs. 
 In fact, even in the case without selection the observational equivalence classes are not completely characterized for mDAGs with four or more visible variables~\citep{ansanelli2025observational}.
In Section~\ref{sec:observational-equivalence} we have presented some rules for turning smDGs into observationally equivalent mDAGs, 
such as removing self-loops, and removing selected faces in some cases.
However, we expect that significant further work is required, firstly to exhaust the 
cases where an smDG can be transformed into an observationally equivalent mDAG, and 
secondly to exhaust the equivalences between smDGs that have no observationally 
equivalent mDAG at all.

Secondly, we also do not have a complete characterization of equivalence for the case of weaker interventional probing schemes, such as for the Observe-or-Do probing scheme. This is unlike the case without selection, where both Observe\&Do equivalence and Observe-or-Do equivalence is completely characterized by the mDAG structure~\citep{ansanelli2024everything}. 

Thirdly, we have treated the case where a variable is \emph{selected} to a single value 
so that under selection, their children would be unaffected by them, which is why 
we were able to remove outgoing edges from selected vertices without altering the selected model.
More generally, we could consider scenarios where variables are \emph{conditioned} to a 
set of values, and where outgoing edges from selected variables would have to be considered.

Fourthly, we focused on the causal models of liftable smDGs, regarding non-liftable 
smDGs as not physically realizable because they yield a cyclic canonical graph.
Physical or not, however, there exist formal models for cyclic causal graphs 
\citep{bongers2021foundations}, 
which could be used to assign interventional distributions to cyclic canonical graphs,
and thus to non-liftable smDGs.

Fifthly, although our Observe\&Do interventional scheme is highly general, 
it excludes the edge interventions of \citet{shpitser2016causal}, 
and further work could seek to evaluate the circumstances in which edge interventions
can distinguish a pair of DAGs.

\section*{Acknowledgements}

MMA thanks Robert Spekkens for conceptual discussions regarding interventional probing schemes in the presence of selection effects. RC thanks Open Philanthropy for supporting this work via the
Future of Humanity Institute scholarship program.
MMA and EW were supported by the Perimeter Institute for Theoretical Physics. Research at Perimeter Institute is supported in part by the Government of Canada through the Department of Innovation, Science and Economic Development and by the Province of Ontario through the Ministry of Colleges and Universities. MMA is also supported by the Natural Sciences and Engineering Research Council of Canada (Grant No. RGPIN-2024-04419).

\appendix\clearpage

\section{Background on Regular Conditional Probabilities} \label{app:regular-conditional-probabilities}

A \emph{regular conditional probability} is a way to define conditional probabilities given a $\sigma$-algebra or random variable, even when conditioning on events of probability zero.  
Formally, let $(\Omega,\mathcal{F},\mathbb{P})$ be a probability space and let $\mathcal{H}\subseteq\mathcal{F}$ be a sub-$\sigma$–algebra.  
A \emph{regular conditional probability (RCP)} of $\mathbb{P}$ given $\mathcal{H}$ is a family of probability measures
\[
\bigl\{\mathbb{P}(\,\cdot\mid\mathcal{H})(\omega)\bigr\}_{\omega\in\Omega}
\]
such that for every $A\in\mathcal{F}$ the map $\omega\mapsto\mathbb{P}(A\mid\mathcal{H})(\omega)$ is $\mathcal{H}$-measurable and satisfies
\[
\int_{H}\mathbb{P}(A\mid\mathcal{H})(\omega)\,d\mathbb{P}(\omega)=\mathbb{P}(A\cap H)\qquad\text{for all }H\in\mathcal{H}.
\]
When $\mathcal{H}$ is the $\sigma$–algebra generated by a random variable $Y$, writing $\mathbb{P}(A\mid Y=y)$ for the kernel $\omega\mapsto\mathbb{P}(A\mid\mathcal{H})(\omega)$ gives the familiar \emph{conditional distribution of $A$ given $Y=y$}.  
Thus ordinary conditional probabilities like $\mathbb{P}(X\in B\mid Y=y)$ are special cases of RCPs, while the general definition works even when conditioning on probability-zero events or on $\sigma$-algebras not generated by countable partitions.

Regular conditional probabilities are only needed for 
our results to generalize to continuous random variables.
To see where we rely on them, one can see the proof of
\Cref{lemma_interchangeedges},
where we simulate a causal relationship $a \to b$
using a path $a \to s \gets m \to b$ by letting 
$X_m$ be sampled randomly over the state space of $a$,
and then letting $X_s=0$ only when $a=m$.
In such cases, the probability of $a=m$ may be zero, 
but
the use of regular conditional probability ensures that 
given $X_s=0$,
$b$ depends on $a$ just as it would in the graph $a \to b$.

\section{Proofs that the transformations of Section~\ref{sec:invariances} maintain Observe\&Do equivalence}

\subsection{Proof of \Cref{prop_exog_term}} \label{app_exog_term}

Before establishing interventional \emph{equivalence}, we begin with a useful lemma that shows interventional \emph{dominance}.

\begin{lemma}[Interventional dominance of subgraphs] \label{le:subgraph}
If $\calD$ is a DAG with vertices partitioned as $(V,M,S)$ and $\calD'$ is a DAG with vertices
partitioned as $(V,M',S')$, 
where $\calD$ is a subgraph of $\calD'$, and $M \subseteq M'$, $S \subseteq S'$, 
then $\calD'$ Observe\&Do dominates $\calD$. That is, $\calC(\calD, V \mid S) \subseteq \calC(\calD', V\mid S')$.
\end{lemma}
\begin{proof}
Let $P_*$ be a set of pairs in the SMI model of $\calD$ produced by 
some kernels $P(X_y \mid X_{\pa{y}})$.
Consider the same kernels in $\calD'$, but set the additional 
variables $(M' \setminus M)$ and $(S' \setminus S)$ to constant values.
Since intervention is only possible on visible variables
and the kernels are equal,
we will recover $P_*$, meaning that $P_* \in \calC(\calD', V\mid S')$.
This is true for every $P_* \in \calC(\calD, V \mid S)$ so
$\calC(\calD, V \mid S) \subseteq \calC(\calD', V\mid S')$.
\end{proof}

We now restate Proposition \ref{prop_exog_term}.

\ExogAndTerm*

\begin{proof}
    \textbf{Proof of a)}
    This result would be immediate from 
    Theorem 1 of \citet{ansanelli2024everything}
except that our smDGs have deterministic 
variables, so we instead construct a proof from first principles.

As shorthand, we will denote:
\begin{itemize}
\item the exogenized graph by
$\tilde{\calD}:=\exogenize_m(\calD)$,
\item the parents in $\calD$ of a vertex $z$
by $R_z:=\pa{\calD}{z}$,
\item the sharp parents in $\calD$ of a vertex $z$
by $R^\sharp_z:=\pas{\calD}{z}$,
\item the parents in $\tilde{\calD}$ by
$\tilde{R}_z:=\pa{\tilde{\calD}}{z}$.
\end{itemize}

\ding{212} \textbf{Proof of $\calC(\calD,V \mid S) \subseteq \calC(\tilde{\calD},V \mid S)$:}

Let $P_*\in \calC(\calD, V \mid S)$ be a set of SMI pairs compatible with $\calD$, and let $P(X_{y} \mid X_{\RR_y})$ for $y\in V\dotcup M\dotcup S$ be the kernels associated with $P_*$. In  $\tilde{\calD}$, the vertex $m$ does not have parents. We will define kernels in  $\tilde{\calD}$ in such a way that the variable $X_m$ encodes a library of values, one for each assignment to the variables $X_{\pa{\calD}{m}}$, and the children of $m$ look only at the element of this library that correspond to the actual value of $X_{\RR_m}$:
\begin{itemize}
    \item $P'(X_m=\bm{x}_m) = 
    \prod_{\alpha \in \dom{\RR_m}} P(X_m =x^{\alpha}_m \mid X_{\RR_m}=\alpha)$
    , where each $\bm{x}_m=\times_{\alpha \in \dom{\RR_m}}x^{\alpha}_m$. That is,  $X_m$ is a Cartesian product of variables, one for each assignment $\alpha$ to the parents $R_m$ of $m$,
    where $P'(X_m^{\alpha}) = P(X_m \mid X_{R_m}=\alpha)$. 
    \item $P'(X_w \mid \sx_m, X_{\RR_w\setminus m}) = P(X_w\mid X_m=x^{\alpha}_m,x_{\RR_w \setminus m})$ for each 
    $w \!\in\! \ch{\calD}{m}$ and $\sx_m$
    \item $P'(X_w\mid X_{\RR_w}) = P(X_w\mid X_{\RR_w})$ otherwise.
\end{itemize}

Let $P'_*\in \calC(\tilde{\calD}, V \mid S)$ be the set of SMI pairs realizable by $\tilde{\calD}$ with the choice of kernels above, and $Y:=(V^\flat \dotcup M \dotcup S)\setminus(\{m\}\cup \ch{\calD}{m})$.   Every pair $\big(Q(X_{V^\sharp}),\,P'_{Q(X_{V^\sharp})}(X_{V^\sharp\dotcup V^\flat}\cmid X_s=0)\big) \in P'_*$ must obey \eqref{eq_Markov_interventional}. Therefore, the joint distribution is:
	\begin{align*}
	&P'_{Q(X_{V^\sharp})}(x_{V^\sharp\dotcup V^\flat},x_M,x_S)\\
    &
    =Q(x_{V^\sharp}) 
    \displaystyle\prod_{\alpha \in \dom{{\RR_m}}}\hspace{-1mm} P(X_m =x^{\alpha}_m \mid X_{\RR^\sharp_m}=\alpha)  \displaystyle\prod_{w\in\ch{m}}  P\big(x_w\mid X_m=x^{\alpha}_m,x_{R_w \setminus m}\big)
    \displaystyle\prod_{y \in Y} P(x_y \mid x_{R^\sharp_y})
	\end{align*}

After marginalizing over $\bm{x}_m$, 
the only term that does not go to $1$ is the term corresponding to the actual value of $x_{\RR^\sharp_m}$. That is:
	\begin{align*}
	&P'_{Q(X_{V^\sharp})}(x_{V^\sharp\dotcup V^\flat},x_{M\setminus\{m\}},x_S)\\
    &=Q(x_{V^\sharp})\displaystyle
    \!\int_{x_{\RR_m} \in \dom{{\RR_m}}}\hspace{-9mm}
    P\big(X_m =x^{x_{\RR_m}}_m \,\big|\, X_{\RR^\sharp_m}=x_{\RR_m}\big) 
    \!\!\!\!\displaystyle\prod_{w\in\ch{m}}
    \!\!\!\! P\big(x_w \,\big|\, X_m=
    x^{{\RR^\sharp_m}}_m,x_{\RR^\sharp_w\setminus m}\big) 
    \!\displaystyle\prod_{y \in Y}\! P(x_y \mid x_{\RR^\sharp_y}) 
    \, dx_{\RR^\sharp_m} \\
    &=P_{Q(X_{V^\sharp})}(x_{V^\sharp\dotcup V^\flat},x_{M\setminus\{m\}},x_S).
	\end{align*}
    Since the joint distribution is the same, the probabilities of $X_S=\zeros$ will also be the same, i.e.~$P'_{Q(X_{V^\sharp})}(X_{S}=\zeros)=P_{Q(X_{V^\sharp})}(X_{S}=\zeros)$. Thus, by marginalizing $X_{M\setminus\{m\}}$ and selecting on $X_{S}=\zeros$, we get:
\begin{equation*}
P'_{Q(X_{V^\sharp})}(X_{V^\sharp\dotcup V^\flat}\cmid X_{S}=\zeros)=P_{Q(X_{V^\sharp})}(X_{V^\sharp\dotcup V^\flat}\cmid X_S=\zeros).
\end{equation*}
Therefore, our choice of kernels gives $P'_*=P_*$. This shows that the SMI model of $\tilde{\calD}$ contains every element of the SMI model of $\calD$; that is, $\calC(\calD, V \mid S) \subseteq \calC(\tilde{\calD}, V \mid S)$.

\ding{212} \textbf{Proof of $\calC(\calD,V \mid S) \supseteq \calC(\tilde{\calD},V \mid S)$:}

Let $P'_*\in \calC(\tilde{\calD}, V \mid S)$ be a set of SMI pairs realizable by $\tilde{\calD}$, and let $P'(X_{y} \mid X_{\RR_y}), y\in V\dotcup M\dotcup S$ be the kernels associated with $P'_*$. Define the following kernels in $\calD$:
\begin{itemize}
    \item for each $\bar{x}_m,\bar{x}^1_m$,
    $P(X_m=(\bar{x}_m,\bar{x}^1_m) \mid X_{\RR_m}) = P'(X_m=\bar{x}_m ) \textbf{1}_{\{\bar{x}_{1}=X_{\RR_m}\}}$,
    where the indicator variable $\textbf{1}_{\{Y=Z\}}$ is equal to $1$ when $Y=Z$, and $0$ otherwise. That is, $X_m$ contains two variables, one that behaves as $X_m$ in $\tilde{\calD}$ and other that records the values of the parents of $m$.
    \item $P(X_w \!\mid \!X_m\!=\!(\bar{x}_m,\bar{x}^1_m),x_{\RR_w\setminus m})\!=\!P'(X_w\mid X_m=\bar{x}_m, X_{\RR_m}=\bar{x}_1, x_{\RR_w\setminus (\RR_m\cup\{m\})})$ for $w \in \ch{\calD}{m}$.
    \item $P(X_w \mid X_{\RR_w}) = P'(X_w \mid X_{\RR_w})$ otherwise.
\end{itemize}

Letting $Y:=(V^\flat \dotcup M \dotcup S)\setminus(\{m\}\cup \ch{\calD}{m}) $, the joint distribution realized by this choice of kernels is:
\begin{align*}
    &P_{Q(X_{V^\sharp})}(x_{V^\sharp},x_{V^\flat},x_{M\setminus\{m\}},x_S) \\
    &=\prod_{y \in Y} \int_{\bar{x}_m}\int_{\bar{x}_1} P'(x_y \mid 
    x_{\tilde{\RR}^{\sharp}_y}
    )  
    P'(X_m=\bar{x}_m )\, \textbf{1}_{\{\bar{x}_{1}=x_{\RR^\sharp_m}\}} \, Q(x_{V^\sharp})\\ & \quad\qquad\qquad\qquad\times\prod_{w\in\ch{\calD}{m}}P'\big(x_w \,\big|\, X_m=\bar{x}_m, X_{\RR_m}=\bar{x}_1, x_{\RR^\sharp_w\setminus (\RR^\sharp_m \cup\{m\})}\big) \,d\bar{x}_m \, d\bar{x}_1 \\
    &= P'_{Q(X_{V^\sharp})}(x_{V^\sharp},x_{V^\flat},x_{M\setminus\{m\}},x_S)
\end{align*}

By similar reasoning as in the previous case, this equality implies that the entire set of SMI pairs $P_*'$ is realizable by $\calD$, thus showing that $\calC(\calD,V \mid S) \supseteq \calC(\tilde{\calD},V \mid S)$.

   \textbf{Proof of b)}

    \ding{212} \textbf{Proof of $\calC(\calD,V\mid S) \supseteq  \calC(\calD_{\underline{s}},V \mid S)$:}
	Follows from \Cref{le:subgraph}, since $\calD_{\underline{s}}$ is a subgraph of $\calD$ with the same set of marginalized and selected vertices.
	
	\ding{212} \textbf{Proof of $\calC(\calD, V \mid S) \subseteq \calC(\calD_{\underline{s}}, V \mid S)$:} 
    The idea of the proof
    is that when an edge $s\rightarrow y$ is removed, we can simply set the variable associated with the child $y$ to the distribution that it would take in $\calD$ when $X_S=\zeros$.

    We will once again use $\RR_z$ as shorthand 
    for $\pa{\calD}{z}$. 
    We will let $\underline{\RR}_z$ denote the parents of $z$ in $\calD_{\underline{s}}$,
    and $\RR^\sharp_z$ the sharp parents of $z$ in $\calD$,
    and $\underline{\RR}^\sharp_z$ the sharp parents of $z$ 
    in  $\calD_{\underline{s}}$.
	Let $P_*\in \calC(\calD, V \mid S)$ be a set of SMI pairs compatible with $\calD$, and let $P(X_{y} \mid X_{\RR_y}), y\in V\dotcup M\dotcup S$ be the kernels associated with $P_*$. Consider now the following choice of kernels for $\calD_{\underline{s}}$:
	\begin{align}
	&y\in W \setminus \{s\} \Rightarrow	P'(X_y \mid X_{\underline{\RR}_y}) =\begin{cases}  P(X_y \mid X_{\RR_y\setminus\{s\}}, X_s=0) & y \in \ch{\calD}{s} \\[6pt] P(X_y\mid X_{\RR_y}) & \text{otherwise.} \end{cases} \label{eq_proof_b_param1} 
	\end{align}

	Let $P'_*\in \calC(\calD_{\underline{s}}, V \mid S)$ be the set of SMI pairs compatible with $\calD_{\underline{s}}$ that is obtained from the choice of kernels above. 
    Every pair $\big(Q(X_{V^\sharp}),\,P'_{Q(X_{V^\sharp})}(X_{V^\sharp\dotcup V^\flat}\cmid X_s=0)\big) \in P'_*$ must obey \eqref{eq_Markov_interventional}. Therefore, 
    letting $Y:=V^\flat \dotcup M \dotcup S \setminus \{s\}$,
    the joint distribution is:
	\begin{align*}
	&P'_{Q(X_{V^\sharp})}(X_{V^\sharp\dotcup V^\flat},X_M,X_{S \setminus s},X_{s}=0)\\
    &=Q(X_{V^\sharp})\,P'(X_s=0 \mid X_{\underline{\RR}^\sharp_s})\,\prod_{y \in Y} P'(X_y \mid X_{\underline{\RR}^\sharp_y}) \\
    &=Q(X_{V^\sharp}) \,P(X_s=0\mid X_{\RR^\sharp_s}) \, 
    \hspace{-3mm}\prod_{y \in Y\setminus \ch{\calD}{s}} \hspace{-3mm} 
    P(X_y \mid X_{\RR^\sharp_y})
    \hspace{-3mm}\prod_{y \in Y\cap \ch{\calD}{s}} \hspace{-3mm} 
    P(X_y \mid X_{\RR^\sharp_y \setminus \{s\}},X_s=0)
    \\
&= P_{Q(X_{V^\sharp})}(X_{V^\sharp\dotcup V^\flat},X_M,X_{S \setminus s},X_s=0).
	\end{align*}
Since the joint distribution is the same, the probabilities of $X_S=\zeros$ will also be the same, i.e.~$P'_{Q(X_{V^\sharp})}(X_{S}=\zeros)=P_{Q(X_{V^\sharp})}(X_{S}=\zeros)$. Thus, by marginalizing $X_M$ and selecting on $X_{S\setminus \{s\}}=\zeros$, we obtain:
\begin{equation*}
P'_{Q(X_{V^\sharp})}(X_{V^\sharp\dotcup V^\flat}\cmid X_{S}=\zeros)=P_{Q(X_{V^\sharp})}(X_{V^\sharp\dotcup V^\flat}\cmid X_S=\zeros).
\end{equation*}
Therefore, the choice of kernels \eqref{eq_proof_b_param1} %
gives $P'_*=P_*$. This shows that the SMI model of $\calD_{\underline{s}}$ contains every element of the SMI model of $\calD$; that is, $\calC(\calD, V \mid S) \subseteq \calC(\calD_{\underline{s}}, V \mid S)$.
\end{proof}

\subsection{Proof of \Cref{prop_merging}} \label{app:smm-invariances-e}

\MergingRestatable*

We will begin by proving part (a).
To show that every SMI pair compatible with $\calD$  is also compatible with $\tilde{\calD}$ 
(where $m_1$ and $m_2$ are merged), 
our approach will be to define the variable $X_m$ in $\tilde{\calD}$ 
as the Cartesian product of variables $X_{m_1}$ and $X_{m_2}$ from $\calD$.
To demonstrate the converse, the idea will be that even though $m_1$ and $m_2$ are separate in $\calD$, we can use the common selected child to make sure that the variables $X_{m_1}$ and $X_{m_2}$ are always equal, so that all their children can treat them as though they 
were a single variable.

\begin{proof}[Proof of \Cref{prop_merging} part (a)]

As shorthand, let $\RR_z:=\pa{\calD}{z}$ and $\tilde{\RR}_z:=\pa{\tilde{\calD}}{z}$.
For sharp variables, let 
$\RR^\sharp_z:=\pas{\calD}{z}$ and 
$\tilde{\RR}^\sharp_z:=\pas{\tilde{\calD}}{z}$.

\ding{212} \textbf{Proof that $\calC(\calD,V \mid S) \subseteq \calC(\tilde{\calD},V \mid S)$:}

To reproduce a set of SMI pairs $P_*$ that is realizable by $\calD$ in $\tilde{\calD}$, we simply define $X_m$ in $\tilde{\calD}$ as the Cartesian product of $X_{m_1}$ and $X_{m_2}$. The distribution over these variables is chosen to be the same as it is in $\calD$. It is not hard to see that this choice allows for $P_*$ to be realized in $\tilde{\calD}$.

\ding{212} \textbf{Proof that $\calC(\calD,V \mid S) \supseteq \calC(\tilde{\calD},V \mid S)$:}

Define $\sm=\{m_1,m_2\}$.

Let $P'_*\in \calC(\tilde{\calD}, V \mid S)$ be an arbitrary set of SMI pairs compatible with $\tilde{\calD}$, and let $P'(X_{y} \mid X_{\tilde{\RR}_y}), y\in V \dotcup M' \dotcup S$ be the kernels associated with $P'_*$.  Consider now the following choice of kernels for $\calD$:
\begin{itemize}
	\item $P(X_{m_1}=x_m) = P'(X_m=x_m)$.
	\item $P(X_{m_2}=x_m) \sim \Unif(\dom{m})$, i.e. a uniform distribution over the domain of $X_m$.
	\item $P(X_s=0 \mid X_{{m_1}}=x_{{m_1}},  X_{{m_2}}=x_{{m_2}}, X_{\RR_s \setminus \sm}) = P'(X_s=0 \mid X_m=x_{m_1}, X_{\tilde{\RR}_s \setminus m}) \textbf{1}_{\{x_{m_1}=x_{m_2}\}}$ for every $x_{\sm}$ and for an arbitrarily chosen $s \in S \cap \ch{m_1} \cap \ch{m_2}$, 
	where the indicator variable $\textbf{1}_{\{Y=Z\}}$ is equal to $1$ when $Y=Z$, and $0$ otherwise. 
	\item $P(X_y \cmid  X_{{m_1}}=x_{m_1},X_{\RR_y \setminus m_1}) = P'(X_y \cmid  X_m=x_{m_1}, X_{\tilde{\RR}_y \setminus \{m\}})$ for every $x_{\sm}$ and
    $y \in \ch{\calD}{m_1} \setminus \{s\}$. Note that if $y$ is also a child of $m_2$ in $\calD$, this choice of kernels says that $y$ ignores the parent $m_2$.
        \item $P(X_y \cmid  X_{{m_2}}=x_{m_2},X_{\RR_y \setminus m_2}) = P'(X_y \cmid  X_m=x_{m_2}, X_{\tilde{\RR}_y \setminus \{m\}})$ for every $x_{\sm}$ and
    $y \in \ch{\calD}{m_2} \setminus \ch{\calD}{m_1}$.
	\item $P(X_y \mid X_{\RR_y}) = P'(X_y \mid X_{\tilde{\RR}_y})$ 
    for $y \in V\dotcup M\dotcup S\setminus (\sm \cup \ch{\calD}{\sm} \cup \{s\})$.
\end{itemize}
{ In this construction, \(m_1\) carries the value of the merged variable, while \(m_2\) is only an auxiliary copy sampled independently. The role of the selected child \(s\) is to filter to the event \(X_{m_1}=X_{m_2}\), so after conditioning all children behave as though they were receiving a single parent \(m\). Thus, given $X_s=0$, it will follow that $X_{m_1}=X_{m_2}$.}

Let $P_*\in \calC(\calD, V \mid S)$ be the set of SMI pairs 
obtained by the kernels above. 
{ We will now prove that any pair in $P_*$ is also in $\calC(\calD,V \mid S)$.
The idea will be to focus on the factors that involve $m_1,m_2,s$; selecting on $X_s=0$ will imply that the assignment $x_{m_1}=x_{m_2}$.
Thus, after integrating and conditioning, these factors will behave the same as the factors of $m,s$ in $\calC$.
Secondly, the factors corresponding to children of $m_1,m_2$ will, when $s=0$, behave just 
as if their parent was $m$ in $\calC$. 
And thirdly, the factors in $Y:= (V^\flat \dotcup M \dotcup S)\setminus \{m_1,m_2,s\}$, the ones that do not involve $m_1,m_2,s$ at all, will be unchanged.}
Every pair $\big(Q(X_{V^\sharp}),P_{Q(X_{V^\sharp})}(X_{V^\sharp\dotcup V^\flat}\cmid X_{S}=\zeros)\big) \in P_*$ must obey \eqref{eq_Markov_interventional}; therefore, the joint distribution over all variables of $\calD$ except for $X_{m_1}$ and $X_{m_2}$ must satisfy, 

for any assignments $x_{V^\sharp \dotcup V^\flat \dotcup (M \setminus \{m_1,m_2\}) \dotcup S\setminus s}$:
\begin{align*}
	&P_{Q(X_{V^\sharp})}(x_{V^\sharp \dotcup V^\flat \dotcup (M \setminus \{m_1,m_2\}) \dotcup S\setminus s},X_s=0) \\
    &= \int_{x_{m_1}}\!\!\int_{x_{m_2}} \hspace{-4mm}Q(x_{V^\sharp})P(x_{m_1})P(x_{m_2}) P(X_s=0 | x_{m_1},x_{m_2},x_{\RRs_s \setminus \sm}) \prod_{y \in Y} P(x_y| x_{\RRs_y}) dx_{m_2}dx_{m_1}\\
    & =c \!\!\int_{x_{m_1}}\!\!\int_{x_{m_2}} \hspace{-4mm} Q(x_{V^\sharp}\!)P'(X_m\!=\!x_{m_1}) 
    P'(X_s\!=\!0 | X_m\!=\!x_{m_1},x_{\tilde{\RR}^\sharp_s\setminus \{m\}}) \textbf{1}_{\{x_{m_1}\!=\!x_{m_2}\}} \hspace{-7mm} \prod_{y \in Y \setminus \ch{\tilde{\calD}}{m}} \hspace{-7mm}P'(x_y|x_{\RRs_y}) \\
    &\qquad \hspace{-6mm}\prod_{y \in \ch{\calD}{m_1} \setminus s}\hspace{-6mm} P'(x_y \cmid  X_m=x_{m_1}, x_{\tilde{\RR}^\sharp_y \setminus \{m\}})
    \hspace{-6mm}\prod_{y \in \ch{\calD}{m_2} \setminus \ch{\calD}{m_1}}\hspace{-6mm} P'(x_y \cmid  X_m=x_{m_2}, x_{\tilde{\RR}^\sharp_y \setminus \{m\}})
    \, dx_{m_2}\,dx_{m_1} \\
    &= c
    \int_{x_{m}} \hspace{-3mm} Q(x_{V^\sharp})\,P'(x_{m})\,P'(X_s=0 \mid x_{\tilde{\RR}^\sharp_s}) \hspace{-4mm}\prod_{y \in Y \setminus \ch{\tilde{\calD}}{m}} \hspace{-5mm} P'(x_y\mid x_{\tilde{\RR}^\sharp_y})\hspace{-4mm} \prod_{y \in Y \cap \ch{\tilde{\calD}}{m}} \hspace{-6mm} P'(x_y\mid x_m, x_{\tilde{\RR}^\sharp_y\setminus m}) \, dx_m\\
    &=c
    P'_{Q(X_{V^\sharp})}(x_{V^\sharp \dotcup V^\flat \dotcup (M \setminus \{m_1,m_2\}) \dotcup S\setminus s},X_s=0)
\end{align*}
where the proportionality constant $c \in \mathbb{R}$  comes from the uniform distribution on $X_{m_2}$.

One consequence of this equation is that 
$P_{Q(X_{V^\sharp})}(X_{S}=\zeros) = c
P'_{Q(X_{V^\sharp})}(X_{S}=\zeros)$. This implies that:
\begin{equation}
	P_{Q(X_{V^\sharp})}(X_{V^\sharp\dotcup V^\flat}\cmid X_{S}=\zeros)=
    \frac{c
    P'_{Q(X_{V^\sharp})}(X_{V^\sharp\dotcup V^\flat},X_{S}=\zeros)}
    {
    c
    P'_{Q(X_{V^\sharp})}(X_{S}=\zeros)}=P'_{Q(X_{V^\sharp})}(X_{V^\sharp\dotcup V^\flat}\cmid X_S=\zeros). \nonumber
\end{equation}
Thus, every SMI pair of $\calC(\tilde{\calD},V\mid S)$ is in $\calC(\calD,V \mid S)$.
\end{proof}

We will now prove part (b).
To prove that every SMI pair compatible with $\calD$  is also compatible with $\tilde{\calD}$ 
(where $s_1$ and $s_2$ are merged), 
we will choose the variable $X_s$ in $\tilde{\calD}$ 
so that its domain contains all of the information from $X_{s_1}$ and $X_{s_2}$.
To show the converse, the idea 
is to use $m$
to uniformly sample over the joint domain of all parents of $s_1$ and $s_2$, 
and for $s_1$ and $s_2$ to 
select the event where $m$ contains the values of all parents of $s_1$ and $s_2$.
Then, $s_1$ can depend on $X_m$ in $\calD$
just as $s$ depends on $X_m$ in $\tilde{\calD}$.

\begin{proof}[Proof of \Cref{prop_merging} part (b)]

\ding{212} \textbf{Proof that $\calC(\calD,V \mid S) \subseteq \calC(\tilde{\calD},V\mid S \cup \{s\} \setminus \{s_1,s_2\})$:} 

To reproduce a set of SMI pairs $P_*$ that is realizable by $\calD$ in $\tilde{\calD}$, we simply define $X_s$ in $\tilde{\calD}$ as the Cartesian product of $X_{s_1}$ and $X_{s_2}$. The distribution over these variables is chosen to be the same as it is in $\calD$. It is not hard to see that this choice allows for $P_*$ to be realized in $\tilde{\calD}$.

\ding{212} \textbf{Proof that $\calC(\calD,V \mid S) \supseteq \calC(\tilde{\calD},V\mid S \cup \{s\} \setminus \{s_1,s_2\})$:}

As shorthand, let $\RR_z:=\pa{\calD}{z}$ and $\tilde{\RR}_z:=\pa{\tilde{\calD}}{z}$.
For sharp variables, let 
$\RR^\sharp_z:=\pas{\calD}{z}$ and 
$\tilde{\RR}^\sharp_z:=\pas{\tilde{\calD}}{z}$.
Let $\bm{s}:=\{s_1,s_2\}$, and let $S':=S\cup \{s\}\setminus \bm{s}$ be the set of selected vertices of $\tilde{\calD}$. 

First let $m\in \RR_{s_1} \cap \RR_{s_2}\cap M$ be a marginalized vertex that is a parent of both $s_1$ and $s_2$ in $\calD$. 
In $\tilde{\calD}$, by assumption, $m$ is a parent of $s$.

Let $P'_*\in \calC(\tilde{\calD}, V \mid S')$ be a set of SMI pairs compatible with $\tilde{\calD}$, and let $P'(X_{y} \mid X_{\RR_y}), y\in V\dotcup M\dotcup S'$ be the kernels associated with $P'_*$. Consider now the following choice of kernels for $\calD$:
\begin{itemize}
	\item Let the variable $X_m$ associated with $m$ in $\calD$
    equal a Cartesian product of independent variables $\bar{X}_m \times \bar{X}_{1} $,
    where $\bar{X}_m$ is distributed like $X_m$ in $\tilde{\calD}$ and
    $\bar{X}_{1} \sim \Unif(\dom{\RR_{s_1}\setminus \RR_{s_2}})$.
    So,
    $P(X_m=\langle\bar{x}_m, \bar{x}_{1}\rangle) \propto P'(X_m=\bar{x}_m)$
    for any $\bar{x}_m, \bar{x}_{1}$.
\item $P(X_{s_1}=0 \cmid  X_m=\langle\bar{x}_m, \bar{x}_{1}\rangle, X_{\RR_{s_1}\setminus m}=x_{\RR_{s_1}\setminus \{m\}})=  \textbf{1}_{\{\bar{x}_{1}=x_{\RR_{s_1}\setminus \RR_{s_2}}\}}$
for each $x_{\RR_{s_1}\setminus m},\bar{x}_m,\bar{x}_{1}$. That is, selecting on $X_{s_1}=0$ establishes that $\bar{x}_{1}$ equals $x_{\RR_{s_1}\setminus \RR_{s_2}}$. 
    \item $P(X_{s_2}=0 \cmid  X_m=\langle\bar{x}_m, \bar{x}_{1}\rangle, X_{\RR_{s_2}\setminus \{m\}}=x_{\RR_{s_2}\setminus \{m\}})\\ 
                =P'(X_s=0\cmid  X_m=\bar{x}_m, X_{\RR_{s_1}\setminus\RR_{s_2}}=\bar{x}_{1}, X_{\RR_{s_2}\setminus \{m\}}=x_{\RR_{s_2}\setminus \{m\}})$
                for each $x_{\RR_{s_2}\setminus m}$, $\bar{x}_m$, $\bar{x}_{1}$.
                That is, $X_{s_2}$ depends on $\bar{X}_m$, $\bar{X}_{1}$ and $X_{\RR_{s_2}\setminus\{m\}}$ similarly to how $X_s$ depends on  $X_m$, $X_{\RR_{s_1}\setminus \RR_{s_1}}$ and $X_{\RR_{s_2}\setminus\{m\}}$ in $\tilde{\calD}$. 
	\item $P(X_y \mid  X_m=\langle\bar{x}_m, \bar{x}_{1}\rangle, X_{\RR_{y} \setminus \{m\}}) = P'(X_y \mid X_m=\bar{x}_m, X_{\RR_{y} \setminus \{m\}})$ for each $\bar{x}_m, \bar{x}_{1}$ for
    $y \in \ch{\calD}{m} \setminus \bm{s}$.
	\item  $P(X_y \mid X_{\RR_{y}}) = P'(X_y \mid X_{\RR_{y}})$ for $y \notin \ch{\calD}{m} \cup \{m\}$.
\end{itemize}

{ The next calculation shows that, before normalization, the joint distribution under selection in \(\calD\) agrees with the corresponding selected joint distribution in \(\tilde{\calD}\) up to a constant factor coming from the uniform choice of \(\bar{X_1}\).}
Let $P_*\in \calC(\calD, V \mid S)$ be the set of SMI pairs compatible with $\calD$ that is obtained from these kernels. 
Every given element $\big(Q(X_{V^\sharp}),P_{Q(X_{V^\sharp})}(X_{V^\sharp\dotcup V^\flat}\cmid X_{S}=\zeros)\big) \in P_*$ must obey \eqref{eq_Markov_interventional}; therefore, 
letting $Y=V^\flat \dotcup (M \setminus m) \dotcup (S \setminus \bm{s})$,
the joint distribution 
$P_{Q(X_{V^\sharp})}(x_{V^\sharp\dotcup V^\flat\dotcup M\dotcup (S\setminus s)}\dotcup X_{s_1}=0,X_{s_2}=0)$
must, for any assignment $x_{V^\sharp\dotcup V^\flat\dotcup (M\setminus m)\dotcup (S\setminus s)}$, satisfy:
\begin{align}
	&P_{Q(X_{V^\sharp})}(x_{V^\sharp\dotcup V^\flat\dotcup (M\setminus m)\dotcup (S\setminus s)},X_{s_1}=X_{s_2}=0) \nonumber\\
	&= \int_{x_m} Q(X_{V^\sharp})\,P(x_m) \,P(X_{s_1}=0\cmid X_{\RRs_{s_1}}) \,P(X_{s_2}=0\cmid X_{\RRs_{s_2}}) \prod_{y \in Y} P(X_y \mid X_{\RRs_{y}})\, dx_m \nonumber\\
	&=c \hspace{-1mm}\int_{\bar{x}_m,\bar{x}_{1}}\hspace{-5mm} 
    Q(x_{V^\sharp})\,P'(X_m=\bar{x}_m) \,
    \textbf{1}_{\{ \bar{x}_{1}=x_{\RRs_{s_1}\setminus \RRs_{s_2}}\}}
    \nonumber\\ \!\!\!
        &\qquad \times P'(X_s=0\cmid  X_m=\bar{x}_m, X_{\RRs_{s_1}\setminus\RRs_{s_2}}=\bar{x}_{1}, X_{\RRs_{s_2}\setminus \{m\}}=x_{\RR_{s_2}\setminus \{m\}})  \\
         &\qquad \times\prod_{y \in Y \setminus \ch{\calD}{m}}\!\!\! P'(x_y \mid x_{\RRs_{y}}) 
    \hspace{-2mm}\prod_{y \in Y \cap \ch{\calD}{m}}\hspace{-2mm} P'(x_y \mid X_m=\bar{x}_m,x_{\RRs_{y}\setminus \{m\}})\,d\bar{x}_m \, d\bar{x}_{1}  \nonumber\\
	&=c 
    \int_{x_m} P'(x_m) 
    \prod_{y \in Y} P'(x_y \mid x_{\RRs_{y}}) \, P'(X_s=0 \mid x_{\tilde{\RR}^\sharp_{s}}) \, Q(x_{V^\sharp}) dx_m \nonumber\\
    &= c
    P'_{Q(X_{V^\sharp})}(x_{V^\sharp\dotcup V^\flat\dotcup (M\setminus m)\dotcup (S\setminus s)},X_s=0)
	\label{eq_aaaa1}
\end{align}
where the proportionality constant $c \in \mathbb{R}$  comes from the uniform distribution on $\bar{X}_1$.
{ The key step is that the indicator 
$1_{\{\bar x_1=x_{R^\sharp_{s_1}\setminus R^\sharp_{s_2}}\}}$
collapses the integral over $\bar x_1$, substituting the realized assignment
$x_{R^\sharp_{s_1}\setminus R^\sharp_{s_2}}$ into the kernel for $X_s$ in $\tilde{\calD}$.
Once this happens, the remaining factors are exactly those of $\tilde{\calD}$, with only the constant $c$ left over.}

By marginalizing $X_{V^\sharp\dotcup V^\flat},X_{M\setminus m}$ 
and conditioning $X_{S\setminus s}=0$,
we obtain 
$$P_{Q(X_{V^\sharp})}(X_{S}=\zeros) =c P'_{Q(X_{V^\sharp})}(X_{S'}=\zeros).$$ 
Together with \eqref{eq_aaaa1}, this implies that:
\begin{equation*}
	P_{Q(X_{V^\sharp})}(X_{V^\sharp\dotcup V^\flat}\cmid X_{S}=\zeros)=P'_{Q(X_{V^\sharp})}(X_{V^\sharp\dotcup V^\flat}\cmid X_{S'}=\zeros).
\end{equation*}
Thus $P_*=P'_*$, so the SMI of $\calD$ dominates the SMI of $\tilde{\calD}$.
\end{proof}

\subsection{Proof of \Cref{prop:splitarrows}} \label{app_split_arrows}
We begin by restating the result.

\SplittingArrows*

{ 
For the $\supseteq$ direction, it is just a matter of having the $m$ and $s$ variables in $\calD$
be the cross-product of all of the marginalized and selected variables in $\tilde{\calD}$ respectively.

For the $\subseteq$ direction, at a high level, the added special edges in $\tilde{\calD}$ transmit the realized values of $X_{V_s}$ to the vertices in $V_m$, while the new variable $X_m$ stores a copy of the original $m$-distribution for each possible value of $X_{V_s}$. After conditioning on the added selection variables, each $b \in V_m$ accesses exactly the copy corresponding to the realized value of $X_{V_s}$, so the same selected interventional distributions are recovered.}

\begin{proof}

\ding{212} \textbf{Proof that $\calC(\calD,V \mid S)\supseteq \calC(\tilde{\calD},V\mid S\cup S')$.}

To reproduce a set of SMI pairs $P'_*$ that is realizable by $\tilde{\calD}$ in $\calD$, we simply 
define $X_m$ in $\calD$ as the Cartesian product of the variable $X_m$ in $\tilde{\calD}$ and every variable $X_{m_{ab}}$, and  define $X_s$ in $\calD$ as the Cartesian product of the variable $X_s$ in $\tilde{\calD}$ and every variable $X_{s_{ab}}$. The distribution over these variables is chosen to be the same as it is in $\tilde{\calD}$. It is not hard to see that this choice allows for $P'_*$ to be realized in $\calD$.

\ding{212} \textbf{Proof that $\calC(\calD,V \mid S)\subseteq\calC(\tilde{\calD},V\mid S\cup S')$.}

As shorthand, we will let $\RR_z:=\pa{\calD}{z}$ and $\tilde{\RR}_z:=\pa{\tilde{\calD}}{z}$ for a vertex $z$.
For sharp variables, let 
$\RR^\sharp_z:=\pas{\calD}{z}$ and 
$\tilde{\RR}^\sharp_z:=\pas{\tilde{\calD}}{z}$.

Let $P_* \in \calC(\calD,V\mid S)$ be an arbitrary set of SMI pairs compatible with $\calD$, with associated kernels $P(X_y \mid X_{\RR_y}), y \in V\dotcup M\dotcup S$. We will show that $P_*$ is also realizable by $\tilde{\calD}$. The idea is as follows: the kernels associated with the { marginalized} and selected variables $X_{m_{ab}}$ and $X_{s_{ab}}$ for each $a \in V_s,b \in V_m$ will be chosen in such a way that $b$ has access to the value of $X_a$. Furthermore, the vertex $m$ in $\tilde{\calD}$ will be associated with a library of variables $X_m^{x_{V_s}}$, one for each assignment of values $x_{V_s}$ for the variables in $X_{V_s}$. Each variable  $X_m^{x_{V_s}}$ will follow the distribution of $X_m$ in $\calD$ conditioned on $X_s=0$ and on $X_{V_s}=x_{V_s}$. Since the vertices $b\in V_m$ have access to the value of every variable $X_{V_s}$, they can look at the library provided by $m$ and pick the variable  $X_m^{x_{V_s}}$ that corresponds to the values that the variables $X_{V_s}$ actually take.

Formally, we choose the kernels on $\tilde{\calD}$ to be:
\begin{itemize}
    \item For each $a \in V_s,b \in V_m$:
    \begin{itemize}
        \item[*] $X_{m_{ab}} \sim \operatorname{Unif}(\dom{a})$,
        \item[*] $P'(X_{s_{ab}}=0\mid X_{m_{ab}},X_{a}) = \textbf{1}_{\{ X_{m_{ab}}=X_a \}}$.
    \end{itemize}
    \item 
    $P'(X_m=\times_{x_{V_s} \in \dom{{V_s}}}x^{x_{V_s}}_m) = 
    \prod_{x_{V_s} \in \dom{{V_s}}} P(X_m =x^{x_{V_s}}_m \mid X_{V_s}=x_{V_s},X_s=0)
    $ for each assignment $\times_{x_{V_s} \in \dom{{V_s}}}x^{x_{V_s}}_m$,
    i.e.~$X_m$ is a Cartesian product of variables $X^{x_{V_s}}_m$
    where $P'(X_m^{x_{{V_s}}}) = P(X_m \mid X_{V_s},X_s=0)$. 
    \item $P'(X_s \mid X_{V_s}) = P(X_s \mid X_{V_s})=\int_{\dom{m}} P(X_s\cmid X_{V_s},x_{m})\, P(x_m)\, dx_m$.
    \item For each $b \in {V_m}$, let $m_{{V_s}b}$ be the set of added marginalized parents $\{m_{ab} \mid a \in {V_s}\}$. Then:
    
    $P'(X_b \mid X_m=x_m,X_{m_{{V_s}}}=x_{m_{{V_s}b}},X_{\tilde{\RR}_b\setminus\{m,m_{{V_s}b}\}}) = P\big(X_b \,\big|\, X_m=x_m^{x_{m_{{V_s}b}}},X_{\pa{\calD}{b}\setminus\{m\}}\big)$, where $x_m=\times_{x_{V_s} \in \dom{{V_s}}}x^{x_{V_s}}_m$.
    \item $P'(X_y \mid X_{\tilde{\RR}_y})=P(X_y \mid X_{\RR_y})$ for $y \in (V \setminus {V_m}) \dotcup (M\setminus m) \dotcup (S\setminus s)$.
\end{itemize}

Define $Y = V^\flat \dotcup (M\setminus m) \dotcup (S \setminus s)$, and $m_{V_sV_m}=\{m_{V_sb}\mid b\in V_m\}$ as the set of all added marginalized vertices. Furthermore, define the shorthand notation $\bm{x}_m=\times_{x_{V_s} \in \dom{{V_s}}} x_m^{x_{V_s}}$.
Let $P'_*\in \calC(\tilde{\calD}, V \mid S\cup S')$ be the set of SMI pairs compatible with $\tilde{\calD}$ that is obtained from the choice of kernels above. Every given element $\big(Q(X_{V^\sharp}),P'_{Q(X_{V^\sharp})}(X_V\cmid X_{S}=\zeros)\big) \in P'_*$ has to obey \eqref{eq_Markov_interventional}; therefore, the joint distribution over all variables of $\tilde{\calD}$ except for $X_m$ and $X_{m_{V_sV_m}}$ has to be:
\begin{align*}
& P'_{Q(X_{V^\sharp})}\big(X_{V^\sharp\dotcup V^\flat\dotcup (M\setminus m)\dotcup (S \setminus s)}=x_{V^\sharp\dotcup V^\flat\dotcup (M\setminus m)\dotcup (S \setminus s)},X_{S'\dotcup s}=\zeros\big) %
\\
 &=\int_{\bm{x}_m}\int_{x_{m_{{V_s}{V_m}}} } \hspace{-7mm}
 Q(x_{V^\sharp})
 P'(X_s=0 \mid x_{{V_s}^\sharp})\,
 P'(X_m=x_m)
 \Big\{\prod_{b \in V_s}  P'\big(x_b \mid x_{m}, x_{m_{V_sb}},x_{\tilde{\RR}^\sharp_b\setminus \{m,m_{{V_s}b}\}}\big)
\\ 
 &\qquad \times \prod_{a \in {V_m}} P'(x_{m_{ab}})P'(X_{s_{ab}}=0\mid x_{m_{ab}}, x_{{V_s}^\sharp})\Big\}
 \prod_{y \in Y \setminus {V_m}}P'(x_y \mid x_{\tilde{\RR}^\sharp_y})
 \, d{x_{m_{{V_s}{V_m}}}} \, d\bm{x}_m
 \\ 
  & = c\int_{\bm{x}_m}\!\int_{x_{m_{{V_s}{V_m}}} }
  \hspace{-10mm}Q(x_{V^\sharp})\hspace{-3mm}
 \prod_{a,b \in {V_s} \times {V_m}} \hspace{-5mm}
 \textbf{1}_{\{x_{m_{ab}}=x_{a^\sharp}\}}
 P(X_s=0 | x_{{V_s}^\sharp}) \hspace{-2mm}
 \prod_{x'_{V_s} \in \dom{{V_s}}} \hspace{-3mm}
 P(X_m =x^{x'_{V_s}}_m | X_{V_s^\sharp}=x'_{V_s},X_s=0)
\\
 &\qquad \times P\big(x_b \,\big|\, X_{m}=x_{m}^{x_{m_{{V_s}b}}},x_{\RRs_b\setminus \{m\}}\big)
 \prod_{y \in Y \setminus {V_m}}P(x_y \mid x_{\RRs_y}) \,
 d{x_{m_{{V_s}{V_m}}}} \, d\bm{x}_m
 \\ 
&=c
\int_{x_m^{x_{V_s^\sharp}}}
Q(x_{V^\sharp})
P(X_s=0 \mid x_{{V_s}^\sharp}) 
P\big(X_m=x_m^{x_{V_s^\sharp}} \,\big|\, x_{{V_s}^\sharp}, X_s=0\big) 
\\ & \qquad\times\int_{\bm{x}_m\setminus\big\{x_m^{x_{V_s^\sharp}}\big\}}\prod_{x'_{V_s} \neq x_{V_s^\sharp}} P(X_m=x^{x'_{V_s}}_m\mid X_{{V_s}^\sharp}=x'_{V_s},X_s=0) \, d\big(\bm{x}_m\setminus\big\{x_m^{x_{V_s^\sharp}}\big\}\big) 
\\&
\qquad \times\prod_{b \in {V_m}} P\big(x_b \,\big|\, X_m= x_m^{x_{V_s^\sharp}},x_{\RRs_b\setminus\{m\}}\big)
\prod_{y \in Y \setminus {V_m}}P(x_y \mid x_{\RRs_y}) 
\, dx_m^{x_{V_s^\sharp}} \\
&= c%
\int_{x_m^{x_{V_s^\sharp}}}
Q(x_{V^\sharp}) 
P(X_m=x_m^{x_{V_s^\sharp}},X_s=0 \mid x_{{V_s}^\sharp})
\prod_{y \in Y} P(x_y \mid x_{\RRs_y}) \, d{x_m^{x_{V_s^\sharp}}} \\
&= c%
 P_{Q(X_{V^\sharp})}(x_{V^\sharp\dotcup V^\flat\dotcup (M\setminus m)\dotcup (S \setminus s)},X_s=0)
\end{align*}
where the proportionality constant $c \in \mathbb{R}$  comes from the uniform distribution on $X_{m_{V_sV_m}}$, and the penultimate step removes the integral over all parts of $\bm{x}_m$ other than $x^{x_{V_s^\sharp}}_m$, 
because they are independent of all other variables, 
and so are marginalized out.
\ryan {Still plenty of room to improve this, maybe improve the $\bm{x}_m$ notation, among other things.}

By marginalizing $M$ and conditioning upon $X_S=\zeros$ we will cancel out $c$ and
recover the marginal selected distribution 
from $\calD$, and so any SMI pairs compatible with $\calD$ can be reproduced in $\tilde{\calD}$.
\end{proof}

\subsection{Proof of \Cref{lemma_interchangeedges}} \label{app_interchange}
    
We begin by restating the proposition.

\interchangeedges*

\begin{proof}
\ding{212} \textbf{Proof that $\calC(\calD,V \mid S) \subseteq \calC(\calD',V\mid S\cup \{s_{ab}\})$.}

The idea is that we can choose the distribution on $X_{m_{ab}}$ and $X_{s_{ab}}$ in such a way that, after selecting on $X_{s_{ab}}=0$, $b$ has access to the value of $X_a$.

Let $P_* \in \calC(\calD,V\mid S)$ be an arbitrary set of SMI pairs compatible with $\calD$, with associated kernels $P(X_y \mid X_{\pa{\calD}{y}}), y \in V\dotcup M\dotcup S$. Now, we define a set of kernels on $\calD'$.

\begin{itemize}
\item $X_{m_{ab}} \sim \operatorname{Unif}(\dom{a})$.
\item $P'(X_{s_{ab}}=0 \mid X_{m_{ab}},X_a) = \mathbf{1}_{\{X_{m_{ab}} = X_a\}}$.
\item $P'(X_{b} \mid X_{\pa{\calD'}{b}}=x_{\pa{\calD'}{b}})=P(X_{b} \mid X_a=x_{m_{ab}},X_{\pa{\calD}{b} \setminus a}=x_{\pa{\calD}{b} \setminus a})$. 
Note that $b$ has a marginalized parent in $\calD'$, so its distribution is allowed to be stochastic given 
its parent's values.
\item $P'(X_y \mid X_{\pa{\calD'}{y}})=P(X_y \mid X_{\pa{\calD}{y}})$ for $y \in (V \setminus b) \dotcup M \dotcup S$.
\end{itemize}

Then, letting 
$Y=(V^\flat \setminus {v'}^\flat) \cup M \dotcup S $,
this choice of kernels gives the joint distribution:
\begin{align*}
& P'_{Q(X_{V^\sharp})}(x_{V^\sharp\dotcup V^\flat\dotcup M\dotcup S},X_{s_{ab}}=0) \\
&= \int_{x_{m_{ab}}} \hspace{-5mm} Q(x_{V^\sharp})\, P'(x_{m_{ab}}) \, P'(X_{s_{ab}}=0 \mid x_{a^\sharp},x_{m_{ab}}) \, P'(x_{b} \mid x_{\pas{\calD'}{b}}) \prod_{y \in Y} P'(x_y \mid x_{\pas{\calD'}{y}}) \, dx_{m_{ab}}\\
&= c \int_{x_{m_{ab}}} \hspace{-5mm} Q(x_{V^\sharp})\,\mathbf 1_{\{x_{m_{ab}} = x_{a^\sharp}\}} \,
   P(x_{b} \mid X_{a^\sharp}=x_{m_{ab}},x_{\pas{\calD}{b} \setminus a^\sharp})\,
   \prod_{y \in Y} P(x_y \mid x_{\pas{\calD}{y}}) \,dx_{m_{ab}}\\
&=c \, Q(x_{V^\sharp}) \, P(x_{b} \mid X_{a^\sharp}=x_{a^\sharp},x_{\pas{\calD}{b}\setminus a^\sharp}) \prod_{y \in Y} P(x_y \mid x_{\pas{\calD}{y}})\\
&= cP_{Q(x_{V^\sharp})}(x_{V^\sharp\dotcup V^\flat\dotcup M\dotcup S})
\end{align*}
where the proportionality constant $c \in \mathbb{R}$  comes from the uniform distribution on $X_{m_{ab}}$.

From this, we conclude that 
$P_{Q(X_{V^\sharp})}(X_{V^\sharp\dotcup V^\flat} \mid X_{S\dotcup s_{ab}}=\zeros) = P'_{Q(X_{V^\sharp})}(X_{V^\sharp\dotcup V^\flat} \mid X_{S}=\zeros)$, thus showing that the set of SMI pairs $P_*$ can be reproduced by $\calD'$.

\ding{212} \textbf{Proof that $\calC(\calD,V \mid S) \supseteq \calC(\calD',V\mid S\cup\{s_{ab}\})$}

To help with this proof, define the DAG $\calD''$ obtained by starting from $\calD'$ and adding an edge $a \to b$. That is, $\calD''$ has \emph{both} a regular edge and a special edge between $a$ and $b$. Clearly, $\calD''$ interventionally dominates $\calD'$. Now, we will show that $\calD$ interventionally dominates $\calD''$, which implies that   $\calD$ interventionally dominates $\calD'$.

Let $P''_* \in \calC(\calD'',V\mid S\cup \{s_{ab}\})$ be an arbitrary set of SMI pairs compatible with $\calD''$, with associated kernels $P''(X_y \mid X_{\pa{\calD''}{y}})$, $y \in V\dotcup M\dotcup S\cup \{s_{ab}\}$. 
To reproduce the same set of SMI pairs in $\calD$, define the following kernels: 
\begin{itemize}
\item $P(X_s \mid X_{\pa{\calD}{s}}) = P''(X_s \mid X_{\pa{\calD}{s}})P''(X_{s_{ab}}=0 \mid X_a)$, where $s$ is any
arbitrarily selected child of $a$ in $\calD$.
\item $P(X_{b} \mid X_{\pa{\calD}{b}}) = P''(X_{b} \mid X_{\pa{\calD''}{b}\setminus\{m_{ab}\}},X_{s_{ab}}=0)$, and
\item $P(X_y \mid X_{\pa{\calD}{y}})=P''(X_y \mid X_{\pa{\calD''}{y}})$ for $y \in V\dotcup M\dotcup S \setminus \{s,b\}$,
\end{itemize}

Expanding $P''(X_b \mid X_{\pa{\calD''}{b}\setminus \{m_{ab}\}},X_{s_{ab}}=0)$, we get:
\begin{align*}
& P''(X_b \mid X_{\pa{\calD''}{b}\setminus \{m_{ab}\}},X_{s_{ab}}=0) \\
   &= \int_{x_{m_{ab}}} P''(X_b, x_{m_{ab}} \mid X_{\pa{\calD''}{b}\setminus \{m_{ab}\}},X_{s_{ab}}=0)\,dm_{ab} \\
   &= \int_{x_{m_{ab}}} P''(X_b \mid x_{m_{ab}}, X_{\pa{\calD''}{b}\setminus \{m_{ab}\}},X_{s_{ab}} = 0)\,
            P''(x_{m_{ab}} \mid X_{\pa{\calD''}{b}\setminus \{m_{ab}\}},X_{s_{ab}} = 0)\,dm_{ab} \\
&= \int_{x_{m_{ab}}} P''(X_b \mid x_{m_{ab}}, X_{\pa{\calD''}{b}\setminus \{m_{ab}\}})\,
            P''(x_{m_{ab}} \mid X_a, X_{s_{ab}} = 0) \,dm_{ab}   \\
   &= \int_{x_{m_{ab}}} P''(X_b \mid x_{m_{ab}}, X_{\pa{\calD''}{b}\setminus \{m_{ab}\}})\,
            P''(x_{m_{ab}} \mid X_a)\frac{P''(X_{s_{ab}}=0\mid x_{m_{ab}},X_a)}{P''(X_{s_{ab}}=0 \mid X_a)}\, dm_{ab} 
\end{align*}
where in the third equality we used $m_{ab} \perp_{\calD''} \pa{\calD''}{b} \setminus \{a,m_{ab}\} \mid a,s_{ab}$ and $b\perp_{\calD''} s\mid \pa{\calD''}{b}$, and in the last equality we used Bayes' theorem.

Letting $Y:=V^\flat \dotcup M \dotcup S \setminus \{s,b\}$, this choice of kernels for $\calD$ gives the following joint distribution:
	\begin{align*}
    &P_{Q(X_{V^\sharp})}(x_{V^\sharp},x_{V^\flat},x_{M},X_{S},X_{s_{ab}}=0) \\
    &=Q(x_{V^\sharp}) 
    P''(X_{s} \mid X_{\pas{\calD}{s}})P''(X_{s_{ab}}=0 \mid X_{a^\sharp})
    P''(X_b \mid X_{\pas{\calD''}{b}\setminus\{m_{ab}\}},X_{s_{ab}}=0) \\
    & \qquad \times \prod_{y \in Y} P(x_y \mid x_{\pas{y}}) \\
    &=Q(x_{V^\sharp}) \,
    P''(X_{s} \mid X_{\pas{\calD}{s}})\prod_{y \in Y} P(x_y \mid x_{\pas{y}}) \\
    &
   \qquad\times \int_{x_{m_{ab}}} \hspace{-0mm}P''(X_b \mid x_{m_{ab}}, X_{\pa{\calD''}{b}\setminus \{m_{ab}\}})\,
            P''(x_{m_{ab}} \mid X_a)P''(X_{s_{ab}}=0\mid x_{m_{ab}},X_a)\,dm_{ab}  \\
    &= P''_{Q(X_{V^\sharp})}(x_{V^\sharp},x_{V^\flat},x_{M},X_{s_{ab}}=0,X_{S\setminus s_{ab}},X_{s_{ab}}=0).
	\end{align*}
By marginalizing $X_M$ and conditioning on $X_{S}=\zeros$, we then obtain
$P_{Q(X_{V^\sharp})}(X_{V^\sharp\dotcup V^\flat} \mid X_{S}=\zeros) = P''_{Q(X_{V^\sharp})}(X_{V^\sharp\dotcup V^\flat} \mid X_{S\dotcup \{s_{ab}\}}=\zeros)$, 
proving the result.
\end{proof}

\subsection{Proof of \Cref{prop_remove_redund}}
\label{app_proof_removeredund}

\removalrestatable*

\begin{proof} 
\textbf{Proof of a).}
This result would be immediate from 
Theorem 1 of 
\citet{ansanelli2024everything}
except that our smDGs have deterministic 
variables, so we will prove it from first principles.

\textbf{Proof that $\calC(\calD,V \mid S) \supseteq \calC(\calD_{-m_1},V \mid S)$:}
Follows from \Cref{le:subgraph}
and the fact that $\calD_{-m_1}$ is a subgraph of $\calD$.

\textbf{Proof that $\calC(\calD,V \mid S) \mathrel{\subseteq} \calC(\calD_{-m_1},V \mid S)$:}
Let $P_* \in \calC(\calD,V \mid S)$ 
be a set of SMI pairs compatible 
with $\calD$ and let 
$P(X_y \mid X_{\pa{\calD}{y}})$,
$y \in V \dotcup M \dotcup S$ be the 
kernels associated with $P_*$.
Then, consider the following choice of 
kernels for $\calD_{-m_1}$.

\begin{itemize}
\item Let $X_m$ in $\calD$ 
be a Cartesian product of independent 
variables $\bar{X}_{m_1} \times \bar{X}_{m_2}$, 
where each $\bar{X}_{m_i}$ 
is distributed in $\calD_{-{m_1}}$
like $X_{m_i}$ in $\calD$.
\item Each child of $m_i$ 
then depends on $\bar{X}_{m_i}$ in 
place of $X_{m_i}$, for $i=1,2$.
\item Every variable $y \not \in \{m_1,m_2\} \cup \ch{\calD}{m_1}$ has the same 
kernel in $\calD_{-m_1}$ as in $\calD$.
\end{itemize}
It is straightforward to verify that 
the selected interventional distributions 
then are identical in $\calD_{-m_1}$ and $\calD$.

\textbf{Proof of b).}	
\ding{212}	\textbf{Proof that $\calC(\calD,V \mid S) \supseteq \calC(\calD_{-s_1},V \mid S \setminus \{s_1\})$.} Follows from \cref{le:subgraph}, since $\calD'$ is a subgraph of $\calD$ with the same set of marginalized vertices and with selected vertices $ S \setminus \{s_1\}\subseteq S$.
	
\ding{212} \textbf{Proof that $\calC(\calD,V \mid S) \subseteq \calC(\calD_{-s_1},V \mid S \setminus \{s_1\})$.}	
	Let $P_*\in \calC(\calD, V \mid S)$ be a set of SMI pairs compatible with $\calD$, and let $P(X_{y} \mid X_{\paN_\calD({y})})$, $y\in W$ be the kernels associated with $P_*$.  Consider now the following choice of kernels for $\calD_{-s_1}$:
	\begin{itemize} 
    \item $P'(X_{s_2}=0 \mid X_{\pa{\calD_{-s_1}}{s_2}}) =P(X_{s_2}=0 \mid X_{\pa{\calD}{s_2}})\,P(X_{s_1}=0 \mid X_{\pa{\calD}{s_1}})$; this is possible because $\pa{\calD}{s_1}\subseteq\pa{\calD}{s_2}$.
\item $P'(X_{y} \mid X_{\pa{\calD_{-s_1}}{y}})=P(X_{y} \mid X_{\pa{\calD}{y}}) \quad \forall y\in V\dotcup M\dotcup S \setminus \{s_1,s_2\}$.
	\end{itemize}

Then:
\begin{align*}
	&P'_{Q(X_{V^\sharp})}(X_{V^\sharp\dotcup V^\flat\dotcup M \dotcup (S\setminus\{s_1,s_2\})},X_{s_2}=0)\\
    &=\prod_{y \in V^\flat \dotcup M \dotcup S \setminus \{s_1,s_2\}}  Q(X_{V^\sharp})\, P'\big(X_y \,\big|\, X_{\pas{\calD_{-s_1\!}}{y}}\big) \, P'\big(X_{s_2}=0 \,\big|\, X_{\pas{\calD_{-s_1}\!}{s_2}} \big) \\
	&= 
	\prod_{y \in V^\flat \dotcup M \dotcup S \setminus \{s_1,s_2\}} Q(X_{V^\sharp})\, P\big(X_y \,\big|\, X_{\pas{\calD}{y}}\big) \, P\big(X_{s_2}=0 \,\big|\, X_{\pas{\calD}{s_2}}\big)\, P\big(X_{s_1}=0 \,\big|\, X_{\pas{\calD}{s_1}}\big)\\
    &= P_{Q(X_{V^\sharp})}(X_{V^\sharp\dotcup V^\flat\dotcup M\dotcup (S\setminus s_1)},X_{s_1}=0,X_{s_2}=0).
\end{align*}

From this we obtain:
\begin{equation*}
	P'_{Q(X_{V^\sharp})}(X_{V_\sharp},X_{V^\flat}\cmid X_{S \setminus \{s_1\}}=0)=P_{Q(X_{V^\sharp})}(X_{V^\sharp\dotcup V^\flat}\cmid X_S=\zeros).
\end{equation*}
So $P'_*=P_*$, thus every set of SMI pairs compatible with $\calD$ can be reproduced by $\calD_{-s_1}$.
\end{proof}

\section{Algorithmic complexity of obtaining the selected-latent projection}
\label{app:algorithmic}
We will now establish that the selected-latent projection 
may be obtained in quartic time as a function of the number of vertices in the DAG.

\proposition{
\begin{equation*}
    \mathtt{RmvRedS}\circ\mathtt{RmvRedM}\circ \ToBidir \circ\texttt{SplitM->S}\circ\mathtt{MergeS}\circ\mathtt{MergeM}\circ\mathtt{Exog}\circ\mathtt{Term}(\calD) 
\end{equation*}
may be performed in time $O(n^4)$
where $n$ is the number of vertices in $\calD$.
}

\begin{proof}
$\mathtt{Term}$ only deletes outgoing edges from selected vertices, so it takes at most quadratic time in the number of vertices: $O(n^2)$.

$\mathtt{Exog}$ works by, for each marginalized vertex $m$, adding every edge from each parent of $m$ to each child of $m$. For each $m$, this is proportional to $|\operatorname{pa}(m)|\,|\operatorname{ch}(m)|$. Summed over all marginalized vertices, this is $O(n^3)$ in the worst case.

$\mathtt{MergeM}$ and $\mathtt{MergeS}$ need to examine each $M \to S$ edge, which is $O(n^2)$, and then perform at most one union operation for each { marginalized} (or selected) variable, approximately $O(n)$, so the worst-case complexity is $O(n^2)$.

$\texttt{SplitM->S}$: for each edge $m \to s$, of which there can be $O(n^2)$, we need to add an edge from $\operatorname{pa}(s)$ to $\operatorname{ch}(m)$, of which there can also be $O(n^2)$. Hence, the overall number of added edges is $O(n^4)$ in the worst case.

$\mathtt{ToSpecial}$ only needs to inspect visible-to-visible edges and check whether the parent has a selected child and the child has a marginalized parent, so its worst-case complexity is $O(n^2)$.

$\mathtt{RmvRedM}$ and $\mathtt{RmvRedS}$ can be done in $O(n^3)$, which can be seen as follows. There are three types of marginalized variables that still exist in the graph at this point:
\begin{enumerate}
    \item the original ``pure'' marginalized vertices in $\mathtt{MergeM}\circ \mathtt{Exog}\circ\mathtt{Term}(\calD)$, of which there are at most $n$;
    \item each $m_{ab}$ vertex added by $\texttt{SplitM->S}$ or $\mathtt{ToSpecial}$, which has a unique child $s_{ab}$ and so cannot be redundant.
\end{enumerate}
So we need to evaluate, for each pair of pure marginalized vertices, of which there are $n \times n$, whether their set of children have a subset relation, which is an $O(n)$-time operation, giving a total time of $O(n^3)$.

The overall time complexity in the worst case is therefore $O(n^4) + \text{lower-order terms} = O(n^4)$.
\end{proof}
\blk

\section{Proof that different SLP implies interventional inequivalence (\Cref{thm:diff-slp-implies-diff-model})}\label{app_diff_slp_diff_model}
Throughout this section,
we will use the following shorthand: for respective assignments $(x_{a_1},...,x_{a_N})$ to variables $X_{a_1},...,X_{a_N}$, 
we will write $[x_{a_1}x_{a_2}...x_{a_N}]_{a_1...a_N}$, 
e.g.\ $[01]_{ab}$ is an alternative notation for the assignment $X_a=0,X_b=1$.

Here, we will prove that if two DAGs have different selected-latent projections, then they can  be distinguished by the Observe\&Do probing scheme. We begin with the case where the selected-latent projections differ in their 
directed structure.

\begin{restatable}[Different self-loop implies distinguishable]{lemma}{DiffDirectedStructureOne} \label{le_diff_self_loop}
Let $\calD$ be a DAG with vertices partitioned as $(V, M, S)$
and $\calD'$ a DAG with vertices partitioned as $(V, M', S')$.
If the selected latent projection $\slp(\calD,V \mid S)$ 
contains a self-loop edge $(v,v)$ that is not in 
$\slp(\calD',V\mid S')$
then $\calD$ is not interventionally dominated by $\calD'$; that is, $\calC(\calD,V \mid S) \not \subseteq \calC(\calD',V\mid S')$.
\end{restatable}

\begin{proof}

Our approach will be to construct a model for $\calD$ that realizes a set of SMI pairs that cannot be reproduced by $\calD'$. In our model, the variables  $X_s$ and $X_v$ only take two possible values each, while the variable $X_m$ takes four possible values. Let $X_m=(X_m^1,X_m^2)$, where the variables  $X_m^1$ and $X_m^2$ are each binary.

The canonical DAG  of $\slp(\calD,V \mid S)$ contains a path $v \to s \gets m \to v$.
So let:
\begin{itemize}
    \item $X_m^1 \sim \operatorname{Bernoulli}(\frac{1}{2})$. 
    \item $X_m^2 \sim \operatorname{Bernoulli}(\frac{1}{2})$.
    \item $X_v = 0$ if and only if $X_m^1=X_m^2=0$. Otherwise, $X_v=1$.
    \item $X_s = 1$ if and only if $X_m^1=X_v=0$. Otherwise, $X_s=0$.
    \item Every variable in $X_{V\dotcup M\dotcup S\setminus \{m,v,s\}}$ ignores its parents, and takes the constant value $0$.
\end{itemize}

The observational distribution generated by this model gives $P(X_v=0\cmid X_s=0)=0$, while all other visible variables are fixed to the value $0$. 
If we perform an intervention that sets $X_{v^\sharp}$ to the value $0$, denoted $\doo(X_{v^\sharp}=0)$, then the resulting distribution over the natural variable $X_{v^\flat}$ gives $P_{\doo(X_{v^\sharp}=0)}(X_{v^\flat}=0 \mid X_{s}=0)=0$, while all of the other visible variables are still fixed to the value $0$.   If we perform an intervention that sets $X_{v^\sharp}$ to the value $1$, denoted $\doo(X_{v^\sharp}=1)$, then the resulting distribution over the natural variable $X_{v^\flat}$ gives $P_{\doo(X_{v^\sharp}=1)}(X_{v^\flat}=0 \mid X_{s}=0)=1/4$, while all of the other visible variables are still fixed to the value $0$. This is an example where the selected 
distribution over $v^\flat$ depends on the value of $v^\sharp$.\footnote{For a concrete illustrative scenario, see \cref{footnote_sec6}.}

Now, we will prove that this set of data cannot be reproduced by $\calD'$.
In the canonical DAG of $\slp(\calD',V\mid S')$, we will allow for the possibility that $v$ still has selected children and marginalized parents, but there are no edges from marginalized parents of $v$ to selected children of $v$. Let $\ch{v}_{S'}:=\ch{v}\cap S'$ denote the set of all selected children of $v$. When we perform an intervention $Q(X_{V^\sharp})$ that sets all of the variables in $X_{V^\sharp\setminus\{v^\sharp\}}$ to the value $0$ and does something arbitrary to $v^\sharp$, the resulting data gives:
\begin{align*}
    &P_{Q(X_{V^\sharp})}(X_{v^\sharp}X_{v^\flat}X_{M'}X_{\ch{v}_{S'}}=\zeros) \\
    &= \,Q(X_{v^\sharp})\,P(X_{v^\flat}\cmid X_{\pas{\calD'}{v}\cap V}=\zeros, X_{\pa{\calD'}{v}\cap{M'}}) \,P(X_{M'}) &\\ &\qquad\times P\big(X_{\ch{v}_{S'}}=\zeros \,\big|\, X_{v^\sharp}, X_{\pas{\calD'}{\ch{v}_{S'}}\cap V\setminus\{v\}}=\zeros,X_{\pa{\calD'}{\ch{v}_{S'}}\cap M'}\big).
\end{align*}
Since none of the variables of $\ch{v}_{S'}$ share any marginalized parent with $v$, by marginalizing $M'$ we get:
\begin{align*}
    &P_{Q(X_{V^\sharp})}(X_{v^\sharp}X_{v^\flat}, X_{\ch{v}_{S'}}=\zeros) \\
    &= Q(X_{v^\sharp}) \, P(X_{v^\flat}\cmid X_{\pas{\calD'}{v}\cap V}=\zeros)\,P(X_{\ch{v}_{S'}}=\zeros\cmid X_{v^\sharp}, X_{\pas{\calD'}{\ch{v}_{S'}}\cap V\setminus\{v\}}=\zeros).
\end{align*}
This implies: 
\begin{equation*}
    P_{Q(X_{V^\sharp})}(X_{v^\flat}\cmid X_{\ch{v}_{S'}}=\zeros)=P_{Q(X_{V^\sharp})}(X_{v^\flat})=P(X_{v^\flat}\cmid X_{\pas{\calD'}{v}\cap V}=\zeros),
\end{equation*}
which is independent of what is the specific intervention $Q(X_{v^\sharp})$ being performed. Therefore, $\calD'$ cannot reproduce the set of data described, because there the selected distribution over $v^\flat$ depends on $v^\sharp$.
\end{proof}

\begin{restatable}[Different directed edge implies distinguishable]{lemma}{DiffDirectedStructureTwo} \label{le_diff_directed_edge}
Let $\calD$ be a DAG with vertices $V \dotcup M \dotcup S$
and $\calD'$ a DAG with vertices $V \dotcup M' \dotcup S'$.
If the selected latent projections $\calG=\slp(\calD,V \mid S)$ 
and $\calG'=\slp(\calD',V\mid S')$
contain the same self-loops
but $\calG$ contains an edge $a\to b$
that is not in $\calG'$
then $\calD$ is not interventionally dominated by $\calD'$; that is, $\calC(\calD,V \mid S) \not \subseteq \calC(\calD',V\mid S')$.
\end{restatable}

\begin{proof}
We will specify a set of SMI pairs in $\calC(\calD,V \mid S)$
that is not in $\calC(\calD',V \mid S)$.
This distribution will produce the following data:
\begin{subequations}
    \label{eq_everything_proof1}
	\begin{align}
		&P(X_{V\setminus\{a\}}\cmid X_S=\zeros,\doo(X_a=x_a))=[x_a]_{b} \cdot [0]_{V\setminus\{a,b\}} \text{ for } x_a \in \{0,1\}\label{eq_everything_proof1b} \\
		&P(X_{V\setminus\{b\}}\cmid X_S=\zeros,\doo(X_b=1))=[0]_{V\setminus \{b\}}. \label{eq_everything_proof1c}
	\end{align}
\end{subequations}	
That is, when we intervene $X_a$ to $0$ or $1$, then $X_b$ copies it, 
but when we intervene $X_b$ to $1$, then $X_a$ is $0$.
Meanwhile, every other visible variable is fixed to a point distribution on $0$. 
This data is clearly jointly realizable by $\can(\calG)$: if the edge $a\to b$ in $\calG$ corresponds to a regular edge $a\to b$ in the canonical DAG, then we can satisfy both conditions above by choosing the kernels $P(X_a\cmid X_{\pa{a}})=[0]_a$ and $P(X_b\cmid X_{\pa{b}})=X_a$. If the edge $a\to b$ in $\calG$ corresponds to a special edge $a\to s_{ab}\gets m_{ab}\to b$ in the canonical DAG, then we can satisfy both conditions above by choosing the kernels $P(X_a\cmid X_{\pa{a}})=[0]_a$, $P(X_b\cmid X_{\pa{b}})=X_{m_{ab}}$, $P(X_{m_{ab}})\sim \operatorname{Unif}(\dom{a})$, and $X_{s_{ab}}=0$ if and only if $X_a=X_{m_{ab}}$. 

We will now prove that the set of data of \eqref{eq_everything_proof1b} and \eqref{eq_everything_proof1c} is \emph{not} jointly realizable by $\can(\calG')$. 

First notice that, since all of the other visible variables are fixed to a point distribution, the only causal paths between $a$ and $b$ that are open in $\can(\calG')$ are through marginalized parents and/or selected children. Therefore, it is sufficient to look at the subgraph over $a$, $b$, their marginalized parents and their selected children. If the SMI distributions over $X_a$ and $X_b$ of \eqref{eq_everything_proof1} are not jointly realizable by this subgraph, then \eqref{eq_everything_proof1} cannot be realized by the full graph. \Cref{fig_subgraph1} shows the best case scenario for this subgraph, where $a$ and $b$ share both a selected child and a marginalized parent, 
there is a path $b \to s_4 \gets m_4 \to a$, 
and both variables contain self-loops. As promised, the smDG associated with the DAG of \Cref{fig_subgraph1} contains no edge $a\to b$. 
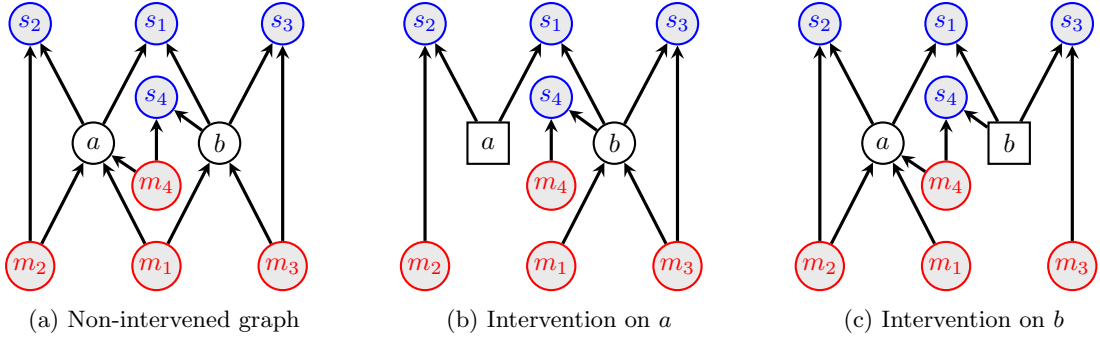
\begin{figure}[h!]
    \centering
    \begin{subfigure}{0.28\textwidth}
    \begin{tikzpicture}[>=stealth]
    \node[rv] (a) at (0,0) {$a$};
    \node[rv, right=11mm of a] (b) {$b$};
    \node[mv, below=13mm of a] (m1) at ($(a)!0.5!(b)$) {$m_1$};
    \node[sv, above=13mm of a] (s1) at ($(a)!0.5!(b)$) {$s_1$};
    \node[mv, below = 13mm of a] (m2) at ($(a)!-0.5!(b)$) {$m_2$};
    \node[sv, above = 13mm of a] (s2) at ($(a)!-0.5!(b)$) {$s_2$};
    \node[sv, above =13mm of a] (s3) at ($(b)!-0.5!(a)$) {$s_3$};
    \node[mv, below =13mm of a] (m3) at ($(b)!-0.5!(a)$) {$m_3$};
    \node[sv, below =4mm of s1] (s4) {$s_4$};
    \node[mv, above =4mm of m1] (m4) {$m_4$};

    \draw[->, very thick] (m1) -- (a);
    \draw[->, very thick] (m1) -- (b);
    \draw[->, very thick] (m2) -- (a);
    \draw[->, very thick] (m2) -- (s2);
    \draw[->, very thick] (m3) -- (b);
    \draw[->, very thick] (m3) -- (s3);
    \draw[->, very thick] (m4) -- (s4);
    \draw[->, very thick] (m4) -- (a);
    
    \draw[->, very thick] (a) -- (s1);
    \draw[->, very thick] (a) -- (s2);
    
    \draw[->, very thick] (b) -- (s1);
    \draw[->, very thick] (b) -- (s3);
    \draw[->, very thick] (b) -- (s4);
\end{tikzpicture}
\caption{Non-intervened graph} \label{fig_subgraph1}
\end{subfigure}\hspace{9mm}
    \begin{subfigure}{0.28\textwidth}
    \begin{tikzpicture}[>=stealth]
    \node[rv,rectangle] (a) at (0,0) {$a$};
    \node[rv, right=11mm of a] (b) {$b$};
    \node[mv, below=13mm of a] (m1) at ($(a)!0.5!(b)$) {$m_1$};
    \node[sv, above=13mm of a] (s1) at ($(a)!0.5!(b)$) {$s_1$};
    \node[mv, below = 13mm of a] (m2) at ($(a)!-0.5!(b)$) {$m_2$};
    \node[sv, above = 13mm of a] (s2) at ($(a)!-0.5!(b)$) {$s_2$};
    \node[sv, above =13mm of a] (s3) at ($(b)!-0.5!(a)$) {$s_3$};
    \node[mv, below =13mm of a] (m3) at ($(b)!-0.5!(a)$) {$m_3$};
    \node[sv, below =4mm of s1] (s4) {$s_4$};
    \node[mv, above =4mm of m1] (m4) {$m_4$};

    \draw[->, very thick] (m1) -- (b);
    \draw[->, very thick] (m2) -- (s2);
    \draw[->, very thick] (m3) -- (b);
    \draw[->, very thick] (m3) -- (s3);
    \draw[->, very thick] (m4) -- (s4);
    
    \draw[->, very thick] (a) -- (s1);
    \draw[->, very thick] (a) -- (s2);
    
    \draw[->, very thick] (b) -- (s1);
    \draw[->, very thick] (b) -- (s3);
    \draw[->, very thick] (b) -- (s4);
\end{tikzpicture}
\caption{Intervention on $a$} \label{fig_subgraph2}
\end{subfigure}\hspace{9mm}
    \begin{subfigure}{0.28\textwidth}
    \begin{tikzpicture}[>=stealth]
    \node[rv] (a) at (0,0) {$a$};
    \node[rv, right=11mm of a,rectangle] (b) {$b$};
    \node[mv, below=13mm of a] (m1) at ($(a)!0.5!(b)$) {$m_1$};
    \node[sv, above=13mm of a] (s1) at ($(a)!0.5!(b)$) {$s_1$};
    \node[mv, below = 13mm of a] (m2) at ($(a)!-0.5!(b)$) {$m_2$};
    \node[sv, above = 13mm of a] (s2) at ($(a)!-0.5!(b)$) {$s_2$};
    \node[sv, above =13mm of a] (s3) at ($(b)!-0.5!(a)$) {$s_3$};
    \node[mv, below =13mm of a] (m3) at ($(b)!-0.5!(a)$) {$m_3$};
    \node[sv, below =4mm of s1] (s4) {$s_4$};
    \node[mv, above =4mm of m1] (m4) {$m_4$};

    \draw[->, very thick] (m1) -- (a);
    \draw[->, very thick] (m2) -- (a);
    \draw[->, very thick] (m2) -- (s2);
    \draw[->, very thick] (m3) -- (s3);
    \draw[->, very thick] (m4) -- (s4);
    \draw[->, very thick] (m4) -- (a);
    
    \draw[->, very thick] (a) -- (s1);
    \draw[->, very thick] (a) -- (s2);
    
    \draw[->, very thick] (b) -- (s1);
    \draw[->, very thick] (b) -- (s3);
    \draw[->, very thick] (b) -- (s4);
\end{tikzpicture}
\caption{Intervention on $b$} \label{fig_subgraph3}
\end{subfigure}
    \caption{The canonical DAG of a two-variable graph before and after intervention. (a) Canonical DAG of the subgraph of $\calG'$ over $a$ and $b$. (b) Graph that represents the data realized by (a) after an intervention on $a$ (when we omit the natural value $X_{a^\flat}$). (c) Graph that represents the data realized by (a) after an intervention on $b$ (when we omit the natural value $X_{b^\flat}$).}
    \label{fig_subgraph}
\end{figure}

An interventional distribution $P(X_{b}\cmid X_S=\zeros,\doo(X_a=x_a)$ is realizable by  \Cref{fig_subgraph1} if it can be factorized as in \Cref{fig_subgraph2}:
\begin{align}
    &P(X_{b}\cmid X_S=\zeros,\doo(x_a)) = \frac{P(X_{b},X_S=\zeros \mid \doo(x_a))}{P(X_S=\zeros \mid \doo(x_a))} \nonumber\nonumber\\ 
    &=\frac{P(X_b,X_{s_3}=0)\,P(X_{s_1}=0\cmid x_{a},X_{b})\,P(X_{s_2}=0 \mid x_{a})\, P(X_{s_{4}}=0\mid X_{b})}{P(X_S=\zeros \mid \doo(x_{a}))}. 
     \label{eq_everything_factorization1b}
\end{align}
We know from \eqref{eq_everything_proof1b} that $P(X_b=1 \mid X_S=\zeros,\doo(X_a=1))>0$, so from \eqref{eq_everything_factorization1b}, it follows that:
i) $P(X_b=1,X_{s_3}=0)>0$, and
ii) $P(X_{s_4}=0 \mid X_b=1)>0$.
Furthermore, from \eqref{eq_everything_proof1b}, $P(X_b=0 \mid X_S=\zeros,\doo(X_a=0))>0$, so from \eqref{eq_everything_factorization1b}, it follows that:
iii) $P(X_{s_2}=0 \mid x_a=0)>0$.

From \eqref{eq_everything_proof1b}, we also have $P(X_{b}=1\cmid X_S=\zeros,\doo(X_a=0))=0$. Together with facts (i-iii), \eqref{eq_everything_factorization1b} then implies that we must have $P(X_{s_1}=0\cmid X_a=0,X_b=1)=0$.

Note that $P(X_{s_1}=0\cmid X_a=0,X_b=1)$ is also going to appear as one of the kernels of the decomposition of a probability $P(X_{a}=0\cmid X_S=\zeros,\doo(X_b=1)$ that is realizable by  \Cref{fig_subgraph1}. Therefore, we conclude that $P(X_{a}=0\cmid X_S=\zeros,\doo(X_b=1)=0$, thus contradicting \eqref{eq_everything_proof1c}. 
We conclude that the observation \eqref{eq_everything_proof1} is \emph{not} jointly realizable by $\can(\calG')$.
\end{proof}

So far, we have proven that two liftable smDGs that differ in their directed structure can necessarily be distinguished by the Observe\&Do probing scheme. Now, we will look at liftable smDGs that share the same directed structure, but are different in their marginalized or selected independence system.

\begin{restatable}[Different marginalized independence system implies distinguishable]{lemma}{DiffMarginalStructure} \label{le_diff_marginal_structure}
Let $\calD$ be a DAG with vertices $V \dotcup M \dotcup S$
and $\calD'$ a DAG with vertices $V \dotcup M' \dotcup S'$.
If the selected latent projections $\calG=\slp(\calD,V \mid S)$ and $\calG'=\slp(\calD',V\mid S')$
have the same directed structure, but
the marginal independence system of $\calG$ 
contains a maximal face $V_m$ that is not in $\calG'$,
then $\calD$ is not interventionally dominated by $\calD'$; that is, $\calC(\calD,V \mid S) \not \subseteq \calC(\calD',V\mid S')$.
\end{restatable}

\begin{proof}
Since $V_m$ is a maximal face in $\calG$ but not $\calG'$,
the vertices of $V_m$ have a shared marginalized parent in $\can(\calG)$, but not in $\can(\calG')$. 

We consider two cases:
(a) $V_m=\{a\}$, and
(b) $\lvert V_m\rvert>1$.
In case (a), $\can(\calG)$ has a pattern $m \to a$ that is not present in $\can(\calG')$.
In $\calG$, it is clearly possible to observe $X_{a} \mid X_S=0 \sim \operatorname{Bernoulli}(\frac{1}{2})$. In $\calG'$, $a$ cannot be in any element of the marginal independence system, 
so $X_a$ is a deterministic function of its parents. Thus, $X_a$ cannot be Bernoulli if other variables are held constant, 
proving the result for this case.

The rest of our proof covers  case (b).

We obtain data in the following ways:
(i) passive observation of all visible variables; and (ii) for each element $T\in 2^{V_m}$ of the powerset of $V_m$, a do-intervention on $X_T$, and passive observation on the remaining visible variables. Suppose that the data obtained from these interventions is:
	 \begin{subequations}
  \label{eq_everything_proof2}
	 	\begin{align}
	 		\label{eq_everything_proof2a}
	 		&\text{(i)}&&	    P(X_{V} \mid X_{S_\DAG}=0)=\left\}p[0,\dots,0]_{V_m}+(1-p)[1,\dots,1]_{V_m}\right\} [0]_{V\setminus {V_m}} \\
	 		&\text{(ii)} &&	P(X_{V\setminus T}\cmid X_{S_\DAG}=\zeros,\doo(X_{T}))=\left\{p[0,\dots,0]_{{V_m}\setminus T}+(1-p)[1,\dots,1]_{{V_m}\setminus T}\right\} [0]_{V\setminus {V_m}} \; \nonumber \\ &&&\forall T\in 2^{V_m}
	 		\label{eq_everything_proof2b}
	 	\end{align}
	 \end{subequations}
where $p \in (0,1)$
and $S_\DAG$ is the set of selected vertices of the DAG in consideration, be it $\can(\calG)$ or $\can(\calG')$. The two equations above say that all of the variables in $X_{V_m}$ are perfectly correlated in the data obtained from passive observations, and performing do-interventions on a subset $X_T\subseteq X_{V_m}$ does not change the marginals over the remaining variables of $V_m$. Meanwhile, all of the visible variables in $V\setminus {V_m}$ are fixed to a point distribution on $0$. This data is jointly realizable by $\can(\calG)$: if the correlation between the variables of $X_{V_m}$ that is presented in \eqref{eq_everything_proof2a} is established through the shared marginalized parent, then interventions will not change the marginal distributions. We will now show that this set of data is \emph{not} jointly realizable by $\can(\calG')$. 

This proof will use and extend the proof of a similar result \citep[Lemma 8]{ansanelli2024everything}. 

From \eqref{eq_everything_proof2b}, 
when we perform do-interventions on every vertex except for one, i.e.~$T={V_m}\setminus\{t\}$, 
the remaining vertex $t$ has full support. Therefore, for any $x_t\in\dom{t}$ and $x_T\in\dom{T}$:
\begin{align}
0&<P(X_{t}=x_t\cmid X_{S'}=\zeros,\doo(X_{T}=x_T))& \nonumber\\
&= \frac{P(X_{S'}=\zeros\cmid X_{t}=x_t,\doo(X_{T}=x_T))\,P(X_{t}=x_t\cmid \doo(X_{T}=x_T))}{P(X_{S'}=\zeros\cmid \doo(X_{T}))}
\label{aab}
\end{align}
where we have used Bayes' theorem. This implies that the selected variables equal $\zeros$ with strictly positive probability for 
any possible assignment to $X_{V_m}$.

Bayes' theorem also tells us that
\begin{equation}
P(X_{V_m},X_{V \setminus {V_m}}\!=\!0\cmid X_{S_\DAG}\!=\!\zeros)=P(X_{V_m}\cmid X_{S_\DAG}\!=\!\zeros)=\frac{P(X_{S_\DAG}=\zeros\cmid X_{V_m})}{P(X_{S_\DAG}=\zeros)}P(X_{V_m}).
\label{aaa}
\end{equation}
Also, \eqref{eq_everything_proof2a} implies $P(X_{S_\DAG}=\zeros)>0$, 
and the reasoning after \eqref{aab} implies $P(X_{S_\DAG=\zeros}\mid X_{V_m})>0$. Therefore:
\begin{equation}
P(X_{V_m},X_{V \setminus {V_m}}=0\cmid X_{S_\DAG}=\zeros) >0 \iff
P(X_{V_m})>0. \label{aa-equiv}
\end{equation}
This implies that, when the interventional distributions follow \eqref{eq_everything_proof2b}, perfect correlation between the variables of $X_m$ in the observational distribution after selection \eqref{eq_everything_proof2a} is possible if and only if  perfect correlation between the variables of $X_m$ in the observational distribution is possible without selection. 

\ryan{I commented out the paragraph above --- I think the logic still makes sense without it.}

Therefore, we can now appeal to the same argument from Lemma 8 of \citet{ansanelli2024everything}: without selection, the only way for a set of vertices to be perfectly correlated is if they share a common ancestor~\citep{Steudel_Ay}, which will be the source of the perfect correlation. In this case, the common ancestor cannot be in $V\setminus V_m$, because all such variables are fixed to $0$. Furthermore, it also cannot be in $V_m$, otherwise the remaining variables of $V_m$ could not continue being perfectly correlated after a do-intervention on this common ancestor,  contradicting \eqref{eq_everything_proof2b}. Therefore, such a common ancestor must be a common marginalized parent, which is \emph{not} present in $\can(\calG')$.   
\end{proof}

Now, we prove that smDGs that only differ in their selected independece system are Observe\&Do distinguishable.

\begin{restatable}[Different selected independence system implies distinguishable]{lemma}{DiffSelectedStructure} \label{le_diff_selected_structure}
Let $\calD$ be a DAG with vertices $V \dotcup M \dotcup S$
and $\calD'$ a DAG with vertices $V \dotcup M' \dotcup S'$.
If the selected latent projections $\calG=\slp(\calD,V \mid S)$ and $\calG'=\slp(\calD',V\mid S')$
have the same directed structure and marginal independence system
but 
the selected independence system of $\calG$
contains a maximal face $V_s$ that is not in $\calG'$,
then $\calD$ is not interventionally dominated by $\calD'$; that is, $\calC(\calD,V \mid S) \not \subseteq \calC(\calD',V\mid S')$
\end{restatable}

\begin{proof}

 By assumption, the vertices of ${V_s}$ have a shared selected child in $\can(\calG)$, but not in $\can(\calG')$.

We consider two cases:
(a) ${V_s}=\{b\}$, and
(b) $\lvert {V_s}\rvert>1$.
In case (a), $\can(\calG)$ has a pattern $b \to s$ that is not present in $\can(\calG')$.
In $\can(\calG)$, set every variable in the graph to a fixed value, 
set $X_s=X_b$, 
and apply the soft intervention $Q(X_{b^\sharp}) \sim \operatorname{Bernoulli}(\frac{1}{2})$.
We will see that the intervened distribution $P(X_{b^\sharp} \mid X_s=0)$
is a point distribution on $0$, and thus $P(X_{b^\sharp} \mid X_s=0)\neq Q(X_{b^\sharp})$, 
something that cannot happen in a model on $\can(\calG)$.

The rest of our proof addresses case (b).

We obtain data from an intervention $Q_\text{un}(X_{V^\sharp})$ that sets each variable of $V_s^\sharp$ to an independent uniform distribution $\operatorname{Bernoulli}(\frac{1}{2})$, and each variable of $V^\sharp\setminus V_s^\sharp$ to the fixed value $0$. Let the resulting distribution be:
\begin{equation}
	P_{Q_\text{un}(X_{V^\sharp})}(X_{V^\sharp}\cmid X_{S_\DAG}=\zeros)=\left(\frac{1}{2^{|{V_s}|-1}}\sum_{x_{V_s}}(1-XOR(x_{V_s}))[x_{V_s}]_{V_s^\sharp}\right) [0]_{V^\sharp\setminus{V_s^\sharp}} \label{eq_everything_proof3a}
\end{equation}
where $XOR(x_{V_s})$ is the parity of the sum of values in $x_{V_s}$ ($0$ if even and $1$ if odd), and $S_\DAG$ is again the set of selected vertices of the DAG in consideration. In \eqref{eq_everything_proof3a}, all of the visible variables that are not in $X_{V_s}$ are set to a point distribution on $0$, and the distribution over $X_{V_s}$ has support on all events with even parity. For example, when ${V_s}$ has three vertices, \eqref{eq_everything_proof3a} is $P_{Q_\text{un}(X_{V^\sharp})}(X_{V^\sharp}\cmid X_{S_\DAG}=\zeros)=\frac{1}{4}([000]_{V_s^\sharp}+[110]_{V_s^\sharp}+[101]_{V_s^\sharp}+[011]_{V_s^\sharp})[0]_{V^\sharp\setminus{V_s^\sharp}}$. 
This data is realizable by $\can(\calG)$, because the selected vertex $s\in S$ that is a shared child of ${V_s}$ can be such that $X_S=\zeros$ if and only if the parity of the sum of variables $X_{V_s}$ is even. We now show that this data is \emph{not} realizable by $\can(\calG')$.

Note that, from \eqref{eq_everything_proof3a}, every strict subset $X_A\subset X_{V_s}$ has full support in $P_{Q_\text{un}(X_{V^\sharp})}(X_{V^\sharp}\cmid X_{S_\DAG}=\zeros)$. If \eqref{eq_everything_proof3a} is realizable by $\can(\calG')$, this implies that every individual selected vertex$s\in S_{\can(\calG')}$ is such that
\begin{equation}
    P(X_s=0\cmid X_{A^\sharp}=x_A)>0 \quad \forall x_A\in \dom{A}, \forall A\subset V_s.
\end{equation}

By assumption, none of the selected vertices $s\in S_{\can(\calG')}$ is a child of every vertex of $V_s$. Thus, the following is a special case of the equation above:
\begin{equation}
    P(X_s=0\cmid X_{\pas{s}\cap V^\sharp_s}=x_{\pa{s}\cap V_s})>0 \quad \forall x_{\pa{s}\cap V_s}\in \dom{\pa{s}\cap V_s}.
\end{equation}

Now, consider the remaining variables of $V_s^\sharp$, i.e.~$X_{V_s^\sharp\setminus\pas{s}}$.  We will  show that in this distribution, these variables are conditionally independent of $X_s$ given $X_{\pas{s}\cap V^\sharp_s}$. Let $S_F\subseteq S_{\can(\calG')}$ be the subset of selected vertices of $\can(\calG')$ that correspond to \emph{facets} of $\calG'$, that is, selected vertices that have only visible parents. The remaining selected vertices $ S_{\can(\calG')}\setminus S_F$ are inside special edges, and have one marginalized parent and one visible parent. (i) When $s\in S_F$: if there is a chain with a visible mediary vertexbetween a variable of $X_{V_s^\sharp\setminus\pas{s}}$ and $s$ or a fork with a visible common cause between a variable of $X_{V_s^\sharp\setminus\pas{s}}$ and $s$, then this visible mediary or common cause is either in $\pas{s}\cap V^\sharp_s$, and the path is blocked, or it is in $V^\sharp\setminus V_s^\sharp$, and it cannot establish correlation because all of the variables of $X_{V^\sharp\setminus V_s^\sharp}$ are fixed to the value $0$. (ii) When $s\in S_{\can(\calG')}\setminus S_F$: in this case, $s$ is part of a special edge $a\to s\gets m\to b$. Since we intervene on all variables of $X_{V_s^\sharp}$ and all of the remaining visible variables are fixed to $0$, the causal connection $m\to b$ is broken regardless of whether $b$ is in $V_s$ or not. Therefore, $s$ behaves as a selected vertexthat has only one parent ($a$), and thus this case reduces back to case (i).

This conditional independence, together with the fact that  $X_{V^\sharp\setminus V_s^\sharp}$ are fixed to the value $0$, implies that:
\begin{equation}
    P\big(X_s=0\,\big|\,X_{V^\sharp_s}=x_{ V_s},X_{V^\sharp\setminus V_s^\sharp}=\zeros \big) > 0 \qquad \forall x_{ V_s}\in \dom{ V_s}.
    \label{ineq_app}
\end{equation}
Then:
\begin{align*}
&P_{Q_\text{un}}\big(X_{V_s^\sharp}=x_{V_s}, X_{V^\sharp\setminus V_s^\sharp}=\zeros,X_{S_{\can(\calG')}}=\zeros\big)\nonumber\\&=Q_\text{un}(X_{V\sharp})\prod_{s\in S_{\can(\calG')}} P\big(X_{s} \,\big|\, X_{V_s^\sharp}=x_{V_s}, X_{V^\sharp\setminus V_s^\sharp}=\zeros\big) & \\
&= \frac{1}{2^{|{V_s}|-1}}\prod_{s\in S_{\can(\calG')}} P\big(X_{s} \,\big|\, X_{V_s^\sharp}=x_{V_s}, X_{V^\sharp\setminus V_s^\sharp}=\zeros\big) & \text{($Q_\text{un}$ is uniform)} \\
&>0 \text{ for all }x_{V_s}. & \text{by }\eqref{ineq_app}
\end{align*}
But this implies that $P_{Q_\text{un}}\big(X_{V_s^\sharp}=x_{V_s}, X_{V^\sharp\setminus V_s^\sharp}=\zeros\mid X_{S_{\can(\calG')}}=\zeros\big) > 0$
for all assignments to $x_{V_s}$, which
contradicts \eqref{eq_everything_proof3a}, proving the result.
\end{proof}

Note that the proofs of \cref{le_diff_directed_edge,le_diff_marginal_structure,le_diff_selected_structure} do \emph{not} simultaneously involve an intervention on a vertex and an observation of the natural value of the variable associated with the same vertex; that is, the expressions \eqref{eq_everything_proof1}, \eqref{eq_everything_proof2}, and \eqref{eq_everything_proof3a} never simultaneously involve the $\flat$ and $\sharp$ version of a variable. Therefore, these proofs also show distinguishability of causal structures under a weaker probing scheme that will be defined in \cref{app_obs_or_do}, called \emph{Observe-or-Do probing scheme}. The same \emph{does not} hold for \cref{le_diff_self_loop}.

We are now in a position to prove the overall result below.

\smgraphSMIEquivalence*

\begin{proof}
The ``if'' side follows from 
\Cref{prop_exog_term,prop_merging,prop:splitarrows,lemma_interchangeedges,prop_remove_redund}.
The ``only if'' direction follows from 
\Cref{le_diff_self_loop,le_diff_directed_edge,le_diff_marginal_structure,le_diff_selected_structure}, 
because differences in the self-loops, directed edges, marginal structure, and selected structure 
exhaust all of the possible ways that two smDGs with the same vertices can differ.
\end{proof}

\section{A weaker probing scheme: Observe-or-do} \label{app_obs_or_do}
To begin with, we define the Observe-or-do probing scheme, 
generalizing Equation 24 of \citet{ansanelli2024everything}, 
to include selection effects.

\begin{definition}[SMI-Observe-Or-Do model]
Let $\calD$ be a DAG with vertices partitioned into $W=V\dotcup M\dotcup S$.
Consider the set of pairs
$P_{*}=\big\{\big(Q(X_Z),P_{Q(X_Z)}(X_V \mid X_S=\zeros)\big):Z \subseteq V,Q(X_Z) \in \calP(\dom{Z})\big\}$
where $\calP(\dom{Z})$ is the set of probability measures on the domain of $X_Z$.

We say that $P_*$ is a set of SMI-Observe-Or-Do pairs compatible with $\calD$ 
if each $\big(Q(X_Z),P_{Q(X_Z)}(X_V \mid X_S=\zeros)\big) \in P_{*}$ satisfies:
    \begin{align*}
    &P_{Q(X_Z)}(X_V \mid X_S=\zeros) = \frac{\int_{x_M} P_{Q(X_Z)}(X_V,X_M,X_S=\zeros) \, dx_{M}}{
    \int_{x_V} \int_{x_M} P_{Q(X_Z)}(X_V,X_M,X_S=\zeros) \, dx_{M} \, dx_{V}
    } \\
    &\text{for} \quad P_{Q(X_Z)}(X_V,X_M,X_S) = Q(X_Z) \prod_{y \in V \setminus Z} P(X_{y} \mid X_{\pa{y}}) 
\prod_{y \in M \dotcup S} P(X_{y} \mid X_{\pa{y}}).
    \end{align*}
The collection of all sets of  SMI-Observe-Or-Do pairs is called the SMI-Observe-Or-Do model.
\end{definition}

Note that in this definition each visible vertex is \emph{either} intervened upon (in which case it belongs to $Z$) \emph{or} observed (in which case it belongs to $V\setminus Z$). Contrast this with \cref{def_smi_pair,def_smi_model} for the Observe\&Do probing scheme.

\begin{proposition}
\label{prop_app_figs}
The DAGs \Cref{fig_obs_or_do1} and \Cref{fig_obs_or_do2} 
are not distinguishable by the Observe-or-do probing scheme.
\end{proposition}

\begin{proof}
Let $\calD$ be the graph $m \to v \to s$; $m\to s$ from \Cref{fig_obs_or_do1}, 
and $\calD'$ be the graph $m \to v \to s$ from \Cref{fig_obs_or_do2}. $\calD'$ is a subgraph of $\calD$, so clearly the SMI-Obs-Or-Do model of 
$\calD$ dominates that of $\calD'$. For the remainder of the proof, we will establish that the opposite is also true.

\blk

Let $P_*$ be a set of SMI-Obs-Or-Do  pairs in the SMI-Obs-Or-do model of $\calD$, that is generated from kernels $P(X_s\cmid X_m,X_v)$, $P(X_m)$ and $P(X_v\cmid X_m)$. Make the following choice of kernels for $\calD'$:
\begin{itemize}
\item $P'(X_m)=P(X_m)$,
    \item $P'(X_s=0\cmid X_v) = \int_{x_m} P(X_s=0\cmid x_m,X_v)\,P(x_m)\,dx_m$,
    \item $P'(X_v\cmid X_m) \\ =\begin{cases} c \frac{P(X_v,X_s=0\cmid X_m)}
    {P'(X_s=0\cmid X_v)}=c\frac{ P(X_v\mid X_m)P(X_s=0\cmid X_m,X_v)}{\int_{x_m} P(X_s=0\cmid x_m,X_v)\,P(x_m)\,dx_m} & \text{if }P'(X_s=0\cmid X_v)\neq 0
    \\
    0 & \text{if }P'(X_s=0\cmid X_v)= 0,\end{cases}$
\end{itemize}
where the constant of proportionality $c$ is chosen so that $\int_{x_v}P'(x_v\cmid X_m) \,dx_v=1$.

Now, we will show that this choice of kernels can reproduce $P_*$ in $\calD'$. First, we show that it gives rise to all of the correct pairs $\big(Q(X_Z),P_{Q(X_Z)}(X_V \mid X_S=\zeros)\big)$ when $Z=\{v\}$, that is, when we perform an intervention $Q(X_v)$ on $v$. 
\begin{align*}
P'_{Q(X_v)}(X_v\cmid X_s=0) &= \frac{P'_{Q(X_v)}(X_v,X_s=0)}{P'_{Q(X_v)}(X_s=0)} = \frac{Q(X_v)\,P'(X_s=0\cmid X_v)}{\int_{x'_v}Q(x'_v)\,P'(X_s=0\cmid X_v=x'_v) \, dx'_{v}} \\
&= \frac{Q(X_v)\int_{x_m}P(X_s=0\cmid x_m,X_v)\,P(x_m) \,dx_{m}}{\int_{x'_v}Q(x'_v)\int_{x'_m}P(X_s=0\cmid x'_m,x'_v)\,P(x'_m) \,dx'_{m}\, dx'_{v}} \\
&= \frac{P_{Q(X_v)}(X_v,X_s=0)}{P_{Q(X_v)}(X_s=0)} = P_{Q(X_v)}(X_v\cmid X_s=0).
\end{align*}
Therefore, our choice of kernels for $\calD'$ can reproduce all of the pairs $\big(Q(X_Z),P_{Q(X_Z)}(X_V \cmid X_S=\zeros)\big)\in P_*$ where $Z=\{v\}$. Now, we look at the pair where $Z=\emptyset$, that is, when we make a passive observation of $v$.
\begin{align*}
P'(X_v\cmid X_s=0) &= \frac{P'(X_v,X_s=0)}{P'(X_s=0)} = \frac{\int_{x_m}P'(X_s=0\cmid X_v)\,P'(X_v\cmid x_m)\,P(x_m)\,dx_m}{\int_{x'_v}\int_{x'_m}P'(X_s=0\cmid x'_v)\,P'(x'_v\cmid x'_m)\,P(x'_m)\,dx'_m\,dx'_v} \\
&= \frac{P'(X_s=0\cmid X_v)\frac{P(X_v,X_s=0)}{P'(X_s=0\cmid X_v)}}{\int_{x'_v}P'(X_s=0\cmid X_v=x_v')\frac{P(x'_v,X_s=0)}{P'(X_s=0\cmid X_v=x_v')}\,dx'_{v}}= \frac{P(X_v,X_s=0)}{\int_{x'_v} P(x'_v,X_s=0)\,dx'_{v}} \\
&= \frac{P(X_v,X_s=0)}{P(X_s=0)} = P(x_v\cmid X_s=0).
\end{align*}
Therefore, our choice of kernels for $\calD'$ reproduces the entire set of SMI pairs $P_*$.
\end{proof}

\section{Proofs of rules to show observational equivalence}

\subsection{Proofs for \cref{sec_added_marg_node}}
\label{app_selected_to_marginal}

We begin by restating \Cref{prop:selected-to-marginal}.

\SelectedToMarginalObsEq*

\begin{proof}
This proof will use the idea of \emph{district factorization} \citep{evans2018margins}. For the definition of districts and district factorization, we refer the reader to that paper.

Since every marginalized face $V_m\in \calL$ such that $V_s\cap V_m\neq \emptyset$ is also such that $V_m\subseteq V_s$, the vertices of $V_s$ are \emph{not} part of the same district  as any  vertex outside of $V_s$ in $\calG$. Thus, using district factorization together   the fact that $\pa{\calG}{V_s}\subseteq V_s$, we obtain:
\begin{equation}\label{eq_distr_fact}
    P(X_V\mid X_S=0)=P(X_{V_s}\mid X_S=0)\,P(X_{V\setminus V_s}\mid X_{V_s}, X_S=0).
\end{equation}

By assumption, all of the vertices $v\in V_s$ have access to local randomness (because each one belongs to some marginalized face in $\calG$). Together with the fact that all of them share one selected child $s\in S$, this implies that the DAG $\can(\calG)$ can realize \emph{any} conditional distribution $P(X_{V_s}\cmid X_S=0)$, by simply setting the variables in $X_{V_s}$ to be  uniformly distributed and choosing the selection condition on $X_s$ according to the desired $P(X_{V_s}\mid X_S=0)$.

In $\calG'$, it is also true that the vertices of $V_s$ are not in the same district as any vertex outside of $V_s$. Therefore, distributions realizable by $\calG'$ can also be factorized as in \eqref{eq_distr_fact}. Furthermore, the only difference between $\can(\calG)$ and $\can(\calG')$ is the addition of a  common marginalized parent between the vertices of $V_s$, a modification that can only possibly affect the term $P(X_{V_s}\mid X_S=0)$ of \eqref{eq_distr_fact}. However, as we saw, $\can(\calG)$ itself can already realize \emph{any} conditional distribution $P(X_{V_s}\mid X_S=0)$; therefore, the addition of this extra marginalized vertex cannot change the set of realizable distributions.

\end{proof}

\subsection{Proofs for \Cref{sec_remove_special_obs}}
\label{app_proof_special_edge_obs_removal}

In this subsection, we present the proof of \Cref{prop:special-edge-in-district-removal}.

\ObsRemoveSpecialEdges* 

\begin{proof}
Let $\calD=\can(\calG)$ be the canonical DAG associated with $\calG$. Then, if its vertices are partitioned as $(V,M,S)$, $\calD$ contains a path $a \to s \gets m \to b$ where $\pa{\calD}{m}=\ch{\calD}{s}=\emptyset$,  $\pa{\calD}{s}=\{a,m\}$ and $\ch{\calD}{m}=\{s,b\}$. Furthermore,  $a$ and $b$ share a common parent $m'$, and $a,b \in V$, $s \in S$, and $m,m' \in M$. Let $\calD'$ be the DAG obtained by removing the vertex $m$ from $\calD$. The SLP of $\calD'$ is $\calG'$.

Clearly $\calD$ can reproduce any distribution from $\calD'$
since $\calD$ is a subgraph of $\calD$, so what remains is to 
prove that the converse is also true.

Let $P(X_V|X_S=0)$ be an arbitrary observational distribution that is realizable by $\calD$, with associated kernels $P(X_y \mid X_{\pa{\calD}{y}}), y \in V\dotcup M\dotcup S$. 

Now, we will define a set of kernels $P'$ on the DAG $\calD'$ as follows:
\begin{itemize}
    \item The vertex $m'$ in $\calD'$ contains two variables: $X_{m'}=(\bar{X}_{m'},X_m)$. They are distributed as $X_{m'}$ and $X_m$ from $\calD$, i.e.,  $P'(X_{m'}=(\bar{x}_{m'},x_m))=P(X_{m'}=\bar{x}_{m'})P(X_m=x_m)$.
\item The variable $X_b$ responds to $\bar{X}_{m'}$ and $X_m$ as it used to respond to $X_{m'}$ and $X_m$ in $\calD$, i.e., $P'\left(X_{b}|X_{m'}=(\bar{x}_{m'},x_m),X_{\pa{\calD}{b}\setminus\{m',m\}}\right)=P\left(X_b|X_{m'}=\bar{x}_{m'},X_m=x_m,X_{\pa{\calD}{b}\setminus\{m',m\}}\right)$.
\item The variable $X_s$ responds to $X_a$ as $P'(X_s=0|X_a)=P(X_s=0|X_a)$.
    \item The variable $X_a$ is such that: 
\begin{align*}
        & P'\left(X_{a}|X_{m'}=(\bar{x}_{m'},x_m),X_{\pa{\calD}{a}\setminus\{m'\}}\right)  \\
        &=\begin{cases}        \frac{P(X_s=0|X_a,X_m=x_m)}{P(X_s=0|X_a)}P(X_a|X_{m'}=\bar{x}_{m'},X_{\pa{\calD}{a}\setminus\{m'\}}) & \text{if } P(X_s=0|X_a)\neq 0 \\ 0 &\text{if } P(X_s=0|X_a)=0 
    \end{cases}.
    \end{align*}
\item The remaining variables respond to their parents in the same way as they did in $\calD$.
\end{itemize}

With these kernels, the distribution $P'(X_V,X_S=0)$ realized by $\calD'$ is:

\begin{align}
    & P'(X_V, X_S=0) \nonumber\\
    &= \nonumber\int_{ x_{m'}} P'(X_a|X_{m'},X_{\pa{\calD}{a}\setminus\{m'\}})P'(X_{b}|X_{m'},X_{\pa{\calD}{b}\setminus\{m',m\}})P'(X_s=0|X_a)P(X_{m'})\cdot \\ 
    &\cdot\prod_{y\in V\cup M\setminus\{a,b,m',m\}}P'(X_y|X_{\pa{\calD}{y}})\prod_{w\in  S \setminus\{s\}}P'(X_w=0|X_{\pa{\calD}{w}})  dx_{m'} \\ 
    &=
    \nonumber\int_{ \bar{x}_{m'}, x_m}\hspace{-8mm} P(X_s=0|X_a,X_m)P(X_a|X_{m'},X_{\pa{\calD}{a}\setminus\{m'\}})P(X_{b}|X_{m'},X_m,X_{\pa{\calD}{b}\setminus\{m',m\}}) \\ 
    &\cdot P(X_{m'}) P(X_m) \prod_{y\in V\cup M\setminus\{a,b,m',m\}}\hspace{-8mm} P(X_y|X_{\pa{\calD}{y}})\prod_{w\in  S \setminus\{s\}}P(X_w=0|X_{\pa{\calD}{w}})  d\bar{x}_{m'} dx_m \\
    & = P(X_V, X_S=0).
\end{align}

Therefore, every observational distribution $P(X_V,X_S=0)$ realizable by $\calD$ can be reproduced by $\calD'$, and thus $\calD$ and $\calD'$  are observationally equivalent.
\end{proof}

\subsection{Proofs for \cref{subsec:remove-selected}} \label{app_selected_removal2}

In this subsection, we will prove \Cref{prop:smDG-split}, that says that when 
an smDG $\calG$ satisfies certain conditions, 
one can split a selected face into many singleton selected faces without any change to the SMO model:

\SelectedNodeRemovalObsEqTwo*

We will make this proof in steps, presenting several lemmas first.

\begin{lemma} \label{le:no-unshielded-colliders}
Let $\calD$ be a DAG with vertices partitioned as $(V, M, S)$, and let $s\in S$ be a selected vertex of $\calD$. If the set of vertices $\an{\calD}{s}$ (including $s$ itself) contains no unshielded colliders, then $\calD$  is observationally equivalent to the DAG $\calD_{-s}$ obtained by starting from $\calD$ and deleting $s$.

\end{lemma}

\begin{proof}

Let $P(X_V \mid X_s=0)$  be a distribution that is realizable by $\calD$, and let $v,v'\in V$ be two vertices such that $X_{v'}$ is conditionally \emph{dependent} of $X_v$ in $P$ given the parents of $v$ and selection on $s$. That is, $P(X_v \mid X_{\pa{\calD}{v}},X_{v'},X_s=0)\neq P(X_v \mid X_{\pa{\calD}{v}},X_s=0)$. Then, by the d-separation criterion, there exists a path $\pi:v \upathto v'$ that is active in $\calD$ given $\pa{\calD}{v}$ and $s$.

Suppose that the path $\pi$ contains colliders. All of the colliders $v_1\to z\gets v_2$ of $\pi$ need to be such that $z$ is an ancestor of either $\pa{\calD}{v}$ or $s$, otherwise $\pi$ would not be active given these vertices. The collider that is closest to $v$ cannot be an ancestor of $\pa{\calD}{v}$ (or it would induce a cycle), so it is an ancestor of $s$. This implies that $v\in \an{\calD}{s}$, which in turn implies that \emph{every} collider in $\pi$ must be in $\an{\calD}{s}$.

By assumption, there are no unshielded colliders in $\an{\calD}{s}$. Therefore, if  $v_1\to z\gets v_2$ is a collider in $\pi$, then $\calD$ includes either the arrow $v_1\to v_2$ or $v_1\gets v_2$. Consider now the path $\pi'$ that is the same as $\pi$ up until $v_1$, but then it goes directly to $v_2$ (without passing through $z$). This path is active in $\calD$ given $\pa{\calD}{v}\cup\{s\}$: its non-collider vertices are a subset of those in $\pi$, and if $\pi'$ has a collider in either $v_1$ or $v_2$, this is not a problem because $v_1$ and $v_2$ are both ancestors of $s$ (since they are parents of $z$).

Since $\pi$ has finite length, by repeating this process 
until we have no colliders, we will eventually 
obtain a path with no colliders, 
that is still active given $\pa{v} \cup \{s\}$.
Since the path is active given $\pa{v} \cup \{s\}$, 
and it has no colliders, it can only be of the form
$v \pathto v'$. 

That is, $v'$ is a descendant of $v$ in $\calD$. Since the vertex $s$ does not have children, $v'$ is also a descendant of $v$ in $\calD_{-s}$. So, if we define  the distribution $P'(X_V):=P(X_V \mid X_s=0)$, we have that $P'(X_V \mid X_{\pa{v}}, X_{v'}) \neq P(X_v \mid X_{\pa{v}})$ implies that  $v'\in \de{v}_{\calD_{-s}}$. This is equivalent to saying that $P'(X_V)$ is realizable by $\calD_{-s}$, which establishes the observational equivalence of $\calD$ and $\calD_{-s}$.

\end{proof}

\begin{lemma}\label{lemma_eq_tilde}
    Let $\calG$ be an smDG that satisfies the 
conditions 
of \Cref{prop:smDG-split}.

Let $\tilde{\calD}$ be constructed from $\can(\calG)$ by performing the following transformations:
\begin{enumerate}[label=\Alph*)]
\item Identify every path in $\can(\calG)$ of the form $a \to s' \gets m' \to b$ where $a,b \in V$,
$a$ and $b$ \emph{do not} share a marginalized parent, $s'\in S$ and $m' \in M$. For each such path, delete the edge $m' \to s'$ and add the edge $a \to b$. 
\item Prior to the transformation, $m'$ had precisely two children, $\{s', b\}$. Now $m'$ has only one child, namely, $b$. If $b$ has any other marginalized parents, remove the marginalized vertex $m'$ per $\mathtt{RmvRedM}$. \label{step:RmvRedM}
\item For every remaining path $a \to s \gets m \to b$ with $s\in S$ and $m \in M$, where $a$ and $b$ \emph{do} share a marginalized parent, delete the vertex $m$.
\end{enumerate}

Then $\tilde{\calD}$ is a DAG, and it is observationally equivalent to $\can(\calG)$.
\end{lemma}
\begin{proof}
     Condition a) of \Cref{prop:smDG-split} says 
     that $\calG$ contains no cycles, so
     step A) does not create cycles. 
     Due to \Cref{lemma_interchangeedges}, we know that step A) preserves Observe\&Do equivalence. Due to \Cref{prop:special-edge-in-district-removal}, we know that step B) preserves observational equivalence.
\end{proof}

\begin{restatable}{lemma}{lemmaremoveS}  \label{le:remove-s-noncanon}
Let $\calG$ be an smDG that satisfies the 
conditions of \Cref{prop:smDG-split}, and let $\tilde{\calD}$ be defined as in \Cref{lemma_eq_tilde}.  
Then, ${\calM(\tilde{\calD},V \mid S) = \calM(\tilde{\calD}_{-s},V \mid S\setminus\{s\})}$.
\end{restatable}

Before proving \Cref{le:remove-s-noncanon}, let us spell out 
what the conditions of \Cref{prop:smDG-split}
will imply about the structure of $\tilde{\calD}$.

\begin{lemma} \label{le:tildeG}
Let $\calG$ be an smDG that satisfies the conditions of \Cref{prop:smDG-split} and 
define $\tilde{\calD}$ as in \Cref{le:remove-s-noncanon}.

Let $M'$ be the marginalized vertices of $\an{\tilde{\calD}}{s}$ 
that have only visible children (i.e.~none of the vertices of $M'$ are part of a special edge). 

Then, $\tilde{\calD}$ will satisfy the following:
\begin{enumerate}
\item It is invariant under the operations 
{\tt SplitM->S}, $\mathtt{MergeS}$, $\mathtt{MergeM}$, $\mathtt{Exog}$, $\mathtt{Term}$ and $\mathtt{RmvRedM}$. \label{cond:remove-s-nearly-canon}
\item \label{cond:remove-s-abcd}
\begin{enumerate}[label=\alph*'),ref=\labelcref{cond:remove-s-abcd}\alph*')]
    \item $\tilde{\calD}$ contains no special edges.
    \label{cond:tildeG1}
    \item Visible vertices in $\an{s}$ that share a child or are in the same 
    selected face in $\calG_{\an{s}}$
    are neighbours 
    or share a marginalized parent in $\tilde{\calD}$. \label{step:merge}
 \label{cond:tildeG2}
    \item \label{cond:tildeG3}
    Letting $V'_m := \ch{\tilde{\calD}}{m}$, the sets $\{V'_m\}_{m \in M'}$ are disjoint.
    \item \label{cond:tildeG4}
    For every $m \in M'$,
    the parents of $V'_m$ that are not in $V'_m$ 
    are parents shared by every vertex in $V'_m$,
    i.e.~${\pa{\calG}{V'_m} \setminus V'_m \subseteq \bigcap_{v \in V'_m}\pa{\calG}{v}}$.
    \end{enumerate}
\end{enumerate}

\end{lemma}
\begin{proof}
Condition \labelcref{cond:remove-s-nearly-canon} holds because removing some 
edges $m'\to s'$ in the canonical graph 
will yield a graph that is still invariant under 
\texttt{SplitM->S}, $\mathtt{MergeS}$, $\mathtt{Exog}$, $\mathtt{Term}$. Then, step \labelcref{step:RmvRedM}
maintains these properties while also ensuring invariance under 
$\mathtt{RmvRedM}$.

Proof that $\tilde{\calD}$ satisfies Condition~\labelcref{cond:tildeG1}: Straightforward from the definition of $\tilde{\calD}$.

Proof that $\tilde{\calD}$ satisfies Condition~\labelcref{cond:tildeG2}: 
If any two visible vertices share a child or are in 
the same selected face in $\calG_{\an{V_s}}$
then by~\labelcref{cond:smDG2},
they will be neighbours in $\calG$
or they will share a marginalized parent (which will be an element of $M'$) in both
$\can(\calG)$ and $\tilde{\calD}$, thereby satisfying~\labelcref{cond:tildeG2}.

Proof that $\tilde{\calD}$ satisfies Condition~\labelcref{cond:tildeG3}: $\can(\calG)$ clearly satisfies Condition~\labelcref{cond:tildeG3}, 
and $\tilde{\calD}$ only adds a single marginalized parent to visible variables
that lack any marginalized parent.

Proof that $\tilde{\calD}$ satisfies Condition~\labelcref{cond:tildeG4}: By  \Cref{prop:smDG-split}'s Condition~\labelcref{cond:smDG4}, all of the parents of $V'_m$ that are not in $V'_m$ are parents of all of the vertices of $V'_m$ in the smDG $\calG$. The arrows between these parents and the vertices of $V'_m$ could represent either normal edges or special edges in $\can(\calG)$. When constructing $\tilde{D}$, however, all of these will turn into normal edges: because of Condition \labelcref{cond:smDG3} of \Cref{prop:smDG-split}, we know that the parents of $V'_m$ that are not in $V'_m$  do not share a { marginalized} parent with any vertex of $V'_m$ in $\can(\calG)$. Therefore, condition A) of \cref{le:remove-s-noncanon} will transform all of those edges that are special into normal edges of $\tilde{D}$.  With this, for every $m \in M$, the parents $\pa{\calG}{\ch{\can(\calG)}{m}}\setminus \ch{\can(\calG)}{m}$ are shared parents of every $v \in \ch{\can(\calG)}{m}$ by \Cref{prop:smDG-split}'s Condition~\labelcref{cond:smDG4}.
In $\tilde\calD$ visible variables have the same marginalized parents as in $\can(\calG)$,
except that some visible variables that lack a marginalized parent in $\can(\calG)$ acquire one,
so we also have that 
the parents $\pa{\calG}{\ch{\calD_{-s}}{m}}$ are shared parents of every $v \in \ch{\calD_{-s}}{m}$.

\end{proof}

We can now prove that there exists an ordering over the { marginalized} variables $M'$ 
with certain properties, that will ultimately allow us to 
factorize the model according to the districts of the graph~\citep{evans2018margins}.

\begin{lemma}
\label{lemma_topological_ordering}
     Let $\calG$ be an smDG that satisfies the conditions of \Cref{prop:smDG-split} and define $\tilde{\calD}$ as in \Cref{lemma_eq_tilde}. Define $M'$ and $V'_m$ as per \Cref{le:tildeG}. Then, $\tilde\calD$ is consistent with a topological ordering of the vertices of $\tilde\calD\cap \an{\tilde\calD}{s} \setminus \{s\}$ that takes the form:
\[
m_i \prec 
v^{i,1} \prec \ldots v^{i,l} 
\prec m_j \ldots
\]
where: 
\begin{itemize}
    \item $m_i,m_j,\ldots \in M'$ are the marginalized vertices with only visible children.
    \item $v^{i,1} \ldots, v^{i,l}$ are the visible children of $m_i$;
\end{itemize}
\end{lemma}
\begin{proof}
We will assume that there is no such ordering, and prove a contradiction. 

If there is no topological ordering, 
then note that the { marginalized} variables have no parents, 
so there must be a cycle in the visible variables.
Specifically, there must exist a cyclic sequence of distinct { marginalized} variables 
$m_1,\ldots,m_N=m_1$ in $M'$ 
such that for each $i$ there are vertices 
$v_i \in V'_{m_i}$ and $v'_{i-1} \in V'_{m_{i-1}}$ with 
$v'_{i-1} \in \an{\calG}{v_i}$. By Condition \labelcref{cond:tildeG4}, if a vertex that is not in $V'_{m_i}$ is a parent of some $v_i \in V'_{m_i}$, then it is a parent of  every vertex in $V'_{m_i}$. This would imply that $\tilde{\calD}$ contains a cycle $v'_{i-1} \pathto v'_i \pathto v'_{i+1} \pathto \ldots \pathto v'_{i-1}$, which 
contradicts the assumption that $\tilde{\calD}$ is a DAG.
\end{proof}

\begin{lemma}
\label{lemma_eq_shielded}
    Let $\calG$ be an smDG that satisfies the conditions  of \Cref{prop:smDG-split} and
define $\tilde{\calD}$ as in \Cref{lemma_eq_tilde}. Considering the topological ordering defined in \Cref{lemma_topological_ordering}, start from $\tilde\calD$ and:

\begin{itemize}
  \item[(ii)] (Fully connect each district) For every $v^{i,l} \prec v^{i,l'}$, add $v^{i,l} \to v^{i,l'}$.
  \item[(ii)] (Parents outside a district) For every $v^{i,l} \prec m^{i'}$, $i\neq i'$, where $\ch{v^{i,l}} \cap \ch{m^{i'}} \neq \emptyset$, add $v^{i,l} \to m^{i'}$.
\end{itemize}

Then, the resulting graph $\tilde{\calD}^\text{shd}$ is a DAG which is observationally equivalent to $\tilde\calD$. Furthermore, $\an{\tilde{\calD}^\text{shd}}{s}$ has no unshielded colliders.

\end{lemma}
\begin{proof}
    Since the transformations (i)--(ii) only add an arrow $a\to b$ in cases where $a\prec b$ in the topological ordering of \Cref{lemma_topological_ordering}, the resulting graph is a DAG.

    To show that operation (i) preserves observational equivalence, we will use district factorization~\citep{evans2018margins}.
    We consider a factorization of the distribution
    by districts (the 
    set of children of a { marginalized} variable $m^i$), 
    and wherever a visible or selected variable $z$ lacks 
    { marginalized} parents, we assign it to its own district $\{z\}$. Even though \citet{evans2018margins} only considers mDAGs, the idea can be extended to the case when there are selected faces by simply treating selected vertices as visible.  
    Consider a district $V'_{m^i}=\ch{\tilde\calD}{m^i}$ 
    belonging to a { marginalized} variable $m^i$. By \Cref{le:tildeG}, all of the visible parents of a vertex in this district are shared between all of the vertices of the district. Therefore, the district subgraph composed by $V'_{m^i}$ and their parents (where the vertices of $\pa{\tilde\calD}{V'_{m^i}}\setminus V'_{m^i}$ are fixed vertices, following Definition 2.9 of \citet{evans2018margins}) is \emph{saturated}, in the sense that it can realize \emph{any} kernel $P\big(X_{V'_{m^i}}\big|X_{\pa{\tilde\calD}{V'_{m^i}}\setminus V'_{m^i}}\big)$. Because of that, adding edges inside this district does not increase the set of observational distributions that are realizable by the causal structure, thus showing that operation (i) preserves observational equivalence. 
    
    Operation (ii) also preserves observational equivalence, as a consequence of \Cref{prop_exog_term}(a).

    To see that $\an{\tilde{\calD}^\text{shd}}{s}$ has no unshielded colliders in $\tilde{\calD}^\text{shd}$, note that:
    \begin{itemize}
        \item Due to \Cref{le:tildeG}(b'), two visible vertices of   $\an{\tilde{\calD}^\text{shd}}{s}$ that share a child are either already connected by an edge in $\tilde\calD$, or are children of the same $m^i$, in which case they become connected by an edge in  $\tilde{\calD}^\text{shd}$ due to operation (i). Therefore, all of the colliders of  $\an{\tilde{\calD}^\text{shd}}{s}$ that take the form $v\to\cdot\gets v'$ for visible $v, v'$ are shielded in $\tilde{\calD}^\text{shd}$.
        \item  Since there are no special edges in $\tilde{\calD}$, there are no colliders of the form $v\to s\gets m$ for visible $v$, selected $s$ and marginalized $m$. 
        \item  Since there are no special edges in $\tilde{\calD}$ and because of \Cref{le:tildeG}c'), there are no colliders of the form  $m\to v\gets m'$  for marginalized $m$ and $m'$.
        \item  Since there are no special edges, colliders of the form $v\to v'\gets m$ for $v,v'$ visible and $m$ marginalized can 
        only arise when $m$ only has visible children. Such colliders are either shielded by an edge $m\to v$, in case $v$ and $v'$ belong to the same district, or by an edge $v\to m$ that is added by operation (ii), in case $v$ and $v'$ belong to different districts.
    \end{itemize}

\end{proof}

Now, we can prove \Cref{le:remove-s-noncanon}:

\begin{proof}[Proof of \Cref{le:remove-s-noncanon}]
Define  $\tilde{\calD}$ as in \Cref{lemma_eq_tilde}, and  $\tilde{\calD}^\text{shd}$ as in \Cref{lemma_eq_shielded}. By \Cref{le:no-unshielded-colliders}, we know that the DAG
$\tilde{\calD}^\text{shd}_{-s}$
is
observationally equivalent to $\tilde{\calD}^\text{shd}$, which is in turn observationally equivalent to $\tilde{\calD}$.

We can then perform the inverse of
operations (i-ii) (\Cref{lemma_eq_shielded}) on  $\tilde{\calD}^\text{shd}_{-s}$ to obtain the observationally equivalent DAG $\tilde{\calD}_{-s}$, which is the subgraph of $\tilde{\calD}$
on variables other than $s$.
   Thus, ${\calM(\tilde{\calD},V \mid S) = \calM(\tilde{\calD}_{-s},V \mid S\setminus\{s\})}$.
\end{proof}

With  \Cref{le:remove-s-noncanon}, we can finally prove \Cref{prop:smDG-split}, which we restate below.

\SelectedNodeRemovalObsEqTwo*

\begin{proof}
    Clearly, $\calG$ observationally dominates $\calG'$. We proceed to prove the converse.

    Due to \Cref{lemma_eq_tilde}, we know that $\can(\calG)$ is observationally equivalent to $\tilde\calD$. Due to  \Cref{le:remove-s-noncanon}, we know that $\tilde\calD$ is observationally equivalent to $\tilde\calD_{-s}$. In $\tilde\calD_{-s}$, we identify every edge $a\to b$ that was created in step A) of \Cref{le:remove-s-noncanon}. All of these were originally special edges $a\to s'\gets m'\to b$ in $\can(\calG)$. Then, we construct a new graph $\calD'$ by starting from $\tilde\calD_{-s}$ and reverting this step, that is, by transforming these normal edges back into special edges. Due to \Cref{lemma_interchangeedges}, we know that $\tilde\calD_{-s}$ and $\calD'$ are Observe\&Do equivalent. Now, we identify every special edge that was broken in step B) of of \Cref{le:remove-s-noncanon}. Construct the DAG $\calD''$ by starting from $\calD'$ and reverting this step, recovering those special edges. Due to \Cref{prop:special-edge-in-district-removal}, $\calD''$ and $\calD'$ are observationally equivalent.
Finally, note that $\calD''=\can(\calG')$.  
    This completes the proof that $\calG$ and $\calG'$ are observationally equivalent. \blk
\end{proof}

\section{A graphical criterion for smDG constraints} \label{sec:sm-sep}
Given any smDG, it is possible to deduce the conditional independence constraints by examining the 
canonical DAG, identifying the deterministic variables (those that lack 
marginal parents) and then applying the
D-separation criterion \citep{geiger1990identifying}.
If, however, we prefer to deduce them directly from the smDG, 
we can devise an alternative criterion called sm-separation, 
which is analogous to m-separation in mDAGs \citep{lauritzen1990independence}.
Graphical criteria such as m-separation have been applied in other proofs 
that identify causal quantities, for example \citet{nabi2020full} and \citet{malinsky2019potential}.

D-separation is differs from d-separation in that it involves 
one additional step. 
Before checking for active paths, the conditioning set $Z$
is expanded to include the functionally determined variables
---those whose values are a deterministic
function of $Z$.

\begin{restatable}[Functionally determined variables]{definition}{FunctionallyDeterminedVariables}
To obtain
the \emph{functionally determined variables} $\lceil Z \rceil_\calG$ for visible variables $Z$ in a DAG (or smDG)
$\calG$, begin with assigning $Z$ to $\lceil Z \rceil_\calG$, then add to $\lceil Z \rceil_\calG$ 
every visible vertex $v\in V$
such that
$\pa{\calG}{v} \subseteq \lceil Z \rceil_\calG$ 
\(\big(\)and such that $v$ is not in $\bigcup_B B \in \calL$, in the 
case of an smDG $(V,\calE,\calL,\calS)$\(\big)\)
until convergence. 
\end{restatable}

Note that this procedure halts even in a cyclic graph, so long as the number 
of nodes is finite. D-separation is defined as follows.

\begin{definition}
We say that $X$ is D-separated from $Y$ conditioned on $Z$, denoted $X \perp_D Y \mid Z$, if $X \perp_d Y \mid \lceil Z \rceil_\calG$.
\end{definition}

\begin{lemma} \label{le_closure}
    Let $\calD$ be a canonical DAG 
    with SLP $\calG$.
    Then, $\lceil Z \rceil_{\calD}=\lceil Z \rceil_{\calG}$.
\end{lemma}
\begin{proof}
In $\calG$, the set $\lceil Z \rceil_{\calG}$ 
    is enlarged to include any $v'$ that 
    is not in any marginal face,
    and where $\pa{\calG}{v'} \subseteq \lceil Z \rceil_{\calG}$.
Any such $v'$ must  have no marginal parent in $\calD$, 
and so $\pa{\calG}{v'} = \pa{\calD}{v'}$, 
meaning that $v'$ is also added to $\lceil Z \rceil_\calD$.
Conversely, if $v'$ is in a marginal face, 
or has a parent in $\calG$ that is not in $\lceil Z \rceil_\calG$, 
then in $\calD$, $v'$ has a marginal parent, 
or has a visible parent that is not in $\lceil Z \rceil_\calD$, 
so it cannot be added to $\lceil Z \rceil_\calD$.
\end{proof}

Now, we define an smDG path.

\begin{definition}[smDG path]
In an smDG $\calG$, a path is a sequence $\langle v_1, t_1,v_2,\ldots t_k,v_k\rangle$ of distinct vertices and edges
where:
\begin{itemize}
    \item each $v_i$ is a vertex in the smDG,
    \item each edge type $t_i$ is either ``$\leftarrow$'' ``$\leftrightarrow$'', ``$\to$'', or ``$-\!\!-$'', 
    that corresponds to an edge between $v_i$ and $v_{i+1}$ that is present in the smDG. That is, the two vertices $v_i$ and $v_{i+1}$ are either connected through a directed edge, or are in the same marginalized face, or are in the same selected face. 
\end{itemize}
In a sequence $v_i,t_i,v_{i+1},t_{i+1},v_{i+1}$, the node $v_{i+1}$
is called a collider if $t_i$ and $t_{i+1}$ have arrowheads at $v_{i+1}$, 
and a non-collider otherwise.
\end{definition}

Now we present the graphical condition for independence in smDGs, termed sm-separation.

\begin{definition}[sm-separation]
In an smDG $\calG = \langle V, \mathcal{E}, \calL,\calS \rangle$, given a path $\langle v_1,t_1,\ldots,t_k,v_k \rangle$, 
we say that the path is active at $v_i$ given vertices $Z$ if:
\begin{itemize}
    \item if $v_i$ is a non-collider, then it is not in $\lceil Z \rceil_\calG$.
    \item if $v_i$ is a collider, then it has a descendant in 
    $Z$ or that is in some selected face $V_s \in \calS$.
\end{itemize}
A path is active if it is active at every vertex in its path.
If, for vertex sets $X,Y$, there exists a path from some 
$x \in X$ to $y \in Y$, active given $Z$, then we say 
$X$ and $Y$ are sm-connected, written as $X \not \perp_{sm} Y \mid Z$.
Otherwise, we say that $X$ and $Y$ are sm-separated given $Z$:
$X \perp_{sm} Y \mid Z.$
\end{definition}

For example, in the smDG from \Cref{fig:example-slp},
we could say that $b$ is sm-connected to $d$ given $c$, 
due to the active path $b \undir c \undir d$.

\begin{lemma} \label{le:selected-descendant}
    Let $\calG$ be an smDG. 
    A visible vertex $v$ is in $\an{\can(\calG)}{S}$
    if and only if $v \in \an{\calG}{A}$ for some $A \in \calS$.
\end{lemma}
\begin{proof}
From \Cref{def_smgraph}, any directed path in $\can(\calG)$ to $s \in S$ is in $\calG$ and terminates at some $A \in \calS$.
Conversely, any directed path in $\calG$ to $A \in \calS$
either is in $\can(\calG)$ and terminates at $S$ or 
at the first missing edge, we can truncate the path there, and 
the missing edge $v \to v'$ must have $v \in \pa{\calD}{s}$.
\end{proof}

We can now prove that sm-separation in an smDG is 
equivalent to D-separation in the canonical DAG.

\begin{theorem}
Let $\calG$ be an smDG with
disjoint sets of vertices $X,Y,Z$, and let $S$ be the set of selected vertices of $\can(\calG)$.
Then, $X \perp_{sm} Y \mid Z$ in $\calG$ 
if and only if $X \perp_D Y \mid Z,S$ in $\can(\calG)$.
\end{theorem}

The idea of the proof is that any active path in an smDG 
can be extended to obtain an active path in the canonical DAG.
Conversely, any active path in the canonical DAG may be contracted 
to obtain an active path in its latent projection.

\begin{proof}
\ding{212}\emph{Proof that $X \not \perp_D Y \mid Z,S$ in $\can(\calG)$ $\implies$ 
$X \not \perp_{sm} Y \mid Z$ in $\calG$:}

For any path $p$ in $\can(\calG)$,
perform the following replacements to obtain a path $q$ in the smDG $\calG$.
We will prove that $q$ is active given $\lceil Z \rceil_\calG$ 
iff $p$ is active given $\lceil Z \rceil_{\can(\calG)}\cup S $.

\begin{itemize}
    \item any $a \gets m \to b$ becomes $a \leftrightarrow b$;
    \item any $a \to s \gets b$ becomes $a \undir b$;
    \item any $a \to s \gets m \to b$ 
    where $a,b \in V,m \in M,s \in S$
    becomes an edge $a \to b$ 
    (which will satisfy
    $a \in \pa{\calD}{\bigcup_{V_s \in \calS}V_s}$
    and
    $b \in \ch{\calD}{\bigcup_{V_m \in \calS}V_m}$);
    \item any other $a \to b$ 
    becomes an edge $a \to b$
    (which will satisfy
    $a \not \in \pa{\calD}{\bigcup_{V_s \in \calS}V_s}$
    or
    $b \not \in \ch{\calD}{\bigcup_{V_m \in \calS}V_m}$).
\end{itemize}
Every vertex in $q \setminus p$ is active.
Every vertex that is in $p \cap q$ is a collider in $q$ iff it is 
a collider in $p$. 
Every non-collider in $\calG$ 
is active iff it is not in $\lceil Z \rceil_\calG$, so it is active in $\calG$
iff it is active in 
$\can(\calG)$ (\Cref{le_closure}).
Every collider in $\calG$ is active iff it has a descendant in $Z \cup \bigcup_{A \in \calS}A$,
so it is active in $q$ iff it is active in $p$.
Therefore,  $q$ is active given $\lceil Z \rceil_\calG$ 
iff $p$ is active given $\lceil Z \rceil_{\can(\calG)}\cup S $.

\ding{212}\emph{Proof that $X \not \perp_D Y \mid Z,S$ in $\can(\calG)$ $\impliedby$ 
$X \not \perp Y \mid Z$ in $\calG$.}
For any path $q$ in $\calG$ that is active given $\lceil Z \rceil_\calG$, 
invert the operation above 
to obtain a path $p$ in $\calD$
that is active given $\lceil Z \rceil_{\can(\calG)}\cup S $, proving the result.
\end{proof}

Since D-separation is known to be a sound and complete criterion for conditional independence in DAGs 
\citep{geiger1990identifying},
and the distribution of an smDG is equal to the selected-marginal distribution from the canonical DAG (\Cref{thm:diff-slp-implies-diff-model}),
this means that sm-separation is a sound and complete criterion for conditional independence 
for any liftable smDG.

\bibliography{refs}

\end{document}